\documentclass[a4paper,10pt]{amsart}
\usepackage{graphicx, mathabx, amssymb,amsfonts,amsmath,amsthm,newlfont}
\usepackage{epsfig,url}
\usepackage{fullpage}
\usepackage[usenames,dvipsnames]{color}
\usepackage{enumerate, enumitem}
\usepackage[colorlinks=true,linkcolor=red,citecolor=blue]{hyperref}
\usepackage{color}
\usepackage{young}

\usepackage[all,2cell]{xy} \UseAllTwocells \SilentMatrices

\newtheorem{thm}{Theorem}[section]
\newtheorem{cor}[thm]{Corollary}
\newtheorem{lem}[thm]{Lemma}
\newtheorem{prop}[thm]{Proposition}

\theoremstyle{definition}
\newtheorem{defn}[thm]{Definition}

\newtheorem{crit}[thm]{Criterion}

\newtheorem{ex}[thm]{Examples}
\newtheorem{Runex}[thm]{Running example}
\newtheorem{example}[thm]{Example}

\theoremstyle{remark}
\newtheorem{rem}[thm]{Remark}

\numberwithin{equation}{section}

\newcommand{\To}{\longrightarrow}

\newcommand{\Z}{\mathbb Z}
\newcommand{\Q}{\mathbb Q}

\newcommand{\C}{\mathbb C}

\newcommand{\R}{\mathbb R}
\newcommand{\N}{\mathbb N}

\newcommand{\dR}{\mathrm{dR}}

\newcommand{\m}{\mathfrak{m}}

\newcommand{\rec}{\mathrm{rec}}

\newcommand{\Or}{\mathcal{O}}
\newcommand{\ts}{\mathrm{Sup}}

\newcommand{\Gm}{\mathbb{G}_m}

\newcommand{\tr}{\mathrm{tr}}

\newcommand{\ff}{\mathfrak{f}}
\newcommand{\rank}{\mathrm{rank}}

\date{\today}

\begin{document}

\title{Mellin transforms, transfinite diameter and rational approximations of integrals}
\author{Francis Brown}
\address{University of Oxford, Radcliffe Observatory, Andrew Wiles Building, Woodstock Rd, Oxford, United Kingdom, OX2 6GG}
\email{francis.brown@all-souls.ox.ac.uk}
\date{\today}

\maketitle

\begin{abstract}
We establish a higher-dimensional irrationality criterion for periods which are presented  as Mellin integrals depending on many parameters. 
The criterion is stated as an upper bound on the  multi-variate  transfinite diameter of the image of the domain of integration under the Mellin arguments. 
Most of the paper is devoted to studying notions of transfinite diameter relative to very general multivariate Vandermonde matrices. 
 As a proof of principle, we illustrate how this approach works with detailed computations in the case of  a 5-parameter family of integrals for $\zeta(2)$
 on $\mathcal{M}_{0,5}$, the moduli space of curves of genus 0 with 5 marked points.
This yields a `higher-dimensional' proof of the irrationality of $\zeta(2)$, based on an upper bound for a certain  kind of transfinite diameter associated to $\mathcal{M}_{0,5}$.    \end{abstract}

\section{Introduction}
Following Ap\'ery and Beukers,  many known irrationality proofs of periods, such as  $\zeta(2)$ or $\zeta(3)$,  involve
constructing families of linear forms which provide rational approximations to the periods in question. 
These linear forms commonly arise from  integrals which depend on a  large number of  additional parameters, typically  greater than the dimension of the integral, which generate a huge family of additional  linear forms. However, the existing criteria for irrationality make no use of these  linear forms and their underlying  geometric structure.
 In this note we introduce a `multi-parameter' irrationality criterion which does exploit them. 
Our guiding principle  is that one should try to   replace ad hoc diophantine  arguments    with  general constructions in algebraic geometry. 

A motivating question for this investigation was the following: does there exist a method for proving the  irrationality of  period integrals  which \emph{improves} as their dimension  increases? 
All known constructions do the exact opposite, and become significantly worse with increasing dimension (or weight). Indeed  the quality of  known approximations to periods of higher weight, e.g., for the odd zeta values $\zeta(2n+1)$, rapidly become very poor indeed as $n$ increases.   As long as this situation persists,  proving irrationality results for periods  by this method is a game where the goalposts are forever moving out of reach.

\subsection{Main example} 
A fascinating geometric feature of period integrals  which  becomes  \emph{richer} with increasing  dimension is the  existence of additional parameters. 
To illustrate, consider the  integrals: 
\begin{equation} \label{intro:Ihijkl} 
 I(h,i,j,k,\ell) = \int_{[0,1]^2}  \frac{x^h (1-x)^i y^k (1-y)^j }{(1-xy)^{i+j-\ell}}  \frac{dxdy}{1-xy}\end{equation} 
which depend on five integer parameters  $h,i,j,k,\ell\geq 0$. They are canonically associated to $\mathcal{M}_{0,5}$ and one shows that $I(h,i,j,k,\ell)$ is a linear form in $\Z \zeta(2) + \Q$.  The \emph{denominator} of $I(h,i,j,k,l)$ is  the smallest positive integer $D$ such that $D I(h,i,j,k,\ell) \in \Z \zeta(2) + \Z$.  The Ap\'ery family  is obtained from the `diagonal'  
\[ I(n,n,n,n,n) = a_n \zeta(2) + b_n \ , \]
where $a_n, D_n b_n \in \Z$.
Since the integrals on the left-hand side tend  to zero exponentially fast, and  the denominators $D_n$ are  not too large, the approximation $\zeta(2) \sim -b_n/a_n$ is sufficient to prove that $\zeta(2) \notin \Q$.

Rhin and Viola's remarkable method \cite{RVzeta2, RVzeta3}, based on earlier work by Hata and others, selects a set of fixed integers    $a_1,\ldots, a_5>0$ to construct a better family    $I(a_1 n, a_2 n, \ldots, a_5 n)$. 
Their preferred choice is $I(12n, 12n, 14n, 14n, 13n)$, whose arguments form a line which lies very  close to the diagonal.  The key point is that changing the direction of this line away from the diagonal will typically compromise the asymptotic behavior of the integrals, but lead to a significant improvement of their denominators. The mechanism for this  is the so-called `group method', introduced in \emph{loc. cit.}, which exploits a rich group of symmetries  satisfied by  the full 5-parameter family of integrals to prove that specific primes  do not divide the denominators  of certain integrals $I(h,i,j,k,\ell)$.   This group structure seems to exist in quite some generality but does not yet have a satisfying explanation from the point of view of algebraic geometry.  Quite how to optimise the choice of the parameters $(a_1,\ldots, a_5)$ is something of an art which is not fully disclosed in the existing literature. 

Whilst  the group method, which involves selecting a single line of integrals, can lead to spectacular improvements in the irrationality measure, it still does not fully exploit the whole 5-dimensional space of parameters. Furthermore, there is a trade-off between the quality of the denominators of the integrals, and the quality of their asymptotics. Indeed, improving the denominators typically leads to worse asymptotics, and improving the asymptotics to poorer denominator bounds. 

These observations were the starting point for the present investigation, as part of a broader  study of the fascinating  properties of Mellin periods and their associated motives. 

\subsection{Mellin periods and the geometry of numbers}
 We consider an algebraic Mellin integral
\begin{equation}\label{introIs}   I(s_1,\ldots, s_r) = \int_{\sigma} f_1^{s_1} \ldots f_r^{s_r} \omega \qquad \mathrm{Re}(s_i)\geq 0
\end{equation}
on a smooth affine variety $X$ of dimension $d$, where  $\omega$ is a regular differential $d$ form on $X$, 
 $\sigma \subset X(\C)$ is a relative  $d$-chain with boundary  contained in  the points of a divisor  $Y\subset X$, and $f_1,\ldots, f_r: X \rightarrow \mathbb{A}^1$
are regular functions on $X$.
We assume that $X,Y, f_1,\ldots, f_r$ are all defined over $\Q$.

The finite-dimensionality of algebraic de Rham cohomology implies that there exists a finite set of periods $\xi_1,\ldots, \xi_m$ such that, for every set of integers $n_1,\ldots, n_r\geq 0$ 
\begin{equation} \label{introIns}  I(n_1,\ldots, n_r)  \in \Q \xi_1 + \ldots + \Q \xi_m\ .
\end{equation} 
By clearing the  denominator $D_{\underline{n}} \in \Z $, where $\underline{n} = (n_1,\ldots, n_r)$  we have 
\[  D_{\underline{n}}  \,   I(\underline{n}) \in  \Z \xi_1 + \ldots + \Z \xi_m\ .\]
Thus a  Mellin integral  \eqref{introIs}  is  a machine for constructing many linear forms in periods. When the $f_i$ vanish along the boundary of the domain of integration, these linear forms will typically be small. 
See \cite{DinnerParties} for a systematic study when $X$ is a (partial compactification of the) moduli space of curves of genus $0$ with  marked points, and $\xi_i$ are multiple zeta values.  In all our examples, $r$ is at least as large as the dimension $d$. For $\zeta(2)$, which will be our main running example, there are $r=5$ parameters, and $m=2$ periods, namely $\xi_1=\zeta(2)$ and $\xi_2=1$. 
\\

How to make use of the plethora of linear forms \eqref{introIns}? Consider all such linear forms where the $n_i$ are bounded above by a fixed integer $M$.  We wish to find integers $c_{\underline{n}}\in \Z$ so that the linear combination   
\[  \sum_{\underline{n}: n_i\leq M}  c_{\underline{n}}  D_{\underline{n}} I(\underline{n})  \quad \in \quad  \Z \xi_1 + \ldots + \Z \xi_m  \]
is as small as possible, and, importantly, non-zero (this cannot be guaranteed unless  the `motivic' versions \cite{NotesMot} of the $\xi_i$  are  linearly independent,  which we shall always assume).  The geometry of numbers already provides a method for tackling this problem.

\begin{thm} (Minkowski) Let $A \in M_{n\times n} (\R)$ be a matrix with $\det(A) \neq 0$. Then there exists an integer vector $c = (c_1,\ldots, c_n) \in \Z^n$ such that 
\[ 0 <  \left|  \sum_{j=1}^n A_{ij} c_j  \right| \leq |\det(A)|^{\frac{1}{n}} \quad \hbox{ for some } 1\leq i \leq n \ . \]
\end{thm}

This theorem suggests a very natural approach for generating small linear forms in the periods $\xi_1,\ldots, \xi_r$.  We  construct a matrix $Q^{\sigma}$ whose entries are integrals $I(\underline{n})$ for $n_i \leq M$, in such a way that its determinant  is non-zero. We will mostly  study a particular construction where   $Q^{\sigma}$ is positive-definite, but there are other possibilities which exploit the existence of Picard-Fuchs recurrence relations (or contiguity relations) between the $I(\underline{n})$.  In our situation, the determinant of $Q^{\sigma}$ can be bounded  above by using a greatly generalized version of the transfinite diameter, which  measures  the size of the image of the domain of integration $\sigma$ under the map $f=(f_1,\ldots, f_r)$ (see below). 

After clearing denominators, we obtain a matrix $A^{\sigma}_{ij}$ which satisfies 
\[  A^{\sigma} = A^{(1)} \xi_1 + \ldots +  A^{(m)} \xi_m\]
where $A^{(i)}$ are matrices with integer entries\footnote{Such  \emph{matrix} linear forms in periods merit investigation in their own right: here we consider them to be `small' when $\det(A^{\sigma})$ is small, but there are other ways to measure the size of a matrix  linear form in $\xi_1,\ldots, \xi_m$}. Minkowski's theorem may then be applied to the matrix $A^{\sigma}$ to deduce the existence of a non-vanishing linear form.  

The denominators occuring in the entries of the matrix $Q^{\sigma}$ are related to the orders of the poles of the integral $I$ along some suitable compactification of $X$. Whilst a general theorem  has not currently been worked out, we expect that  denominator bounds can systematically be obtained  using  overconvergent  $p$-adic de Rham cohomology. Our examples suggest, furthermore,  that denominator bounds can be significantly improved  by exploiting   congruences  in algebraic de Rham cohomology.

A key point is that the multi-parameter approach described above enables one to decouple the problem of optimising the denominators from the asymptotics of integrals.  In fact, examples suggest that a useful strategy is to fill the input matrix $Q^{\sigma}$ with  integrals with small denominators, irrespective of whether they individually give good approximations  or not. Minkowski's theorem will do the work of extracting a linear combination of these integrals which is small in absolute value. 

\begin{rem}
In the case when there is a single parameter $r=1$, our method is an effective version of the determinant criterion of Zudilin \cite{ZudilinDet}, who gave examples of periods (logarithms) to  show that significant improvements can already be obtained  in a one-dimensional setting. This follows from  the fact that the transfinite diameter of the unit interval $[0,1]$ is equal to $\frac{1}{4}$, which is  significantly smaller than $1/e\sim 0.3678\cdots$.  However, he remarks that there are no further gains to be made as the dimension increases, since the transfinite diameter of an $N$-dimensional cube $[0,1]^N$ is no smaller.  

One purpose of this note was to test if this conclusion  is indeed the case, since it stands to reason that the larger the space of available linear forms, the better the approximations one might hope to obtain. Indeed
  we find that his remark is unduly pessimistic, for  at least three reasons. Firstly,  our  criterion depends on  the transfinite diameter of the region $f(\sigma)$, which is typically much smaller than the unit cube. Secondly,  the notion of transfinite diameter which is involved is somewhat different  from the ones usually found in the literature. Finally, and most importantly, experiments suggest  that  the denominator estimates can  be dramatically improved  when one takes into account more parameters.
\end{rem} 

\subsection{Contents and results}
This paper has two main strands. The first is to set up a framework for constructing matrices $Q^{\sigma}$ associated to Mellin integrals and bounding their determinant in terms of a suitable generalisation of the notion of transfinite diameter.  We believe that this underlies  an interesting and very  general algebraic structure associated to a family of  algebraic  Mellin transforms \eqref{introIs}.

The second is to  describe some experimental computations exploring the integrals \eqref{intro:Ihijkl}  as a `proof of principle'.    The main achievement of this paper is that we are able to predict
quite accurately the asymptotics of the matrix determinants using our general theory.  Whilst there is (intentionally!) no possible new irrationality result to be gained from such an example, it enables us to probe the hypothesis that the method described here \emph{improves} as more parameters are brought into play, and suggests that more substantial improvements might be gained from  considering period integrals in higher dimension.

\subsubsection{Quadratic forms $Q^{\sigma}$ and Vandermonde determinants}
We now describe a simple construction of matrices $Q^{\sigma}$. Suppose we are given a finite sequence 
\begin{equation} \label{introff}   \ff_1 ,\ff_2 ,\ldots , \ff_N   \end{equation}
 of linearly independent polynomials  $\ff_i$ in    $f_1,\ldots, f_r$ on $X$.    Consider the symmetric matrix:
   \[  Q^{\sigma} =   \left(\int_{\sigma} \ff_i  \ff_j \,   \omega  \right)_{1\leq i,j\leq N} \quad \in \quad  M_{N\times N}( \Q\xi_1 + \ldots+  \Q \xi_m) \ .\]
    Its  entries are linear combinations of  the integrals \eqref{introIns}. 
The determinant of $Q^{\sigma}$ only depends on the free $\Z$-module $\mathcal{M}_N= \oplus_{i=1}^N \Z \ff_i$.  Consider
  \[  t_{\mathcal{M}_N}(\sigma) =  \sup_{z_1,\ldots, z_N \in \sigma} \left| \det ( \ff_i(z_j) )_{1\leq i,j\leq N} \right| \ . \]
The right-hand side  involves a  generalised Vandermonde determinant associated to the module $\mathcal{M}_N$, and one has the following upper bound
 \[  \left|\det (Q^{\sigma}) \right|  \leq    \left( t_{\mathcal{M}_N}(\sigma)\right)^2 \ ,  \]
 for all sufficiently large $N$. Furthermore, under some mild assumptions on the defining data  $f_1,\ldots, f_r, \omega, \sigma$ we may show that the matrix $Q^{\sigma}$ is positive-definite and hence has non-vanishing determinant.
 
 Let $D^{\ell},D^r \in M_{N\times N}(\Q)$ such that $A^{\sigma} = D^{\ell} Q^{\sigma} D^r$ has entries which are integer coefficients in the $\xi_i$, and denote the `denominator'  of this data by  $\delta_{\mathcal{M}_N} =  \det |D^{\ell}D^r|$.  
Then Minkowski's theorem, applied to $A^{\sigma}$,   guarantees the existence of an integer linear form in $\xi_i$   satisfying 
\[ 0 \quad < \quad \left|  n_1 \xi_1 + \ldots + n_m \xi_m\right| \quad  < \quad        \left(\left( t_{\mathcal{M}_N}(\sigma)\right)^2  \delta_{\mathcal{M}_N} \right)^{\frac{1}{N}}\ . \]
See theorem \ref{thm: smalllinearform}. This gives the following multi-parameter irrationality criterion:

\begin{crit} \label{intro:crit}   Suppose that $m=2$ and $\xi_1= 1, \xi_2=\xi$ (or $m>2$, $\xi_1=1$, and each $\xi_i\in \Q[\xi]$, for $i>1$).  Suppose that there exist sequences \eqref{introff} of arbitrary length $N>\!>0$  such that 
\[  t_{\mathcal{M}_N}(\sigma)^2\,   \delta_{\mathcal{M}_N} \rightarrow 0    \ ,  \]
where $\delta_{\mathcal{M}_N}$ is defined relative to some choice of matrices $D^{\ell},D^r$ as above. 
Then $\xi$ is irrational.
\end{crit}
This criterion is suited to performing numerical experiments.  The reader may wish to turn directly to \S \ref{sect:MainExample} for some detailed computations in the case of the
integrals \eqref{intro:Ihijkl}, where we consider the cases with $r=1$, $r=2$ and finally the full set of $r=5$ parameters.  Experiments suggest, as expected, that the quality of the approximations can improve as the number of parameters increases.

In order to make use of the previous criterion, one needs to have good control over the  quantity $t_{\mathcal{M}_N}(\sigma)$, which is a precursor to the  transfinite diameter and is discussed below.
Most of the work presented here is devoting to studying it, and our results are in close agreement with the numerical data presented in \S \ref{sect:MainExample}.
The denominator $\delta_{\mathcal{M}_N}$, in this case, is  related to a possible  notion of `$p$-adic transfinite diameter' (see \S\ref{subsect:padic}, Remark \ref{rem: zeta2padic}).

\subsubsection{Supremal transfinite diameter}The data of a family of Mellin integrals can be encoded by a morphism of varieties
\[ f:  X \To \mathbb{A}^r\]
where $f= (f_1,\ldots, f_r)$.  We let $V_f \subset \mathbb{A}^r$ denote the Zariski-closure of its image. It is an affine scheme which we call the anciliary image variety. 

In order to better understand the quantities $ t_{\mathcal{M}_N}(\sigma)$, we shall replace the data of the functions \eqref{introff}  with an infinite rank $\Z$-module $\mathcal{N} \subset \Or(\mathbb{A}^r) \cong \Z[x_1,\ldots, x_r]$ equipped with an increasing filtration $\mathcal{N}_n\subset \mathcal{N}$ of submodules of finite rank $N_n = \mathrm{rank} \, \mathcal{N}_n$. We assume that the $\Z$-module $\mathcal{M}=f^* \mathcal{N}$ is free, which amounts to the independence of \eqref{introff}.  
The Mellin transforms define a symmetric  bilinear pairing
\begin{eqnarray}  \mathcal{N}  \otimes_{\Z} \mathcal{N} & \To&  \R \nonumber \\ 
  \langle \m  ,   \m'  \rangle  & = &  \int_{\sigma}  f^*(\m) f^*(\m') \, \omega  \ . \nonumber 
  \end{eqnarray}
Under some mild assumptions it is non-degenerate, and even positive-definite. If the module $\mathcal{N}_n$ has basis $\mathcal{N}_n = \bigoplus_{i=1}^{N_n} \m_i \Z$ then the generalised Vandermonde matrix
\[ V_{\mathcal{N}_n}  =( \m_i(z_j)) \]
is well-defined up to permutations of rows and columns. The absolute value of its  determinant is the function  on $(\C^r)^{N_n}$ given by 
$| \det V_{\mathcal{N}_n}(z)  | =  | \det \mathcal{N}_n|$. 
  Given a region $\tau \subset \C^r$, we set $t_{\mathcal{N}_n} = \sup_{z\in \tau^{N_n}} | \det V_{\mathcal{N}_n}(z)  | $ . Finally we define the \emph{supremal transfinite diameter} to be
  \begin{equation} \label{intro:suptau}   \mathrm{Sup}_{\mathcal{N},e} (\tau) = \limsup_n   \left(t_{\mathcal{N}_n} (\tau)\right)^{1/e^{\mathcal{N}}_n}\end{equation}
  where $e^{\mathcal{N}}_n$ are appropriately chosen integers which we call exponents.  For special cases of modules $\mathcal{N}$, and compact $\tau$, the supremal transfinite diameter reduces to various notions of higher-dimensional transfinite diameter which have been studied in the literature.  Much previous work has been concerned with the existence   of the limits \cite{Zaharjuta}, but  here we only need an upper bound, which is why it is enough to take a limit supremum, which exists unconditionally. Unfortunately,  the higher-dimensional transfinite diameters are only known for some quite limited sets $\tau \in \C^r$,  even for $r=2$, but by leveraging known computations  we can deduce fairly  accurate bounds for the family \eqref{intro:Ihijkl}.  A related notion of transfinite diameter on algebraic varieties 
 was recently studied in   \cite{Mau, CoxMau}. 
  
  The quantity $\mathrm{Sup}_{\mathcal{N},e}(\tau)$ measures, in a certain sense, both the size and shape of the region $\tau$. A crude rule of thumb is that it is typically proportional to the volume of $\tau$ raised to some power which depends on the dimension of $\tau$ and on the exponents $e$.

 If we define a limit of denominators (having fixed a choice of  period representatives $\xi_1,\ldots, \xi_m$) by 
 \begin{equation}  
 \delta_{\mathcal{N},e} = \limsup_n  \delta_{\mathcal{N}_n}^{1/e^{\mathcal{N}}_n}
 \end{equation} 
 then the irrationality criterion \ref{intro:crit}, may be rephrased,  after taking the limit, as
 \begin{crit}  \label{intro: limitcrit}
$
 \mathrm{Sup}_{\mathcal{N},e}(f (\sigma))^2 \,  \delta_{\mathcal{N},e} < 1 
 $ \end{crit}
 In other words, the size of the image of the integration domain $f(\sigma)$  inside the anciliary image variety $V_f(\C)$, as measured by the module of functions $\mathcal{N}$, should not be too large relative to denominators.

\subsection{Intuitive example} \label{intro: exampleintuitive}  In practice, the criterion \eqref{intro: limitcrit} amounts to establishing an upper bound on the  transfinite diameter of $f(\sigma)$. 
In order to gain  some intuition for the precise relationship, consider a family of examples where $m=2$,  i.e., a family of Mellin integrals over $\Q$
\[  I(n_1,\ldots, n_r) = \int_{\sigma} f_1^{n_1} \ldots f_r^{n_r} \omega \  \in \  \Q + \Q \xi\]
where  $\omega$ is positive on $\sigma$,  $f(\sigma)\subset \C^r$ is Zariski-dense, and 
\begin{equation} \label{intuitivedenom}  d^w_{n_1,\ldots, n_r}  I(n_1,\ldots, n_r) \in \Z + \Z \xi
\end{equation} 
for some `weight' $w\geq 0$,   where $d_{n_1,\ldots, n_r} =  \max_{1\leq i \leq r}\{\mathrm{lcm} \{1,\ldots, n_i\}\}$.  Alternatively, one can assume $m>2$ and that every $\xi_i$ for $i>2$ is a polynomial in $\xi$. This kind of denominator bound is fairly typical,  but not universally applicable since  the precise denominators vary from situation to situation. 

In view of the  bound \eqref{intuitivedenom},  we consider the matrix $Q^{\sigma}$ associated to the `rectangular' module 
\[\mathcal{N}_n = \bigoplus_{0\leq n_i \leq n}\,  x_1^{n_1} \ldots x_r^{n_r} \Z \ . \] 
 In other words, we form a symmetric matrix whose entries consists of  integrals $I(n_1,\ldots, n_r)$ with arguments in the range $0 \leq n_i \leq 2n$.  The exponents are asymptotically  $e_n^{\mathcal{N}} \sim \frac{r}{2} n^{r+1}.$  The supremal transfinite diameter $\ts^{\rec}$ is computed using  the determinant of a `rectangular'  multivariable Vandermonde matrix.

The  denominator  computed in lemma \ref{lem: HardDenomBound} and the irrationality criterion   \ref{intro: limitcrit}    reduces via \eqref{deltaboundforintroexample} to 
\begin{equation}  \label{intro: generalSupbound} 
 \ts^{\rec} ( f(\sigma)) \,  \exp \left( \frac{w}{r}  \frac{2r+1}{r+1} \right) <1 \   \end{equation}
 which is implied by the sufficient condition $  \ts^{\rec} ( f(\sigma)) \exp( \frac{2w}{r})<1$. 
 In the case when $r=1$,   $ \ts^{\rec} $ equals the classical 1-dimensional transfinite diameter $\mathrm{tr}$ and 
  we retrieve the criterion of Zudilin:
 \[  \mathrm{tr}( f(\sigma)) <   \exp \left(  -  \frac{3}{2}w \right) \ .  \] 
As we increase the number of parameters $r$, the threshold for irrationality  in \eqref{intro: generalSupbound}  improves. For instance, when  
 $r=w$, we obtain the   criterion for irrationality 
\[   \ts^{\rec} ( f(\sigma))   < \exp\left(-2 + \frac{1}{w+1}\right)\ . \] 
A universal  sufficient condition is thus $4\,   \ts^{\rec} ( f(\sigma))  <4  \exp\left(-2\right) = 0.5413\ldots$
Typically $f(\sigma) \subset [0,1]^r$ is contained in the unit square, which satisfies $\ts^{\rec}([0,1]^r)= \frac{1}{4}$. So for the  criterion to apply when $r=w$,  $\ts^{\rec}(f(\sigma))$ needs to be  (very roughly) less than  half  that of the unit square.  

\begin{example}   Consider the $\zeta(2)$ integrals  \eqref{intro:Ihijkl}, for which $w=2$ and $m=2$.  In \S\ref{sect:MainExample} we study in some detail  the two-parameter subfamily $I(n_2,n_1,n_2,n_1,n_2)$ which amounts to setting $r=2$ and letting 
\[ f_1 = \frac{y(1-x)}{1-xy} \quad , \quad f_2 = x(1-y) \ .  \]
All conditions assumed above hold. The map  $f=(f_1,f_2)$ maps the region $[0,1]^2$ to a subset of the unit square which lies below the hyperbola $\{(x,y): xy=0.091\}$. Criterion \eqref{intro: generalSupbound}
requires that 
\[ 4  \,\ts^{\rec}( f(\sigma)) < 4  \exp(-5/3) = 0.7555\ldots  \] 
Using the techniques developed in this paper, we can prove that $\ts^{\rec}( f(\sigma))\sim  0.14$ to an accuracy of about $\pm 0.01$, which is amply suffficient for a `2-dimensional'  irrationality proof for $\zeta(2)$. In this example, however, we observe that the   denominators for $Q^{\sigma}$ are considerably smaller than the rather crude  bounds used above, and a much weaker inequality is sufficient in practice to establish irrationality (for instance, we can comfortably replace $0.7555\ldots$ with $0.81$ in the above). 
 \end{example}

\subsection{Plan and comments}
In \S\ref{sect: setup}, we review some background on algebraic Mellin integrals, introduce the anciliary image variety and illustrate with some simple examples. In \S\ref{sect: QuadraticformsMellin} we define positive-definite matrices $Q^{\sigma}$ and bound their determinants.  In \S\ref{sect: VdMModules}  we define a very general notion of Vandermonde determinant associated to a module, 
and in \S\ref{sect: Suptransfinite} we define the corresponding notion of supremal transfinite diameter.  The literature is unfortunately lacking effective techniques for estimating the transfinite diameter as far as we are aware, but in this section we introduce some methods which give some upper  bounds for tensor products and direct sums of modules, which are in turn applied in \S\ref{sect: hyperbolaregion} to a 2-dimensional region bounded by the unit cube and a hyperbola.   In \S\ref{sect: smalllinearforms} we derive in some more detail the irrationality criterion  outlined in the introduction and \S\ref{sect: Comments} discusses generalisations, $p$-adic variants and so on. The reader may wish to turn immediately to \S\ref{sect:MainExample} where the family of integrals for $\zeta(2)$ are studied in several different ways.  These examples are not intended to be anything more than a proof of principle; we did not attempt to optimise the choices for the modules $\mathcal{N}$. Indeed it might be interesting to consider modules  whose monomial bases  in $r$ variables   are centered around  on one or more of the Rhin-Viola `lines' mentioned earlier.

There are many other families of integrals for which it would be interesting to study the multi-parameter irrationality criterion introduced here. Likewise, it would be very interesting to generalise the irrationality criterion \ref{intro: limitcrit} to a linear independence criterion, by bounding the numerators in the linear forms obtained from Minkowski's theorem. This will be postponed to another day.

 \subsection{Acknowledgements}
  This project has received funding from the European Research Council (ERC) 
  under the European Union's Horizon Europe programme
  (grant agreement No. 101167287).  
   For the purpose of Open Access, the  author has applied a CC BY public copyright licence to any Author Accepted Manuscript (AAM) version arising from this submission. 
   Many thanks to Wadim Zudilin for  discussions on related topics and Shachar Weinbaum for comments.




 \section{Mellin transforms and anciliary image variety} \label{setup}
We briefly review some properties of Mellin transforms before introducing the anciliary image variety. 

\subsection{Algebraic Mellin integrals} \label{sect: setup} 
Let $X$ be an affine  algebraic variety over $\Q$ equipped with a morphism to affine space of dimension $r>0$:
\begin{equation}\label{geom: mapf}  f: X \To \mathbb{A}^r= \mathrm{Spec} \, \Q[ x_1,\ldots, x_r] \ . 
\end{equation} 
Denote the coordinates  of $f$ by  $f_i = x_i \circ f$, for $1 \leq i \leq r$, so that  $f= (f_1,\ldots, f_r)$.
Let $\omega \in \Omega^d(X)$ (the global sections of the sheaf of regular $d$-forms)  be a regular differential form of degree $d$ and 
$\sigma \subset  X(\C)$   a relative $d$-chain whose boundary is contained in the complex points of a  divisor $Y \subset X$, which we can assume to be normal crossing. The integral
\begin{equation} \label{Isdefn}  I(s_1,\ldots, s_r) = \int_{\sigma}  f_1^{s_1} \ldots f_r^{s_r} \,\omega \end{equation}
defines a function of $s_1,\ldots, s_r$ which is holomorphic for $\mathrm{Re}(s_i)\geq 0$  and extends meromorphically to $\C^{r}$.  In some examples, the form $\omega$ may additionally have poles along the boundary of $\sigma$. One can reduce to the previous situation by performing blow-ups to resolve singularities locally  along the boundary of $\sigma$.

\subsubsection{Periods} \label{sect: periods} 
For non-negative integer values of $s_i$, the integrals $I(n_1,\ldots, n_r)$ are periods of $H^d(X,Y)$. In particular, there exist $m$  regular differential forms $\omega_1,\ldots, \omega_m\in \Omega^d(X)$,   whose relative de Rham cohomology classes are independent in $H_{\dR}^d(X,Y)$ such that for all $\underline{n}= (n_1,\ldots, n_r)$  non-negative integers there exist unique rational numbers $a_{1}(\underline{n}), \ldots, a_{m}(\underline{n})$  with:
\begin{equation} \label{ai(n)defn}   {[}f_1^{n_1} \ldots f_r^{n_r} \omega ] = a_{1}(\underline{n})  [\omega_1] + \ldots + a_m(\underline{n}) [\omega_m]  \qquad \in\quad  H_{\mathrm{dR}}^d(X,Y) \ .
\end{equation} 
Note that $m$ may be strictly smaller than $\dim_{\Q} \, H^d_{\mathrm{dR}}(X,Y)$.  If we denote by 
\[ \xi_i = \int_{\sigma} \omega_i \quad  \hbox{ for } \quad 1\leq i\leq m\]
then  $\xi_i$ is a period  of $H^n(X,Y)$ since   $\partial \sigma \subset Y(\C)$. We deduce the 
 \begin{lem} \label{lem: periodsxi} The Mellin integral \eqref{Isdefn} at integer values: 
\[  I(n_1,\ldots, n_r) =  a_{1}(\underline{n})  \xi_1 + \ldots + a_m(\underline{n}) \xi_m \quad \in \quad  \Q[\xi_1,\ldots, \xi_m] \]
is a $\Q$-linear form in finitely many periods $\xi_1,\ldots, \xi_m$ of $H^n(X,Y)$.
\end{lem}

\subsubsection{Holonomicity} It is well-known that, as we vary $n_i$, the rational numbers \eqref{ai(n)defn} satisfy holonomic recurrence equations (in several variables) (e.g., \cite{LoeserSabbah}). On a more basic level, the cohomology classes $ {[}f_1^{n_1} \ldots f_r^{n_r} \omega ] $  satisfy linear relations with coefficients which are polynomials in the $n_i$, akin to integration-by-parts identities.
From a different perspective, for every $1\leq i \leq r$, there  is a shift operator $\tau_i$ which represents the action of multiplication by $f_i$ on cohomology classes of degree $d$ in the twisted algebraic de Rham cohomology of 
$X\,  \setminus\, V(f_1\ldots f_r) $ relative to $Y$, with coefficients in the algebraic vector bundle with connection $(\mathcal{O}_X, \nabla)$, where $\nabla = d + \sum_{i=1}^r s_i \frac{d f_i}{f_i}$. In a suitable basis, these may represented by contiguity matrices, from which the matrices $Q^{\sigma}$ may often be derived. We will  not exploit this structure in the present work, but provide some details in the case of the integrals \eqref{intro:Ihijkl} in an Appendix.

\subsection{Anciliary image variety}  Let $(X,f)$ be as above. An important role is played by what we shall call the  \emph{anciliary image variety} associated to the data of a Mellin transform.

\begin{defn} Let $V_f\subset \mathbb{A}^r$   denote the Zariski closure of $f(X)$. 
By abuse, we shall simply write $f: X \rightarrow V_f$ for the natural map through which \eqref{geom: mapf} factorises. 
\end{defn}

The affine coordinate ring of  $V_f$  is the quotient $\Q[f_1,\ldots, f_r]/I$,  where $I$ is the ideal generated by the algebraic relations satisfied by the functions $f_i$. By abuse of notation, we write  
\begin{equation} f(\sigma) \subset V_f(\C) \ .\end{equation}
Two key quantities are the \emph{number of parameters} $r$, and the \emph{anciliary  dimension}  $\dim V_f \leq \dim X$.

%

\subsection{Some examples}  We give two very basic examples to illustrate the role  played by  the number of parameters  and of the ramification locus of the map $f: X \rightarrow V_f$.

\subsubsection{Rational numbers} Let  $X= \mathbb{A}^1$, $Y=\{0,1\}$ and $\sigma = [0,1]$. 
First let $r=1$,  and  let $f: X \rightarrow \mathbb{A}^1$ denote the map $f=x(1-x)$, where $x$  is the coordinate on   $\mathbb{A}^1$. 
 Set $\omega = dx$.  Then 
the family of integrals 
\[  I(n) =  \int_{0}^1  \left(x(1-x)\right)^n dx\]
are periods of $H^1(\mathbb{A}^1, \{0,1\})\cong \Q$  and are all rational. 
The anciliary image variety is  $V_f \cong  \mathbb{A}^1= \mathrm{Spec}\,  \Q[f]$. 
 The image  $f(\sigma) \subset V_f(\C)$ is the interval $[0,\frac{1}{4}]$, as the function $f$ is maximised on the interval $[0,1]$ at the point $x=\frac{1}{2}.$
The boundary points $\{0,\frac{1}{4}\}$ of  $f(\sigma)$ can be   
 determined algebraically. To see why, note that 
the critical locus of $f$ is the  vanishing locus of $f' = 1-2x$. Thus the discriminant of $f$ is
$\mathrm{Disc}(f) = \{ \frac{1}{4}\} \subset V_f$, and 
$f: X \rightarrow V_f = \mathbb{A}^1$ is a double covering of affine spaces  ramified at $\frac{1}{4}$. 
The locus $f(Y)$ is the single point  $\{0\} \subset V_f$.

Now let $r=2$,   and  let $f: X \rightarrow \mathbb{A}^2$ denote the map $f=(x,1-x)$.
Set $\omega = dx$. Then 
the following family of integrals generalises the previous example:
\[  I(n_1,n_2) =  \int_{0}^1     x^{n_1} (1-x)^{n_2}  dx \ . \]
They are  periods of $H^1_{\mathrm{dR}}(\mathbb{A}^1, \{0,1\}) \cong \Q$ and are again all rational.
The anciliary image variety  $V_f  \subset \mathbb{A}^2$ is defined by the locus $x_1+x_2 =1$ and is isomorphic to $\mathbb{A}^1$. Then $f: X\rightarrow V_f$ is an isomorphism and the image of $f(\sigma)$ is simply the unit interval in $V_f$ relative to the coordinate $x_1$.

\subsubsection{A logarithm} Let   $X= \mathbb{A}^1 \backslash \{-1\}$, $Y=\{0,1\}$ and $\sigma = [0,1]$. 
First  let $r=1$ and define $f: X \rightarrow \mathbb{A}^1$ by 
$f=\frac{x(1-x)}{(1+x)}$. Let $\omega = \frac{dx}{1+x}$. The associated integrals  
\[  I(n) =  \int_{0}^1  \left(\frac{x(1-x)}{1+x}\right)^n \frac{dx}{1+x}\]
are periods of $H^1(\mathbb{A}^1\backslash \{-1\}, \{0,1\})$ which is a  Kummer extension of $\Q(-1)$ by $\Q(0)$. It has two periods $\xi_1=1$ and $\xi_2= \log 2$ and one may easily deduce that $I(n) \in \Q + \Z \,\log 2$ (the coefficient of $\xi_1$ is an integer since it is obtained as  a  residue of the cohomology class of the integrand). 
The function $f'$ vanishes at $x = -1 \pm \sqrt{2}$, and the corresponding critical values of $f$ are $3 \pm 2 \sqrt{2}$.  Thus the boundary of $f (\sigma)$ is contained in $f(Y) \cup \mathrm{Disc}(f) = \{0\} \cup \{ 3 \pm 2 \sqrt{2}\}$ and indeed 
\[ f(\sigma) =  [ 0, 3- 2 \sqrt{2}]\ .\]

Now let $r=2$,    and  let $f: X \rightarrow \mathbb{A}^2$ denote the map $f=(\frac{x}{1+x},\frac{1-x}{1+x})$.
Set $\omega = \frac{dx}{1+x}$ and $\sigma = [0,1]$. Then 
the family of integrals 
\[  I(n_1,n_2) =  \int_{0}^1   \frac{x^{n_1} (1-x)^{n_2} }{(x+1)^{n_1+n_2}} \frac{dx}{(x+1) }\]
are periods of the same cohomology group as before. 
Thus  $I(n_1,n_2) $ is a $\Q$-linear combination of $1$ and $\log 2$.
The anciliary image variety  $V_f  \subset \mathbb{A}^2$ is  isomorphic to the affine line $2x_1 + x_2 =1$, and $f(Y)= \{ (0,1),  (\frac{1}{2},0)\}$. The region $f(\sigma)$ is thus the interval $[0,\frac{1}{2}]$ in the coordinate $x_1$.

 \section{Quadratic forms associated to Mellin transforms } \label{sect: QuadraticformsMellin}

\subsection{A bilinear form} Let $(X,f)$ be as in \S\ref{setup} and let $V_f$ denote the anciliary image variety. 
 Consider the  bilinear symmetric pairing
\begin{eqnarray} \label{bilinear}
\Or(V_f) \otimes_{\Z} \Or(V_f) & \To & \R    \\
g \otimes h & \mapsto &  \int_{\sigma} gh \, \omega \ . \nonumber 
\end{eqnarray} 
We may restrict this pairing to any free $\Z$-module
\[ \mathcal{M}  \subset  \Or(V_f)\ . \] 
If $\mathcal{M}$ has finite rank with  generators  $\ff_1,\ldots, \ff_N$, then this pairing is represented by a symmetric matrix.

\begin{defn} With these notations, define the $N\times N$ symmetric matrix 
\[ (Q^{\sigma}_N)_{ij} =   \int_{\sigma}  \ff_i \ff_j \, \omega   \qquad \hbox{ for  }   1\leq i,j\leq N \ . \] 
Note that the definition uses the multiplicative (ring) structure on  $\Or(V_f)$, but that we \emph{do not}  need to  assume that   $\mathcal{M}$ is stable under multiplication. 
\end{defn} 
The determinant $\det(Q^{\sigma}_N)$ is well-defined, since 
changing the basis of $\mathcal{M}$ by   $P \in \mathrm{GL}_N(\Z)$ results in the matrix 
$P^T  Q^{\sigma}_N P$ which has the same determinant.

\subsection{Positive-definiteness}
We make the following  assumptions.
\begin{enumerate} [label=(P\arabic*)] 
\item  The form $\omega$ is real and non-negative on $\sigma$.
\item The generators  $\ff_1,\ldots, \ff_N$ of $\mathcal{M}$ are real-valued on $\sigma$.
\item The domain $\sigma$ is Zariski-dense in $X$.
\end{enumerate}
In some  circumstances these assumptions can  be relaxed (see \S \ref{sect: Comments}).

\begin{lem}  \label{lem: posdefinite} Assume $(P1)$-$(P3)$. Then  the bilinear pairing \eqref{bilinear} is positive semi-definite. 
Let $\ff_1,\ldots, \ff_N $ be any elements in $\mathcal{O}(V_f)$.  The following are equivalent:

(i) The elements $\ff_1,\ldots, \ff_N$  are linearly independent. 

(ii)  The matrix  $Q^{\sigma}_{N}$ is positive-definite.
\end{lem} 

\begin{proof}  Let $v=(v_1,\ldots, v_N)^T \in \R^N$ and 
define an element $P_v \in \Or(V_f)\otimes_{\Z}\R$  by  
$ P_v  = \sum^N_{i=1} v_i    \ff_i$.
By assumption  (P1) and (P2),  the  integral 
\[  I (v)= \int_{\sigma}   P_v^2 \,   \omega  \geq 0 \]
 is  non-negative and vanishes if and only if $P_v=0$. Since
 \[   P_v^2= \left( \sum_{i=1}^N  v_i \ff_i  \right)^2 =\sum_{i,j=1}^N  v_i v_j \ff_i \ff_j  \ , \] 
  the integral $I(v)$  
  may be written as the  quadratic form
\[ I(v) = v^T  Q^{\sigma}_{N} \, v   \  .   \]
By  (P3), the restriction $\Or(X)\otimes_{\Q} \R \rightarrow \mathcal{C}^{\infty}(\sigma)$ is injective; the morphism 
$f^* : \Or(V_f) \otimes_{\Q}\R \rightarrow \Or(X)\otimes_{\Q} \R$ is injective by definition. 
Suppose that  $\ff_1,\ldots, 
\ff_{N}$ are linearly independent. Then $P_v$ vanishes on $\sigma$ if and  only if $v_i=0$ for all $i=1,\ldots, N$. Equivalently, $I(v)=0$  if and only if all $v_i=0$ vanish, and hence  $Q^{\sigma}_N$ is positive definite. The converse is trivial, since 
if  $\ff_1,\ldots, \ff_N$ are not independent,  there exists a non-trivial vector $v$ such that $P_v=0$, and hence $I(v)= v^T Q^{\sigma}_N v=0$. 
 \end{proof} 

In particular, when $\ff_1,\ldots, \ff_N$ are a basis for a free $\Z$-module $\mathcal{M}$, the determinant $ \det( Q^{\sigma}_{N} ) $ only depends on $\mathcal{M}$ and under assumptions $(P1)$-$(P3)$ we have 
 \[ \det( Q^{\sigma}_{N} ) >0 \ .\]

\subsection{Vandermonde determinant associated to $\mathcal{M}$}
Let  $X^{N}$ denote the $N$-fold product $X \times \ldots \times X$ and let $z_i: X^N \rightarrow X$ denote the projection onto the $i^{\mathrm{th}}$ component.
Let $\sigma^{N} \subset X^{N}(\C)$ denote the $N$-fold product of  the chains $\sigma$.

\begin{defn}
Let $\mathcal{M}   \cong \Z \ff_1 \oplus \ldots \oplus \Z \ff_N $ be a free $\Z$-submodule of $\mathcal{O}(V_f)$ of rank $N$.
For any $ z= (z_1,\ldots, z_{N}) \in X^N(\C)$ we consider an  $N\times N$ matrix 
\[ V_{\mathcal{M}} (z) =     (  \ff_j(z_i)) \quad 1\leq i,j\leq N\ ,  \]
whose rows are indexed by the components of $z$ and whose columns are indexed by $\ff_1,\ldots , \ff_N$, in some chosen order. Its determinant is well-defined up to a sign and hence its  square is well-defined (it is independent of the choice of generators of $\mathcal{M}$).
\end{defn} 

\begin{rem} \label{rem:VdMonlydependsonM} The natural map $\mathcal{O}(V_f)  \hookrightarrow \mathcal{O}(X) \rightarrow  C^{\infty} (\sigma)$,  which restricts functions to $\sigma$, induces a map  
$
\textstyle{\bigwedge^{\!\!N}} \mathcal{O}(V_f)   \To    C^{\infty} (\sigma^N) $. 
The function 
$ \left(\det V_{\mathcal{M}}(z)\right)^2 \   \in \      C^{\infty} (\sigma^N) $
is simply the square of the image of  a generator of 
$\det(\mathcal{M})\cong \Z$.
\end{rem}

The matrix $V_{\mathcal{M}}(z)$ is a generalised Vandermonde matrix. 

\subsection{Calculation of the determinant of $Q^{\sigma}_N$}

The $N$-fold  external tensor product of $\omega$   on $X^{N}$ will be denoted by
\[ \omega^{\boxtimes  N  } = \underbrace{\omega \boxtimes \omega \boxtimes \ldots \boxtimes \omega}_{N}  \  \in  \ \Omega^{dN}(X^N) \ . \]
Many thanks to a referee for pointing out that the following result goes back to the work of Heine and Sz\"ego in the theory of orthogonal polynomials \cite{Andreief1883}. 
\begin{prop} \label{propdetQub}  The determinant of $Q^{\sigma}_N$ is given by the integral:
\begin{equation}  \label{secondformulaDetQ} 
   \det Q^{\sigma}_{N} =  \frac{1}{N!} \int_{\sigma^{N}}  \left(  \det  V_{\mathcal{M}}(z) \right)^2\,     \omega^{\boxtimes  N}  \ . 
  \end{equation}
\end{prop} 
\begin{proof} 
Let $\Sigma_N$ be the symmetric group on $N$ elements and let  $\varepsilon: \Sigma_N \rightarrow \{1,-1\}$ be the sign representation.  If $h\in \Or(X)$    we write $h(z_i)$ for $z_i^*(h) \in \Or(X^N)$.  Recall that $\mathcal{O}(V_f)$ embeds in $\Or(X)$. Then:
\begin{align*}
\det(Q^{\sigma}_N)  = &  \det \left(  \int_{\sigma}  \ff_{i} \ff_{j}  \,  \omega  \right)_{ij} \\
 = &   \sum_{\tau \in \Sigma_N} \varepsilon_{\tau} \int_{\sigma} \ff_{1} \ff_{\tau(1)} \omega   \, \times   \int_{\sigma} \ff_{2} \ff_{\tau(2)}  \,   \omega \times  \ldots  \times \int_{\sigma} \ff_{N} \ff_{\tau(N)}   \,\omega\\
 = & \sum_{\tau \in \Sigma_N} \varepsilon_{\tau} \int_{\sigma^N} \ff_1(z_1)  \ldots \ff_N(z_N)   \ff_{\tau(1)}(z_1)   \ldots \ff_{\tau(N)}(z_N)   \,  \omega^{\boxtimes N}\\
 = &  \int_{\sigma^N} \ff_1(z_1)  \ldots \ff_N(z_N)   \det\left(  \ff_j(z_i)     ) \right)   \,  \omega^{\boxtimes N}  \\  
= &     \int_{\sigma^{N}}     \left(  \prod_{i=1}^N   \ff_i(z_i )  \right) \det  V_{\mathcal{M}} (z)      \,     \omega^{\boxtimes N } \ .
\end{align*} 
Now consider the action of $\Sigma_N$ on $X^{N}$ by permuting factors. By functoriality of integration, we deduce that  for any $\pi \in \Sigma_N$, 
\[   \det(Q^{\sigma}_N)  =   \int_{\pi^{-1}( \sigma^{N})}   \left( \prod_{i=1}^N   \ff_i(z_{\pi(i)} )  \right)  \pi^*\left( \det  (V_{\mathcal{M}} (z) )\right)     \,    \pi^*( \omega^{\boxtimes N })\ . \]
The action of $\pi$ on the determinant is  $ \pi^*\left( \det  (V_{\mathcal{M}_N} (z) )\right)= \varepsilon(\pi)   \det  (V_{\mathcal{M}_N} (z) )$.  Furthermore, the action of $\pi$ 
  results in a 
 simultaneous  change in orientation of both the integration cycle $\sigma^{N}$ and the form $\omega^{\boxtimes N}$ which cancels out, i.e., there exists $\eta \in \{1,-1\}$ such that
 $\pi^* (\omega^{\boxtimes N} )= \eta \,\omega^{\boxtimes N}$ and $\pi^{-1} (\sigma^N) = \eta \, \sigma^N$. Therefore, for any such $\pi \in \Sigma_N$ one has the formula 
 \[    \det(Q^{\sigma}_N)  =   \int_{\sigma^{N}}      \left(\varepsilon(\pi)  \prod_{i=1}^N   \ff_i(z_{\pi(i)} )  \right)   \det  (V_{\mathcal{M}} (z) )    \,     \omega^{\boxtimes N } \ , \]
and  by summing over all elements $\pi \in \Sigma_N$ we conclude that 
\begin{align*} 
N! \det(Q^{\sigma}_N)  
 = \int_{\sigma^{N}}  \sum_{\pi \in \Sigma_N}     \left(\varepsilon(\pi)  \prod_{i=1}^N   \ff_i(z_{\pi(i)} )  \right)    \det  (V_{\mathcal{M}} (z) )      \,     \omega^{\boxtimes N }  =  \int_{\sigma^{N}} \left( \det  (V_{\mathcal{M}} (z) )\right)^2\,     \omega^{\boxtimes  N}  \ . 
\end{align*}

\end{proof}
\subsection{An upper bound for the determinant}
Finally, the determinant of the matrix $Q^{\sigma}_N$ may be bounded above in the limit as $N\rightarrow \infty$.

\begin{defn} Let $\mathcal{M} \subset \mathcal{O}(V_f)$ be a free $\Z$-module of rank $N$. Define 
\begin{equation}\label{tMDef}  t_{\mathcal{M}} (\sigma) = \sup_{z \in \sigma^N}  \left| \det V_{\mathcal{M}}(z) \right| \ . 
\end{equation}
It is well-defined since  changing  basis of $\mathcal{M}$ replaces  $V_{\mathcal{M}}(z) $  by 
$P \, V_{\mathcal{M}}(z)  \, P^T$  for some $P \in \mathrm{GL}_N(\Z)$, which has determinant in  $\{1,-1\}$. 
\end{defn} 

Under the assumptions $(P1)$, $(P2)$,  the right-hand side of \eqref{tMDef} is    the non-negative square root of the  Vandermonde determinant  defined in remark 
\ref{rem:VdMonlydependsonM}.

\begin{cor} \label{cor: detQNBoundTr}  
If  $N= \mathrm{rank}\, \mathcal{M}$, is larger than a constant which  depends only on $\int_{\sigma} \omega$, then
\begin{equation} \label{detQsigmaUpperBound}    |\det Q^{\sigma}_N|^{\frac{1}{N}}     < t_{\mathcal{M}} (\sigma)^{\frac{2}{N}}  \ .  
\end{equation} 
If, furthermore, the assumptions $(P1)$--$(P3)$ hold then $  0< (\det Q^{\sigma}_N)^{\frac{1}{N}}$. 
\end{cor} 
\begin{proof} By equation \eqref{secondformulaDetQ} we have
\[    \left| \det Q^{\sigma}_N \right|    \leq   \frac{1}{N!} \left| \int_{\sigma^{N}}  \, \omega^{\boxtimes N} \right|   \sup_{z \in   \sigma^{N}}  \left| \det     V_{\mathcal{M}}( z)   \right|^2  \]
and hence 
\[   | \det Q^{\sigma}_N | \leq 
   \frac{1}{N!}  \left| \int_{\sigma} \omega \right|^{N}  \, \sup_{z \in \sigma^{N}}  \left| \det     V_{\mathcal{M}}( z)   \right|^2 \ . \]
     By Stirling's formula, $(N!)^{\frac{1}{N}} $   exceeds $ |\int_{\sigma} \omega | $ for all sufficiently large $N$, hence 
     \[    | \det Q^{\sigma}_N |^{\frac{1}{N}} <  \sup_{z \in  \sigma^{N}}  \left| \det     V_{\mathcal{M}}( z)   \right|^{\frac{2}{N}}  \ .  \]
    Finally, if $(P1)$--$(P3)$ hold,   $Q^{\sigma}_N$ is positive definite by  lemma \ref{lem: posdefinite}.
\end{proof}

\subsection{Linear algebra interpretation}  The following remark can be skipped, but may illuminate proposition \ref{propdetQub}. 
The identity \eqref{secondformulaDetQ} is a manifestation of a  completely general algebraic identity. 
 Let $V$ be a free $\Z$-module of rank  $N$ and embed 
\[\textstyle{\bigwedge^{\!\!N} (\mathrm{Sym}^{2} V) } \quad \subset \quad V^{\otimes 2N}\ . \]
Let $e_1,\ldots, e_N$ denote a $\Z$-basis of $V$.  The determinant $\det(V) = \bigwedge^{\!\!N} V$ is the free  $\Z$-module of rank $1$  generated by 
$e_1\wedge \ldots \wedge e_N$. We presume that the following lemma must be known.

\begin{lem} There is a  natural inclusion 
$ \rho: \mathrm{Sym}^2  \left(\det(V)\right)   \hookrightarrow    \textstyle{\bigwedge^{\!N}   (\mathrm{Sym}^2 V) } $
 which sends
\[ (e_1 \wedge \ldots \wedge e_N)^2  \mapsto N! \sum_{\alpha \in \Sigma_N} \epsilon(\sigma) (e_1. e_{\alpha(1)} ) \otimes \ldots \otimes   (e_N.e_{\alpha(N)}) \]
where the symmetric group $\Sigma_N$ acts on $(\mathrm{Sym}^2)^{\otimes N}$ by permuting factors.  
\end{lem}
Note that the module  $\mathrm{Sym}^2  \left(\det(V)\right)$ is the Schur functor $\mathbb{S}_{\lambda} V$ where $\lambda$ is the rectangular Young diagram with $2$ columns and $N$ rows.

\begin{proof} Via the inclusion $\det(V) \hookrightarrow V^{\otimes N}$ we may identify
\[  e_1 \wedge \ldots \wedge e_N  =   \sum_{\sigma \in \Sigma_N}  \epsilon(\sigma) \, e_{\sigma(1)} \otimes \ldots \otimes e_{\sigma(N)}  \ . \]
Thus the map $\rho$ satisfies
\begin{eqnarray} 
\rho \left( (e_1 \wedge \ldots \wedge e_N)^2 \right)  & = &   \sum_{\sigma,\tau  \in \Sigma_N}  \epsilon(\sigma ) \epsilon(\tau) \, (e_{\sigma(1)}.e_{\tau(1)} ) \otimes \ldots \otimes (e_{\sigma(N)} . e_{\tau(N)})  \nonumber\\ 
& = &   \sum_{\sigma \in \Sigma_N} \sum_{\alpha  \in \Sigma_N}  \epsilon(\alpha) \, (e_{\sigma(1)}.e_{\alpha \sigma(1)} ) \otimes \ldots \otimes (e_{\sigma(N)} . e_{\alpha \sigma(N)})\nonumber\\
&  =&  N!    \sum_{\alpha  \in \Sigma_N}  \epsilon(\alpha) \, (e_{1}.e_{\alpha (1)} ) \otimes \ldots \otimes (e_{N} . e_{\alpha(N)}) \nonumber
\end{eqnarray} 
where in the first line we view  $V^{\otimes 2 N} = V^{\otimes N}\otimes V^{\otimes N}$ with the natural action of $\Sigma_N \times \Sigma_N$ permuting factors in each component, the notation $a . b$ denotes the symmetrised tensor product  $a \otimes b + b \otimes a$, and in the  third line we write $\alpha =\tau \sigma^{-1}$  and use the fact that $\epsilon(\alpha) = \epsilon(\sigma)^{-1}  \epsilon(\tau) = \epsilon(\sigma)  \epsilon(\tau) $. \end{proof} 
If we have a symmetric product $\langle , \rangle : \mathrm{Sym}^2 V \rightarrow R$ to some commutative ring $R$, then the $N$-fold product defines a map
$(\mathrm{Sym}^2(V))^{\otimes N} \rightarrow R$. It may be applied to both sides to give the well-known formula for a Gram determinant:
\[  \langle e_1\wedge \ldots \wedge e_N,  e_1\wedge \ldots \wedge e_N\rangle=   N! \det \left( \langle e_i, e_j\rangle \right)_{1\leq i,j \leq N}  \ .\]


  \section{Vandermonde determinants associated to modules} \label{sect: VdMModules}

 Having  already considered    Vandermonde determinants on $X^N(\C)$,  we now need to  consider generalised Vandermonde determinants on the anciliary image variety $V^N_f(\C)$.    
  \subsection{Vandermonde matrices}  \label{subsect: GenVdM} Let $V= \mathrm{Spec} \, R$ be an affine algebraic variety of finite type  over $\Q$.

  \begin{defn}  \label{defn: VdMdet} Let $\mathcal{N}$ be a free $\Z$-submodule of $R$ of rank $N$.  Choose a $\Z$-basis $\m_1,\ldots, \m_N$.  Let $z=(z_1,\ldots z_N) \in  V^N(\C)$ and define the matrix
    \begin{equation} \label{eqn: VNdef} V_{\mathcal{N}}(z) =  \left( \m_j(z_i) \right)  \ . \end{equation}
  It is a generalised Vandermonde matrix, and  its determinant only depends on the  choice of generators $\m_i$ by  a sign. In particular, the function
  \[ \left| \det V_{\mathcal{N}}(z) \right|  \quad  \hbox{ on  } \quad  V^N(\C) \]
   is well-defined and symmetric: it equals $ \left|  \det \mathcal{N} \right|$, where $\det \mathcal{N} = \bigwedge^{\!\!N} \mathcal{N}$. 
   For any set $\tau \subset V(\C)$, denote its $N$-fold product  by $\tau^N \subset V^N(\C)$,  and  define
  \begin{equation} \label{eqn: tNdef}  t_{\mathcal{N}}( \tau ) = \sup_{z\in \tau^N}   \left|  \det   V_{\mathcal{N}}(z)      \right|  \  . \end{equation}
  It is a non-negative real number, or infinite. 
  \end{defn} 
  The number $ t_{\mathcal{N}}( \tau )\geq 0$  is well-defined and is related to notions of transfinite diameter on an algebraic variety.
  Clearly,  if $\tau\subset \tau'$ then $ t_{\mathcal{N}}( \tau )\leq  t_{\mathcal{N}}( \tau' )$. 
  The definition makes sense when $\mathcal{N}$ is not free, but in this case the determinant function  is identically zero, so we will usually assume that $\mathcal{N}$ is free.

  If $\mathcal{M} \subset \mathcal{N}$ is a free submodule of  rank equal to that of $\mathcal{N}$ then 
  \begin{equation}  \label{eqn: tscalestorsion}  t_{\mathcal{M}}( \tau ) =  t_{\mathcal{N}}( \tau ) \,   C  \end{equation} 
  where $C$ is the cardinality of the finite torsion module  $\mathcal{N}/\mathcal{M}$.

 \subsubsection{Morphisms}  \label{subsect: morphisms} Consider a morphism
 $f: V \rightarrow V'$
 between two affine varieties of finite type over $\Q$ where $V=\mathrm{Spec}\, R$, $V'=\mathrm{Spec}\, R'$.
 
   \begin{defn}
 A  free $\Z$-module $\mathcal{N} \subset R'$ 
    will be called $f$-\emph{admissible} if       \[ f^*: \mathcal{N} \To f^{*} \mathcal{N} \]
    is injective, where $f^* \mathcal{N}$, for want of a better notation,  denotes the $\Z$-module which is the image of $\mathcal{N}$ under the morphism $f^*: R' \rightarrow R$ of $\Z$-modules. 
    \end{defn}
    For a free  $\Z$-module $\mathcal{N}$ of rank $N$, we denote a choice of  basis by $\m_1,\ldots, \m_N$. 
    Thus $\mathcal{N}$ is $f$-admissible if and only if   $\ff_i =  \m_i \circ f$, for $1\leq i \leq N$, are linearly independent functions on  $V(\C)$. 
    When this holds,
  \begin{equation} \label{eqn: tmorphism} 
  t_{f^*\mathcal{N}} (\tau) =  t_{\mathcal{N}}( f \tau )\ .
  \end{equation} 
  
  
   \subsubsection{Main example: the anciliary image variety}
    Recall  that the anciliary image variety is 
  \[ V_f = \mathrm{Spec}\left( \Q[x_1,\ldots, x_r] / I \right) \]
and the natural inclusion $V_f \subset \mathbb{A}^r$ is denoted by $f^*$. We will consider  $f$-admissible free $\Z$-modules $\mathcal{N} \subset \Z[x_1,\ldots, x_r]$. Equation  
 \eqref{eqn: tmorphism}  trades the usually more complicated functions $\ff_i$ on the set $\sigma\subset X(\C)$, which in our applications is simply a hypercube, against simpler functions $\m_i$, which are typically monomials in the coordinates $x_1,\ldots, x_r$,  but on the region $f \sigma \subset V_f(\C)$.

  \subsection{Classical examples of Vandermonde determinants}  \label{sect: ClassicalVdM} Let  $V \cong \mathbb{A}^r$ and  $R=   \Z[x_1,\ldots, x_r]$. 
  
  \subsubsection{One-dimensional Vandermonde matrix} \label{sect: onedimVdM}  Let $r=1$, and let $\mathcal{N}_{n} \subset \Z[x_1]$ be the $\Z$-module consisting of polynomials in $x=x_1$ of degree $<n$. In the basis
  $\m_1=1, \m_2=x, \ldots, \m_n= x^{n-1}$,  the matrix   
     \[ V_{\mathcal{N}_n}(z_1,\ldots, z_n) =   \begin{pmatrix} 
     1 & z_1 & \ldots &  z_1^{n-1} \\
     1 & z_2 & \ldots &  z_2^{n-1} \\
     \vdots & \vdots  &  & \vdots   \\
     1 & z_n & \ldots &  z_n^{n-1} \\
     \end{pmatrix}   \  \]
  is the classical Vandermonde matrix, and hence 
  \[   t_{\mathcal{N}_n}( \tau ) = \sup_{z\in \tau^N}    \prod_{1\leq i < j \leq n} \left| z_i -z_j     \right|  \   \]
is related to the logarithmic capacity, or  transfinite diameter, of $\tau$.  There are several possible higher-dimensional generalisations. We consider two  special cases, the `rectangular' and `homogeneous' versions.  
  
  \subsubsection{`Rectangular' Vandermonde matrices}  \label{sect: rectVdM}  Let $n_1,\ldots, n_r\geq 0$ and let
  \[ \mathcal{N}^{\mathrm{rec}}_{(n_1,\ldots, n_r)} =  \bigoplus_{0 \leq i_1 < n_1,  \ldots, 0\leq  i_r < n_r}  \Z\,   x_1^{i_1} x_2^{i_2} \ldots x_r^{i_r} \]
  be the free module generated by all monomials of degree $< n_i$ in each variable $x_i$. 
    It is the tensor product 
    \begin{equation}
    \mathcal{N}^{\mathrm{rec}}_{(n_1,\ldots, n_r)}  \cong     \mathcal{N}_{n_1} \otimes_{\Z}    \mathcal{N}_{n_2} \otimes_{\Z}
    \ldots    \otimes_{\Z}     \mathcal{N}_{n_r}  
    \end{equation} 
    where $ \mathcal{N}_{n} \subset \Z[x]$ consists of polynomials of degree $<n$  as in \S\ref{sect: onedimVdM}.   The  matrix $V_{(n_1,\ldots, n_r)} (z) = V_{\mathcal{N}^{\mathrm{rec}}}(z)$ is called the \emph{rectangular multivariable Vandermonde matrix}.
  Its rank and the  determinantal degree are
  \begin{equation} \label{degree :rectangularcase}
  \rank \, \,\mathcal{N}^{\mathrm{rec}}_{(n_1,\ldots, n_r)} =  n_1\ldots n_r\qquad \hbox{ and } \qquad 
   \deg\,   \left(\det\,   \mathcal{N}^{\mathrm{rec}}\right) =  \frac{n_1 \ldots n_r}{2} (n_1+ \ldots + n_r- r)   \ . 
  \end{equation} 
\begin{example} Let $r=2$, and $n_1=3, n_2=2$. Then if we denote the coordinates on $\mathbb{A}^2$ by $(x,y)$, we may write $z_i = (x_i, y_i) \in \C^2$ and we have
\[ \mathcal{V}_{(3,2)}(z_1,\ldots, z_6) =  \begin{pmatrix}  1 & x_1 &x_1^2 &  y_1 & x_1y_1 & x_1^2 y_1  \\
1 & x_2 &x_2^2 &  y_2 & x_2y_2 & x_2^2 y_2\\
1 & x_3 &x_3^2 &  y_3 & x_3y_3 & x_3^2 y_3\\
1 & x_4 &x_4^2 &  y_4 & x_4y_4 & x_4^2 y_4\\
1 & x_5 &x_5^2 &  y_5 & x_5y_5 & x_5^2 y_5\\
1 & x_6 &x_6^2 &  y_6 & x_6y_6 & x_6^2 y_6
\end{pmatrix}
\]
Its determinant has degree $9$. 
Its  absolute value   is a function on $(\C^2)^6$.
\end{example}

  \subsubsection{`Homogeneous' Vandermonde matrices}  \label{sect: homVdM}  Let $n\geq 0$ and let
  \[ \mathcal{N}_n^{\mathrm{hom}} =  \bigoplus_{0\leq k < n} \bigoplus_{ i_1 +\ldots + i_r =k}  \Z\,   x_1^{i_1} x_2^{i_2} \ldots x_r^{i_r} \]
  be the direct sum of the free modules generated by all homogeneous monomials  in  $x_1,\ldots, x_r$ of degree $k$, for $0 \leq k <  n$. 
  When $r=1$ it coincides with  the rectangular case.
   It has rank  and determinantal degree 
  \begin{equation}
  \label{degree:hom}
        \rank  \, \, \mathcal{N}_n^{\mathrm{hom}}  = \binom{n+r-1}{r} \qquad \hbox{ and } \qquad   \deg  \left(\det  \mathcal{N}_n^{\mathrm{hom}} \right) =r  \binom{n+r-1}{r+1} \ .
    \end{equation} 
\begin{example} Let $r=2$ and let $\mathbb{A}^2$ have coordinates $(x,y)$ as above. Then 
\[ \mathcal{V}^{\mathrm{hom}}_{3}(z_1,\ldots, z_6) =  \begin{pmatrix}  1 & x_1 & y_1 &   x_1^2 &  x_1y_1 &  y^2_1  \\
1 & x_2 &  y_2 & x_2^2&  x_2y_2 &  y^2_2\\
1 & x_3 &  y_3 &  x_3^2& x_3y_3 &  y^2_3\\
1 & x_4 &  y_4 &  x_4^2&x_4y_4 &  y^2_4\\
1 & x_5 &  y_5 &x_5^2 &x_5y_5 &   y^2_5\\
1 & x_6 &  y_6 &x_6^2  & x_6y_6 &  y^2_6
\end{pmatrix}
\]
The determinant has degree $8$.
Its  absolute value  is a  function on $(\C^2)^6$.

\end{example}


  \subsection{Upper bounds for Vandermonde determinants}
  We can deduce  upper bounds for  $t_{\mathcal{M}}(\tau)$ relative to basic operations on $\Z$-modules by applying formulae for determinants.
 Here we consider direct sums (using a formula due to Laplace which essentially computes $\det ( \mathcal{N}_1 \oplus \mathcal{N}_2 )$) and tensor products (using a formula in Appendix \ref{sect: Appendix}, which  computes  $\det ( \mathcal{N}_1 \otimes \mathcal{N}_2 )$). 
  
  \subsubsection{Tensor products}
 
   Let $\mathcal{N}_1$ and $\mathcal{N}_2$ be  two   free $\Z$-submodules of $R$ of finite rank,  such that the image of their tensor product $\mathcal{N}_1 \otimes \mathcal{N}_2$  in $R$ via the map $R\otimes R \rightarrow R$ is also free.

    
  \begin{thm}   \label{thm: upperboundtau}
  Let  $\mathcal{N}_1$ and $\mathcal{N}_2$ be free $\Z$-modules as above, and let $n_i = \rank\, \mathcal{N}_i$, for $i=1, 2$. 
  Let $\tau \subset  V(\C)$. Then there is an upper bound 
  \[ 
  t_{\mathcal{N}_1 \otimes \mathcal{N}_2} (\tau ) \leq   \frac{(n_1n_2)!}{H_{n_1,n_2}} \, \,   t^{n_2}_{\mathcal{N}_1}(\tau)     t^{n_1}_{\mathcal{N}_2}(\tau)    \  , 
  \]
  where the number $H$ is defined in \eqref{eqn: Hmndef}. 
      \end{thm} 
  \begin{proof} Let $z=(z_1,\ldots, z_{n_1n_2})$ in $V^{n_1n_2}(\C)$. 
  For a suitable ordering on the generators, the Vandermonde matrix  of the tensor product is an amalgam: 
  \[ V_{\mathcal{N}_1 \otimes \mathcal{N}_2}(z)  =   A(z) \star   B(z) \qquad \hbox{ where } \qquad 
A(z) =  \left(  \m_j (z_i) \right) \quad \hbox{,} \quad B(z) =  \left(  \m'_j (z_i) \right)  \]
  are matrices with $n_1n_2$ rows and $n_1$ (resp. $n_2$) columns, 
  where $\m_1,\ldots, \m_{n_1}$ is a basis for $\mathcal{N}_1$ and $\m'_1,\ldots, \m'_{n_2}$ is a basis for $\mathcal{N}_2$. 
    It follows from theorem \ref{thm: detAmalgam} that
  \[   \left| \det  V_{\mathcal{N}_1 \otimes \mathcal{N}_2}(z) \right| \leq  \frac{(n_1n_2)!}{H_{n_1,n_2}}  \,   \max_{I,J}  \left( \left| \det(V_{\mathcal{N}_1}(z_i)_{i\in I}) \right|^{n_2}
   \left| \det(V_{\mathcal{N}_2}(z_j)_{j\in J}) \right|^{n_1}  \right)    \]
   where $I, J$ are subsets of $\{1,\ldots, mn\}$ with $|I|=n_1$ and $|J|=n_2$. It follows that if $z  \in \tau^{n_1n_2}$, i.e., all components $z_i \in \tau$,    we have
   \[    \left| \det  V_{\mathcal{N}_1 \otimes \mathcal{N}_2}(z) \right| \leq  \frac{(n_1n_2)!}{H_{n_1,n_2}}    t^{n_2}_{\mathcal{N}_1}(\tau)     t^{n_1}_{\mathcal{N}_2}(\tau)            \ .  \]
   By taking the supremum over all $z\in \tau^{n_1n_2}$, we obtain the stated inequality.
  \end{proof}
  
  Note that even if the generators $\m_i \m'_j$ satisfy linear relations (i.e., the image of $\mathcal{N}_1\otimes \mathcal{N}_2$ is not free) the  left-hand side vanishes and the bound is still valid.
  The inequality in the previous proposition can, after taking suitable limits, become an equality in certain situations: see \cite{BrVdM}[Theorem 6.1] for an example in the case when $\tau$ is a product. The examples of  \S\ref{sect: hyperbolaregion}  show that the further one deviates  from such a  product situation, the worse the upper bound can become.

  \subsubsection{Direct sums}
   A similar, but more elementary, formula holds for direct sums. 
   Let $\mathcal{N}_1, \mathcal{N}_2$ be free $\Z$-submodules of $R$ of finite rank such that the image of  $\mathcal{N}_1 \oplus \mathcal{N}_2$ in $R$ via the map $R\oplus R \rightarrow R$ is free.
   
   \begin{prop}   \label{prop: upperboundtaudirectsum}
  Let  $\mathcal{N}_1$ and $\mathcal{N}_2$ be as  above, and let $n_i = \rank\, \mathcal{N}_i$, for $i=1, 2$.  Let
   $\tau \subset  V(\C)$. Then there is an upper bound 
  \[ 
  t_{\mathcal{N}_1 \oplus \mathcal{N}_2} (\tau ) \leq    \binom{n_1+n_2}{n_1}  \, \,   t_{\mathcal{N}_1}(\tau)     t_{\mathcal{N}_2}(\tau)    \  .  
  \]
    \end{prop} 
  
  \begin{proof} Let $z =(z_i)_{1\leq i \leq n_1+n_2} \in V(\C)^{n_1+n_2}$.
  Expanding  the determinant of $V_{\mathcal{N}_1 \oplus \mathcal{N}_2}(z)$  via Laplace's formula  as  a sum over  complementary minors gives
     \[ \det  V_{\mathcal{N}_1 \oplus \mathcal{N}_2} (z)  = \sum_{I,J} \varepsilon_{I,J}  \det V_{\mathcal{N}_1} (z_I)  \det V_{\mathcal{N}_2} (z_J)     
    \] 
    where the sum is over all subsets $I \subset \{1,\ldots, n_1+n_2\}$ such that $|I|=n_1$,  and $J$ is its complement.   The coefficients satisfy $\varepsilon_{I,J} \in \{1,-1\}$.  
    Now take the absolute value and the supremum over  $z\in \tau^{n_1+n_2}$.  \end{proof}

  \section{Supremal transfinite diameter} \label{sect: Suptransfinite}

  We define a notion of transfinite diameter in which the limit is replaced by a limit supremum, and work in the setting of  affine schemes $V$.  The limits will be taken with respect to  certain exponents $e$, which are introduced to allow additional flexibility. In this section we state and prove some properties of the supremal diameter with respect to natural operations on modules.
  
  In brief, the supremal transfinite diameter provides a pairing between subsets $\tau \subset V(\C)$ and free modules $\mathcal{M}\subset \Or(V)$ with some additional data. It produces a number denoted $\ts_{\mathcal{M}}(\tau) \in \R_{\geq 0} \cup \{\infty\}$. 
 
  \subsection{Definition} \label{sect: setupexponentfunctions}
  Throughout, let  $V = \mathrm{Spec} \, R$ be as in  \S\ref{subsect: GenVdM}. Consider  a free $\Z$-module 
  $\mathcal{M} \subset R$ which has a filtration by  (necessarily free) $\Z$-submodules  of finite rank
  \[  \mathcal{M}_n \subset \mathcal{M} \qquad \hbox{ for } n \geq 1 \ ,  \]
  such that $\mathcal{M}_n \subset \mathcal{M}_{n+1}$ and $\mathcal{M}= \varinjlim \mathcal{M}_n.$    Consider a sequence of  \emph{exponents} 
  \[  e_n^{\mathcal{M}} \in \mathbb{N} \ . \]
  The exponents will usually  grow a little faster than the rank: 
  \begin{equation}  \label{exponentgrowthcondition} 
  \lim_{n\rightarrow \infty}  \frac{ e_n^{\mathcal{M}}}{r_n \log r_n } = \infty\  \qquad \hbox{ where } r_n= \rank \, \mathcal{M}_n \ .
  \end{equation} 
  \begin{defn} Given a filtered $\Z$-module $\mathcal{M}$, and exponents $e^{\mathcal{M}}$, we define the supremal transfinite diameter of a set $\tau \subset V(\C)$  to be
  \[  \ts_{\mathcal{M},e} (\tau)=  \limsup_{n\rightarrow \infty}   \left(t_{\mathcal{M}_n}(\tau) \right)^{\frac{1}{e^{\mathcal{M}}_n}} \ ,\]
  where $t_{\mathcal{M}_n}(\tau)$ is the supremum of the Vandermonde determinant $|\det \mathcal{M}_n|$ on $\tau^{\rank \mathcal{M}_n}$  (definition \ref{defn: VdMdet}).
  \end{defn}
  Note that the limit supremum $\ts_{\mathcal{M}} (\tau)$ may be infinite, but always exists. It follows immediately from the definition that if $\tau \subset \tau'$ then
  \begin{equation}  \label{eqn: Subsetinequality}   \ts_{\mathcal{M},e} (\tau) \leq \ts_{\mathcal{M},e} (\tau') \ . 
  \end{equation}
  
 \begin{rem} \label{rem: noe} When the exponent function is clear from the context, we shall drop it  and simply  write $\ts_{\mathcal{M}}$. Note that if $e'_n= \alpha_n e_n$, where the $\alpha_n$ tend to a limit $\alpha$ as $n\rightarrow \infty$, one has  the identity:
 \[  \ts_{\mathcal{M},e'} (\tau) = \ts_{\mathcal{M},e} (\tau)^{\frac{1}{\alpha}}\ . \]
 \end{rem}

  \begin{ex} \label{ex: Supstandardexamples} Consider the  case $V= \mathbb{A}^r$. Denote its coordinate ring by $\mathcal{O}(V) = \Z[x_1,\ldots, x_r]$. 
  \begin{enumerate}
  \item
  If $r=1$ we may take $\mathcal{M}_n \subset \mathcal{M}=\Z[x]$ to be the subspace of polynomials in $x$ of degree $<n$, and $e^{\mathcal{M}}_n = \binom{n}{2}$. Then, as in  \S\ref{sect: ClassicalVdM}, 
  \[  \ts_{\mathcal{M}} (\tau) =   \limsup_{n\rightarrow \infty}  \left(\sup_{z_i \in \tau}  \prod_{1\leq i<j\leq n} | z_i -z_j| \right)^{\frac{2}{n(n-1)}} \    \]
  which equals  the  transfinite diameter of  $\tau \subset \C$ whenever it exists (e.g., if $\tau$ is compact).  \\
  
  \item For general $r$, let  $0 <  a_1,\ldots, a_r \in \mathbb{N}$ and let $\mathcal{M}=\Z[x_1,\ldots, x_r]$ equipped with the filtration by submodules $\mathcal{M}_n = \mathcal{N}^{\rec}_{(a_1n, \ldots, a_rn)}$ with exponent function
  \begin{eqnarray} \label{exponent:rectangularcase}
 e^{\mathcal{M}}_{n}  &=  &  n^{r+1}  \frac{a_1 \ldots a_r}{2} (  a_1+ \ldots + a_r- \frac{r}{n})    \\
 & \sim &   \frac{n^{r+1}}{2}    a_1\ldots a_r (a_1+\ldots + a_r)   \qquad \hbox{ as } n \rightarrow \infty \nonumber
\end{eqnarray} 
 as in  \S\ref{sect: rectVdM}. 
  Then we define for any $\tau \subset \C^r$ 
  \[ \ts_{(a_1,\ldots, a_r)}^{\ \rec} (\tau) =    \limsup_{n\rightarrow \infty}   \left(t_{\mathcal{N}^{\rec}_{(a_1 n,\ldots, a_rn)}}(\tau) \right)^{\frac{1}{e^{\mathcal{M}}_n}}  \ . \]
  It equals  the `rectangular' transfinite diameter of $\tau$ if the latter  exists, which holds for example when $\tau$ is compact. This is   a special case of \cite{BBL2018}[Theorem 5.1, Rem. 5.2]. \\

   \item  Let $n\geq 0$ and let $\mathcal{M}=\Z[x_1,\ldots, x_r]$  with the filtration by submodules $\mathcal{M}_n =  \mathcal{N}_n^{\mathrm{hom}}$ as defined in 
   \S\ref{sect: homVdM}. The exponent function is 
    \begin{equation}
  \label{exponent:hom}
  e^{\mathrm{hom}}_{n} =r\, \binom{n+r-1}{r+1}  \sim  \left(\frac{r}{r+1} \right) \frac{n^{r+1}}{r!}  \ .
    \end{equation} 
Then we define for any $\tau \subset \C^r$ 
  \[ \ts_{n}^{\hom} (\tau) =    \limsup_{n\rightarrow \infty}   \left(t_{\mathcal{N}^{\hom}_{ n}}(\tau) \right)^{\frac{1}{e^{\mathcal{M}}_n}}  \ . \]
  It is   the usual `homogeneous'  transfinite diameter of $\tau$  when it  exists.   Similarly, the limit exists when $\tau$ is compact as a consequence of  \cite{BBL2018}[Theorem 5.1, Rem. 5.2].
   \\
    \end{enumerate}
  These particular examples are special cases of what is  called  the $C$-transfinite diameter  in \cite{LevenbergWielonsky}, which in turn is a special case of the more general definition  for an arbitrary module given above.
   \end{ex}

   \begin{rem} 
   When the degree   of the Vandermonde determinant makes sense, we define  the exponent function to be  $\det {\mathcal{M}_n} \subset R^{\rank  \,  \mathcal{M}_n}$, as in the examples above.
  More generally, let $\Gm \times \mathbb{A}^r \rightarrow \mathbb{A}^r$ be  the action which on complex points scales  coordinates by $\lambda \in  \Gm(\C) = \C^{\times}$. Then  $(1)$-$(3)$ satisfy
  \begin{equation} \label{eqn: scalinglaw}
  \ts_{\mathcal{M}}(\lambda \tau) =  \lambda \, \ts_{\mathcal{M}}(\tau)\ \quad \hbox{ for all } \lambda \in \C^{\times} \ .  
  \end{equation} 
  In general, whenever $V$ admits a $\Gm$-action of this sort and  its affine ring is graded with  non-negative degrees and is of finite rank in each graded degree,  it is advisable to choose the   exponent function to be 
  \[ e_n^{\mathcal{M}} = \deg  \left(\det \mathcal{M}_n\right) = \sum_{i=1}^{\rank \mathcal{M}_n} \deg(\m_i) \]
  where the $\m_i$ are any choice of independent  homogenous generators for $\mathcal{M}_n$. 
 It is compatible with the grading and guarantees  that  a  scaling law  \eqref{eqn: scalinglaw} holds.
 \end{rem}
 
 \subsubsection{Functoriality and torsion}   The supremal diameter is functorial in the following sense.
Let $f: V \rightarrow V' = \mathrm{Spec}(R')$ be a morphism as in  \S\ref{subsect: morphisms}, and consider a free $\Z$-module 
 $\mathcal{N} \subset R'$ which is $f$-admissible.  Suppose that $e^{\mathcal{N}}$ is an exponent function and let  $e^{f^*\mathcal{N}}_n=  e^{\mathcal{N}}$. Then  by \eqref{eqn: tmorphism}
  \begin{equation} \label{eqn: functorial} 
  \ts_{f^*\mathcal{N}} (\tau) =  \ts_{\mathcal{N}}( f \tau )\ .
  \end{equation}

The supremal diameter is insensitive to `small' modifications of the module $\mathcal{M}$. For example, 
 suppose that $\mathcal{M}' \subset \mathcal{M}$ are as above  and $\mathcal{M}'_n \subset \mathcal{M}_n$ is free of equal rank. Let $C_n$ denote the cardinality of the torsion module $\mathcal{M}_n/\mathcal{M}'_n$. If the corresponding exponent functions are equal: $e^{\mathcal{M}'}_n = e^{\mathcal{M}}_n$,  then  by \eqref{eqn: tscalestorsion}
 \[   \ts_{\mathcal{M}} (\tau) =   \ts_{\mathcal{M}'} (\tau)  \quad \hbox{ whenever } \quad 
 \lim_{n\rightarrow \infty}  C_n^{\frac{1}{e^{\mathcal{M}}_n}}=1  \ .\] 
 The latter condition on the limit   holds unless $C_n$ grows  rapidly as a function of the rank of $\mathcal{M}_n$. 
 Consequently, the $\Z$-module structure on $\mathcal{M}$ does not usually play an important role in the theory. 
 
 \begin{example}
 For example, if $\mathcal{M}$ is as in $(1)$ above, and \
 \[ \mathcal{M}'_n= \Z[1, 2x, 3  x^2, \ldots, n x^{n-1}]\ , \]
 then $C_n= n!$ and $\ts_{\mathcal{M'}}(\tau) = \ts_{\mathcal{M}}(\tau)$. But if, for example, 
 \[ \mathcal{M}'_n= \Z[1, 2x,  4 x^2, \ldots,  2^{n-1} x^{n-1}] \ , \]
 then  $C_n \sim 2^{n^2/2}$ and so $\ts_{\mathcal{M}'}(\tau) =    \ts_{\mathcal{M}}(2\tau)  =2  \,\ts_{\mathcal{M}}(\tau)  $ by \eqref{eqn: functorial} where $f$ is the map $x\mapsto 2 x$.
\end{example}

  \subsection{Summary of known results}
    For our applications, we are interested in upper bounds for the supremal diameter. 
    We briefly list some known results on the transfinite diameter.
    
    \subsubsection{$1$-dimensional case} This situation is classical and well-studied. One knows that 
    \begin{equation} \label{eqn:supinterval}   \ts^{\hom}([a,b])= t([a,b]) =   \frac{(b-a)}{4}\end{equation}
    whenever  $a<b$.  The module and exponent function are as in \S\ref{sect: ClassicalVdM}.
  
  


  \subsubsection{Rectangular transfinite diameter}
   Unfortunately, there are few results which establish the precise value of the transfinite diameter in higher dimensions.  For the rectangular case \S\ref{sect: rectVdM}, very little seems to be known. We show below that the rectangular diameter $\ts^{\rec}_{(a_1,\ldots,a_r)}$  is multiplicative. In particular  
   \[ \ts^{\rec}_{(1,\ldots,1)} (\tau_1 \times \tau_2) =    \left(\ts^{\rec}_{(1,\ldots,1)} (\tau_1 )\right)^{\frac{r}{r+s}}  \left(  \ts^{\rec}_{(1,\ldots,1)} ( \tau_2) \right)^{\frac{s}{r+s}}   \]
   for any $\tau_1 \subset \C^r$ and $\tau_2 \subset \C^s$.
   In the case $r=s=1$ we recover the formula  \cite{LevenbergWielonsky}[Theorem 5.1].    
   Since, in one dimension, the rectangular and homogeneous transfinite diameters coincide,  we may use \eqref{eqn:supinterval} to deduce the rectangular transfinite diameter of a rectangular box:
   \begin{equation} \label{eqn: trecsquarebox}  \ts^{\rec}_{(1,1)} (  [a,b]\times [c,d]) =  \frac{\sqrt{(b-a)(d-c)}}{4} \qquad \hbox{ where }  a<b \ , \ c<d \ .
   \end{equation}
   
  \subsubsection{Homogeneous transfinite diameter} This is the standard notion of transfinite diameter. We shall denote it by $\tr$.
  The existence of the limits for compact sets was proven by \cite{Zaharjuta}, in which case $\tr = \ts^{\hom}$. 
The  multiplicativity  of the homogeneous transfinite diameter was established in \cite{SchifferSiciak}, \cite[Theorem 1]{BloomCalvi}. Let $\tau_1 \subset \C^{m}$ and $\tau_2 \subset \C^{n}$ be compact sets. Then
  \begin{equation} \label{homogeneoustransfinitemultiplicative}
  \tr(\tau_1 \times \tau_2)  =  \tr(\tau_1)^{\frac{m}{m+n}}  \tr(\tau_2)^{\frac{n}{m+n}}   \ . 
  \end{equation} 
 Thus $\ts^{\hom}(\tau_1 \times \tau_2) =    \ts^{\hom}(\tau_1)^{\frac{m}{m+n}} \ts^{\hom}(\tau_2)^{\frac{n}{m+n}} $.   It may be possible to derive this  directly from the formula for the determinant of a matrix amalgam (Appendix \S\ref{sect: Appendix}).  In particular, the homogeneous supremal diameter for a rectangular box coincides with the rectangular version, and is given by \eqref{eqn: trecsquarebox}.
  
  Let $\tau \subset \C^n$. Then for any $P \in \mathrm{GL}_n(\C)$ one has the formula
  \begin{equation} \label{eqn: thomsimilarity}   \tr(P  \tau) =  |\det(P)|^{\frac{1}{n}}\,  \tr(\tau)\ ,  \end{equation} 
  which can  be proven from the definition \S\ref{sect: homVdM} using standard  properties of the determinant and the scaling law \eqref{eqn: scalinglaw}.
  Equation \eqref{eqn: thomsimilarity} may be generalised to certain homogeneous polynomial transformations. Indeed,  \cite{BloomCalvi, DeMarcoRumely} show that
  for any regular polynomial map $F: \C^n \rightarrow \C^n$ of degree $d$ one has
  \[  \tr (F^{-1} \sigma) = |\mathrm{Res} F_h|^{-\frac{1}{nd^n}}  \tr(\sigma)^{\frac{1}{d}}\]
   where $\mathrm{Res}$ denotes a certain resultant associated to  the homogeneous part $F_h$ of  $F$.

  In addition, the   following specific results in two dimensions will be useful:
  \begin{enumerate}
  \item (Unit ball). It is known by Bos  \cite{Bos}  that if $B= \{(x,y) \in \R^2: x^2+y^2 \leq 1\}$ denotes the unit ball in $\R^2$,   the transfinite diameter equals:
\[   \tr(B) =  \frac{1}{\sqrt{2e}}\ .  \]
\item (Unit triangle).
By applying the map $(x,y) \mapsto (x^2, y^2)$, one  deduces  \cite[p. 302]{BloomCalvi} that   the transfinite diameter of a triangle $T\subset \R^2$ with vertices $(0,0)$, $(1,0)$ and $(0,1)$ is 
\[ \tr(T) = \frac{1}{2e} \sim 0.18393\ldots \]

\item (General triangle). 
By applying an element  \eqref{eqn: thomsimilarity} in $\mathrm{GL}_2(\R)$,  we deduce  that for any triangle $T\subset \R^2$, its transfinite diameter equals
\begin{equation}  \label{eqn: transfinitetriangle}  \tr(T) =  \frac{\mathrm{vol}(T)^{\frac{1}{2}} }{ e \, \sqrt{2} }\ .
\end{equation} 
This  will be used to estimate the transfinite diameter associated to $\zeta(2)$.
  \end{enumerate} 
  The homogeneous transfinite diameter of the ball and simplex in $\R^n$ were computed in \cite{BloomBosLevenberg}.  
    
  \subsection{Supremal  diameter for tensor products}

  Let $\mathcal{M}$, $\mathcal{N}$ be two filtered   $\Z$-submodules of $R$ and suppose that $\mathcal{M} \otimes \mathcal{N} \subset R$ is a free $\Z$-module equipped with the filtration $(\mathcal{M} \otimes \mathcal{N} )_n =  \mathcal{M}_n \otimes \mathcal{N}_n$.

  \begin{thm}  \label{thm: tsuptensorproductUB}  Let $\mathcal{M}$, $\mathcal{N}$, $\mathcal{M} \otimes\mathcal{N}$   be as above with exponent functions  $e^\mathcal{M}$, $e^\mathcal{N}$, $e^{\mathcal{M}\otimes\mathcal{N}}$. Suppose that  $e^{\mathcal{M}\otimes \mathcal{N}}$ satisfies \eqref{exponentgrowthcondition} and 
      \[   \lim_{n \rightarrow \infty}  \frac{e^{\mathcal{M}}_n \, \rank\, \mathcal{N}_n}{ e^{\mathcal{M}\otimes\mathcal{N}}_n}  =\alpha \quad \hbox{ and } \quad    \lim_{n \rightarrow \infty}  \frac{e^{\mathcal{N}}_n \, \rank\, \mathcal{M}_n}{ e^{\mathcal{M}\otimes\mathcal{N}}_n}= \beta \ .   \]
   Then we have the upper bound
    \begin{equation} 
     \label{eqn: tsuptensorproductUB} 
      \ts_{\mathcal{M} \otimes \mathcal{N}} (\tau )  \leq     \, \,   \left(\ts_{\mathcal{M}}(\tau) \right)^{\alpha}   \left(\ts_{\mathcal{N}}(\tau) \right)^{\beta}     \  .  
  \end{equation}
      \end{thm} 
    
  \begin{proof} For brevity write $r_n=\rank \, \mathcal{M}_n$, $s_n=  \rank \, \mathcal{N}_n$. Let $n$ be sufficiently large that $r_n,s_n>0$. 
  By Theorem \ref{thm: upperboundtau},
   \[ 
   t_{(\mathcal{M} \otimes \mathcal{N})_n} (\tau )  \leq   \frac{ (  r_n s_n  )!}{H_{r_n,s_n}} \, \,   t^{s_n}_{\mathcal{M}_n}(\tau)     t^{r_n}_{\mathcal{N}_n}(\tau)    \  .  
  \]
 By applying the bound $m! \leq m^m$ for integers $m\geq 0$, we deduce that
\[  \left( \frac{ (  r_n s_n  )!}{H_{r_n,s_n}}\right)^{\frac{1}{e_n^{\mathcal{M}\otimes\mathcal{N}}}} \leq     ((  r_n s_n  )!)^{\frac{1}{e_n^{\mathcal{M}\otimes\mathcal{N}}}} \leq \exp \left( \frac{r_ns_n}{e_n^{\mathcal{M}\otimes\mathcal{N}}} \log(r_ns_n) \right) \ . \]
Since $\rank (\mathcal{M} \otimes \mathcal{N})_n = r_n s_n$,  the assumption  \eqref{exponentgrowthcondition} implies that the previous expression tends to $1$ in the limit as $n\rightarrow \infty$.
 It follows that 
  \[ 
   \ts_{\mathcal{M} \otimes \mathcal{N}} (\tau )  \leq     \, \, \limsup_{n\rightarrow \infty}    \left(   \left(t^{\frac{1}{e^{\mathcal{M}}_n}}_{\mathcal{M}_n}(\tau) \right)^{\alpha_n}        \left(t^{\frac{1}{e^{\mathcal{N}}_n}}_{\mathcal{N}_n}(\tau) \right)^{\beta_n} \right)    \quad \hbox{ where } \quad 
    \alpha_n =     \frac{e^{\mathcal{M}}_n \, \rank\, \mathcal{N}_n}{ e^{\mathcal{M}\otimes\mathcal{N}}_n} \ ,  \ \beta_n =  \frac{e^{\mathcal{N}}_n \, \rank\, \mathcal{M}_n}{ e^{\mathcal{M}\otimes\mathcal{N}}_n}  \ . \]
 By assumption $\alpha_n \rightarrow \alpha$ and $\beta_n\rightarrow \beta$ as $n\rightarrow \infty$.   The result follows from the definition of the limit supremum, together with  the fact that  the previous expression only involves positive quantities,  and that  the exponential and logarithm restricted to the positive real axis are   increasing monotone functions. 
  \end{proof} 
  \begin{cor}  \label{cor: stationarytp} Let $\mathcal{M}$, $\mathcal{N}$ and $\mathcal{M} \otimes \mathcal{N}$ be as above, but  suppose that $\mathcal{N}_n $ is stationary, i.e., $\mathcal{N}_n = \mathcal{N}$ for all $n$ sufficiently large. In particular, $\mathcal{N}$ is of finite rank. Suppose we are given exponent functions $e^{\mathcal{M}}$ and $e^{\mathcal{M} \otimes \mathcal{N}}$
  where the latter satisfies \eqref{exponentgrowthcondition}, such that the following limit exists:
  \[   \lim_{n \rightarrow \infty}  \frac{e^{\mathcal{M}}_n \, \rank\, \mathcal{N}}{ e^{\mathcal{M}\otimes\mathcal{N}}_n} =\alpha\ .   \] 
    Then,  if $\tau$ is bounded,  we have the upper bound:
   \[   t^{\sup}_{\mathcal{M} \otimes \mathcal{N}} (\tau )  \leq     \, \,   \left(t^{\sup}_{\mathcal{M}}(\tau) \right)^{\alpha}    \  .  
  \]
  \end{cor} 
  
  \begin{proof}  Define an exponent function for $\mathcal{N}_n$ to be $e^{\mathcal{N}}_n = \rank\,  \mathcal{N}$.  With this definition, the limit $\beta$ in theorem \ref{thm: tsuptensorproductUB} is zero. Since $\tau$ is bounded,  $t^{\sup}_{\mathcal{N}}(\tau) $ is finite, and   hence \eqref{eqn: tsuptensorproductUB} implies the stated inequality, since
  $ \left(t^{\sup}_{\mathcal{N}}(\tau) \right)^{\beta}  =1$. 
  \end{proof}
  
  \begin{rem}
  A natural definition of $e^{\mathcal{M} \otimes \mathcal{N}}$, which applies in the  examples \ref{ex: Supstandardexamples}, is 
  \begin{equation} \label{etensorcanonical}
  e^{\mathcal{M} \otimes \mathcal{N}}_n =  e_n^{\mathcal{M}} \, \rank \, \mathcal{N}_n + e_n^{\mathcal{N}}  \rank\, \mathcal{M}_n\ .  \
  \end{equation} 
  When this holds,  we have $\alpha+\beta=1$ in the setting of theorem \ref{thm: tsuptensorproductUB}. 
  \end{rem} 
  In some cases, \eqref{eqn: tsuptensorproductUB} is an equality.  
  The following corollary is not as general as possible since it assumes the existence of  limits,  but is related to results in the literature on transfinite diameters \cite{LevenbergWielonsky}. 
  It  proves that the (supremal)  transfinite diameter is multiplicative with respect to external products.
  
 
 \begin{cor} Let $\mathcal{N}_i \subset R_i$ be  free $\Z$-modules filtered by finite rank submodules and equipped with exponent functions $e^{\mathcal{N}_i}$. Consider the tensor product $\mathcal{N}_1 \otimes \mathcal{N}_2 \subset R_1 \otimes R_2$, which  is free and equipped with the tensor product filtration. 
 Suppose that
     \[   \lim_{n \rightarrow \infty}  \frac{e^{\mathcal{N}_i}_n \, \rank\, (\mathcal{N}_{2-i})_n}{ e^{\mathcal{N}_1\otimes\mathcal{N}_2}_n}  =\alpha_i 
    \quad  \hbox{ for } \quad i =1,2\  .  \]
 Write  $V_i = \mathrm{Spec} \,R_i$, and  consider $\tau_i \subset V_i(\C)$ for $i=1,2$. Suppose that  the limit suprema $\ts_{\mathcal{N}_i}(\tau_i)$  are actual limits for $i=1,2$.   Then 
  \[  \ts_{\mathcal{N}_1 \otimes \mathcal{N}_2} (\tau_1 \times \tau_2)  = \left(\ts_{\mathcal{N}_1}(\tau_1)\right)^{\alpha_1}    \left(\ts_{\mathcal{N}_2}(\tau_2) \right)^{\alpha_2} \  
\] 
and the limit supremum on the left is also a limit.
 \end{cor}
 
 \begin{proof} Let $R= R_1 \otimes R_2$. There are canonical embeddings of $\Z$-modules $R_i \rightarrow R$. 
 The inequality  $ \ts_{\mathcal{N}_1 \otimes \mathcal{N}_2} (\tau_1 \times \tau_2)  \leq \left(\ts_{\mathcal{N}_1}(\tau_1)\right)^{\alpha_1}    \left(\ts_{\mathcal{N}_2}(\tau_2) \right)^{\alpha_2} $
 follows from theorem \ref{thm: tsuptensorproductUB}.
 To obtain a lower bound  the idea is as follows. For all sufficiently large $n$,  choose a configuration of $r_i=\rank (\mathcal{N}_i)_n$ points $z' \in \tau_1^{r_1} $ 
 and $z'' \in \tau_2^{r_2}$ such that $|\det V_{(\mathcal{N}_1)_n}(z')|$ and $|\det V_{(\mathcal{N}_2)_n}(z'')|$ are arbitrarily close to $\ts_{\mathcal{N}_1}(\tau_1)$ and $\ts_{\mathcal{N}_2}(\tau_2)$, respectively. Consider the configuration of  $r_1r_2$ points $w_{ij}=(z'_i, z''_j)$ in $\tau_1 \times \tau_2$ where $1\leq i \leq r_1$, $1\leq j \leq r_2$. For a suitable ordering of these points, one may show \cite[\S3.2,\S6]{BrVdM}  that the Vandermonde matrix $V_{(\mathcal{N}_1\otimes \mathcal{N}_2)_n} (w_{ij}) $, expressed as an amalgam of two matrices, is in fact their  Kronecker tensor product and hence 
 \[ |\det V_{(\mathcal{N}_1\otimes \mathcal{N}_2)_n} (w_{ij})| = 
 \left| \det V_{(\mathcal{N}_1)_n}(z') \right|^{r_2}  \left|\det V_{(\mathcal{N}_2)_n}(z'')\right|^{r_1} \ . \]
 By raising to the power $(e_n^{\mathcal{N}_1 \otimes \mathcal{N}_2})^{-1}$, we obtain a lower bound for  $t_n^{\mathcal{N}_1\otimes \mathcal{N}_2}(\tau_1\times \tau_2)$ which is arbitrarily close to  $\left(\ts_{\mathcal{N}_1}(\tau_1)\right)^{\alpha_1}    \left(\ts_{\mathcal{N}_2}(\tau_2) \right)^{\alpha_2}$.
  \end{proof}

     
     \subsection{Supremal diameter for direct sums}
     Let $\mathcal{M}$, $\mathcal{N} \subset R$ be two  free $\Z$-modules equipped with filtrations by finite rank submodules.  The direct sum 
    $\mathcal{M} \oplus \mathcal{N} $
     is given the  filtration $(\mathcal{M} \oplus \mathcal{N})_n=   \mathcal{M}_n \oplus \mathcal{N}_n$. We assume that it is  free when viewed as a submodule of $R$ via the map $R\oplus R \rightarrow R$.

     \begin{prop} \label{prop: directsums}
     Suppose that we are given free $\Z$-submodules $\mathcal{M}$, $\mathcal{N}$ and $\mathcal{M} \oplus \mathcal{N}$  of $R$ as above, and exponents $e^{\mathcal{M}}$, $e^{\mathcal{N}}$, $e^{\mathcal{M}\oplus \mathcal{N}}$, where the latter satisfies \eqref{exponentgrowthcondition}. 
      If 
    \[  \lim_{n\rightarrow \infty} \frac{ e_n^{\mathcal{M}} }{ e_n^{\mathcal{M} \oplus \mathcal{N} } } =  \alpha \quad \  , \quad   \lim_{n\rightarrow \infty} \frac{ e_n^{\mathcal{N}} }{ e_n^{\mathcal{M} \oplus \mathcal{N}  }} = \beta   \]
     we have the inequality
   \[ \ts_{\mathcal{M} \oplus \mathcal{N}} (\tau) \leq \left(\ts_{\mathcal{M}}(\tau) \right)^{\alpha}    \left(\ts_{\mathcal{N}}(\tau)\right)^{\beta} \ . \]
     \end{prop} 
     
     \begin{proof} Let $n\geq 0$,  $r_n =  \rank \,   \mathcal{M}_n$,  and  $s_n =  \rank \,  \mathcal{N}_n$.       By proposition \ref{prop: upperboundtaudirectsum}
    \[  t_{\mathcal{M}_n \oplus \mathcal{N}_n} (\tau) \leq    \binom{r_n+s_n}{r_n}   t_{\mathcal{M}_n} (\tau) t_{\mathcal{N}_n} (\tau) \ . \]
  Since  $ \binom{r_n+s_n}{r_n} \leq 2^{r_n+s_n}$ and $\frac{e_n^{\mathcal{M}\oplus \mathcal{N} } }{ r_n+s_n }   \rightarrow  \infty$
    as $n\rightarrow \infty$,  we have 
  \[ \binom{r_n+s_n}{r_n}^{\frac{1}{e_n^{\mathcal{M} \oplus \mathcal{N}}}}\rightarrow 1\quad  \hbox{ as } n\rightarrow \infty\ , \]
 and we conclude by raising both sides of the previous inequality to the power $\frac{1}{e_n^{\mathcal{M} \oplus \mathcal{N}}}$ and letting $n\rightarrow \infty$. This uses similar properties of limit suprema as set out at the end of the proof of theorem \ref{thm: tsuptensorproductUB}. 
        \end{proof}
     
     \begin{rem}  \label{rem: exponentdirectsum} The canonical exponent function for a direct sum is
     \[ e_n^{\mathcal{M} \oplus \mathcal{N}} = e_n^{\mathcal{M}} +  e_n^{\mathcal{N}}\ . \]
 With this definition, the exponents in the previous theorem satisfy 
      \[ \alpha  =  \lim_{n\rightarrow \infty}   \frac{e_n^{\mathcal{M}}}{ e_n^{\mathcal{M}} +  e_n^{\mathcal{N}}} \qquad , \qquad   \beta  =  \lim_{n\rightarrow \infty}   \frac{e_n^{\mathcal{N}}}{ e_n^{\mathcal{M}} +  e_n^{\mathcal{N}}}  \]
     and in particular, $\alpha + \beta =1$.    The proof only requires that
     $  \frac{e_n^{\mathcal{M}\oplus \mathcal{N} } }{ \rank(\mathcal{M}_n \oplus \mathcal{N}_n) }   \rightarrow \infty$ 
      which is weaker than    \eqref{exponentgrowthcondition}.
         \end{rem}

   
   One may also retrieve  corollary \ref{cor: stationarytp} using the previous  proposition by writing $\mathcal{M} \otimes \mathcal{N}$ as a direct sum 
   $\oplus_i \m_i \mathcal{M}$ where the $\m_i$ are a basis  of $\mathcal{N}$.

   \begin{cor}  \label{cor: directsums} Suppose that  $\mathcal{M}, \mathcal{N}_1,\ldots, \mathcal{N}_k \subset R$ are  free $\Z$-modules with a  filtration by finite rank submodules such that, for all  $n \geq 0$, 
   \[  \mathcal{M}_n = \bigoplus_{i=1}^k   (\mathcal{N}_i)_n\ ,  \]
and exponents $e^{\mathcal{N}_i}$ and  $e^{\mathcal{M}}$, where the latter satisfies \eqref{exponentgrowthcondition}. 
      If 
    $ \lim_{n\rightarrow \infty} \frac{ e_n^{\mathcal{N}_i} }{ e_n^{\mathcal{M}  } } =  \alpha_i
    $
    then 
   \[ \ts_{\mathcal{M} }  (\tau) \leq \prod_{i=1}^k  \left(\ts_{\mathcal{N}_i}(\tau) \right)^{\alpha_i}   \ . \]
     \end{cor}

     

  \section{Upper bounds for a square region bounded by a hyperbola} \label{sect: hyperbolaregion}
  As an application of these techniques, we compute some  upper estimates for the  rectangular supremal transfinite diameter $\ts^{\rec}_{(1,1)}(\tau_{\epsilon})$ of the region:
  \[ \tau_{\epsilon} = \{ (x,y) \in \R^2:    0\leq x, y\leq 1 \ , \   xy \leq \epsilon\}
  \]
 and show that $\ts^{\rec}_{(1,1)}(\tau_{\epsilon}) \leq \left( \frac{\epsilon}{16}\right)^{\frac{1}{3}}$ as $\epsilon \rightarrow 0$.
  A key point is that we reduce the \emph{rectangular} supremal diameter (about which little seems to be known) to the \emph{homogeneous}  diameter (about which much more is known)  of a different region.  
   If  formulae for the transfinite diameters of two-dimensional figures in $\R^2$ such as a quadrilateral were available, our estimates could be significantly improved. 
  
  \subsection{Na\"ive bound} Since $\tau_{\epsilon} \subset [0,1]^2$ we deduce from \eqref{eqn: Subsetinequality}   that
  \[  \ts_{(1,1)}^{ \rec}(\tau_\epsilon) \leq   \ts_{(1,1)}^{\rec} ([0,1]^2) =   (\ts^{\hom} ([0,1])^{\frac{1}{2}})^2= \frac{1}{4}\ ,  \]
which follows from   \eqref{eqn: trecsquarebox} and    \eqref{eqn:supinterval} since $\ts_{(1)}^{\rec}= \ts^{\hom}$ in dimension $1$.

  \subsection{Use of tensor products: map to a curved rhombus} 
  \begin{lem} \label{lem: isoZmodulesxyxplusy}  Let $\mathcal{N} = \Z\oplus \Z x$. The inclusion 
  \[  i: \Z[xy, x+y] \otimes \mathcal{N} \To \Z[x,y]\]
  is an isomorphism of $\Z$-modules.
  \end{lem} 
 \begin{proof} Let  $u = xy$ and $v =x+y$. Then one has the relation 
 $x^2 = xv - u$. Consider the ring structure on $ \Z[u ,v] \otimes \mathcal{N}$ induced by the  identification of  $\Z$-modules:
 \[ \Z[u ,v] \otimes \mathcal{N} \cong \left( \Z[u,v] \otimes \Z[x] \right)/I \]
 where $I$ is the ideal generated by $x^2-xv+u$. Thus the natural map $i$  of $\Z$-modules  can be promoted to a map of rings for the induced ring structure:
\[ (u,v,x) \mapsto (xy, x+y, x) :  \left( \Z[u,v] \otimes \Z[x] \right)/I \To \Z[x,y]\ .\]
 Since the ring generators $x, y$ of $\Z[x,y]$ are in the image of $i$, it  is surjective.
 To prove injectivity, consider the graded Poincar\'e series of $\Z[u,v]$, where $u$ has degree $2$ and $v$ has degree $1$. It is 
 \[ P_{\Z[u,v]} (t) =  \frac{1}{(1-t)(1-t^2)}\]
If $x,y$ are given degree $1$, we have $P_{\mathcal{N}}(t) = 1+t$ and hence
\[ P_{\Z[u,v] \otimes \mathcal{N}}(t) = \frac{1+t}{(1-t)(1-t^2)} = \frac{1}{(1-t)^2}\]
which equals the Poincar\'e series of $\Z[x,y]$, since it has two generators in degree $1$. This proves that after tensoring with $\Q$, the map $i\otimes \Q$ is an isomorphism of graded $\Q$-vector spaces. Since a submodule of a free $\Z$-module is free, $(\ker i)\otimes  \Q = \ker (i \otimes \Q) =0$, and hence $i$ is injective.
 \end{proof}

 \begin{prop} \label{prop: mapxyxplusy} 
  Let $\tau \subset \R^2$ be a region in $\R^2$, and consider the map 
  \[ \phi: \R^2 \To \R^2\]
  which maps $(x,y)$ to $(xy, x+y)$. Then $ \ts^{\rec}_{(1,1)} (   \tau ) \leq (\ts^{\hom}( \phi \tau))^{\frac{2}{3}}$. 
  \end{prop}
  \begin{proof} 
  Let us write $\mathcal{N} = \Z \oplus x \Z$ and  
 set $\mathcal{M}= \Z[x,y]$, $\mathcal{P}=\Z[xy,x+y]$,   equipped with the rectangular and homogeneous  filtrations, respectively:
\[  \mathcal{M}_n  =   \bigoplus_{0\leq i , j< n} x^i y^j \Z \ , \qquad \ , \qquad 
  \mathcal{P}_n  =  \bigoplus_{\substack{0\leq i,j \\  i+j< n}} (xy)^i (x+y)^j \Z \ .  \nonumber 
 \]
By lemma  \ref{lem: isoZmodulesxyxplusy} and the fact that the map $i$ respects the filtrations by degree in $x$ and $y$, we find that
\begin{equation}\label{inproofclaim:MisPtensN}  \mathcal{M}_n \cong   \left( \mathcal{P}_n \otimes \mathcal{N} \right)  \oplus \mathcal{C}_n \  \quad \hbox{ and } 
\quad \mathcal{P}_n \otimes \mathcal{N} = \mathcal{M}_{n} \oplus \mathcal{K}_n
\end{equation}
where $\mathcal{C}_n$ (resp.   $\mathcal{K}_n$) is the free $\Z$-module of rank $n$ with generators $x^i y^{n-1}$ for $0\leq i\leq n-1$ (resp.   
 generated  by $x  (xy)^i (x+y)^j$ for $i+j =n-1$). Recall from \eqref{exponent:rectangularcase} and \eqref{exponent:hom} that  
\[ e^{\mathcal{M}}_n =  n^2(n-1) \sim  n^3  \quad , \quad   e^{\mathcal{P}}_n =  \frac{n(n-1)(n-2)}{3} \sim \frac{n^3}{3} \ . \] 
If we define $e_n^{\mathcal{N}}=0$ then using the convention \eqref{etensorcanonical} we have 
\[ e^{\mathcal{P}\otimes \mathcal{N}}_n = e^{\mathcal{P}}_n \, \rank \,\mathcal{N}_n \sim \frac{2}{3} n^3\ .\] 
Since the total degrees of the determinants of $\mathcal{C}_n,\mathcal{K}_n$ are quadratic in $n$, we have
\[ t_{\mathcal{C}_n} (\tau )^{\frac{1}{e^{\mathcal{M}}_n}} \rightarrow  1 \  \quad \ , \quad t_{\mathcal{K}_n} (\tau)^{\frac{1}{e^{\mathcal{M}}_n}} \rightarrow 1    \qquad \hbox{ as } n\rightarrow \infty\ .  \]
It follows from  proposition \ref{prop: upperboundtaudirectsum}  applied  to each direct sum in   \eqref{inproofclaim:MisPtensN} that 
\[    \ts_{\mathcal{M}}(\tau) \leq  \ts_{\mathcal{P} \otimes \mathcal{N}}(\tau)^{\frac{2}{3}} \quad \hbox{ and } \quad  \ts_{\mathcal{P} \otimes \mathcal{N}}(\tau)  \leq  \ts_{\mathcal{M}}(\tau)^{\frac{3}{2}}   \  \]  
 and hence $ \ts_{\mathcal{M}}(\tau) =  \ts_{\mathcal{P} \otimes \mathcal{N}}(\tau)^{\frac{2}{3}} $. 
  Now, by the corollary to theorem \ref{thm: tsuptensorproductUB}  and the asymptotics of the exponent functions given above, we deduce that 
 \[  \ts_{\mathcal{P}\otimes \mathcal{N}}(\tau)  \leq   \ \ts_{\mathcal{P}} (\tau ) \ .  \]
 We conclude that 
\[ \ts_{\mathcal{M}}(\tau) =  \ts_{\mathcal{P} \otimes \mathcal{N}}(\tau)^{\frac{2}{3}} \leq     \ts_{\mathcal{P}}(\tau)^{\frac{2}{3}} \ . \] 
Since $\mathcal{P}= \phi^* \Z[u,v]$ where $\phi(u,v) = (xy, x+y)$,  the statement follows from \eqref{eqn: functorial}. 
 \end{proof}


  \begin{example}  Let  $\tau = [0,a]\times [0,b]$ where $a,b>0$. Then  $\phi(\tau)$ is (comfortably) contained in both the  rectangle  $[0,ab] \times [0,a+b]$, or a triangle with vertices $(0,0), (0, a+b), (\sqrt{ab}(a+b)/2,a+b)$. We obtain
  \[ \ts^{\rec}_{(1,1)}(\tau)  = \frac{\sqrt{ab}}{{4}} \leq  \left(\ts^{\mathrm{hom}}(\phi \tau) \right)^{\frac{2}{3}}\leq  \min \left\{  \frac{(ab)^{\frac{1}{3}} (a+b)^{\frac{1}{3}}}{4^{2/3}}  ,  \frac{(ab)^{\frac{1}{6}} (a+b)^{\frac{2}{3}}}{(8e^2)^{1/3}} \right\} \ , \] 
  by  \eqref{homogeneoustransfinitemultiplicative}, \eqref{eqn: trecsquarebox}  and \eqref{eqn: transfinitetriangle} respectively. The above is 
 consistent with the arithmetic-geometric mean inequality. The slippage in the constant in the denominator comes from the fact that the image of $\phi(\tau)$ lies in the smaller region satisfying $y\geq 4x^2$ (once more by the AM-GM inequality).
  \end{example}
  
  \begin{example} \label{ex: regiontauepsilonviatensor}  The image $\phi(\tau_{\epsilon})$  is contained in a quadrilateral with vertices $(0,0)$, $(0,1)$, $(\epsilon,1+\epsilon)$, $(\epsilon, 2 \sqrt{\epsilon})$. By applying the linear transformation 
  $(x,y) \mapsto (x, y-x)$, which has determinant $1$,  the latter  is transformed to  a region  $R_{\epsilon}$ which  fits inside a rectangle $[0,\epsilon] \times [0,1]$. Therefore we conclude that 
  \[  \ts^{\rec}_{(1,1)} (\tau_{\epsilon}) \leq \left(\ts^{\hom}(  [0,\epsilon] \times [0,1])\right)^{\frac{2}{3}} =  \left(\frac{\sqrt{\epsilon}}{4}\right)^{\frac{2}{3}} = 
 \left( \frac{\epsilon}{16}\right)^{\frac{1}{3}}
  \ . \]
The same region is  contained in a triangle $T_{\epsilon}$ with vertices $(0,0)$, $(0,1)$ and $(\frac{\epsilon}{2\sqrt{\epsilon}-\epsilon},1)$ and so 
  \[  \ts^{\rec}_{(1,1)} (\tau_{\epsilon}) \leq  \ts^{\hom}(  T_{\epsilon})^{\frac{2}{3}} =  \left( \frac{\mathrm{vol}(T_{\epsilon})}{2 e^2} \right)^{\frac{1}{3}} =\left(  \frac{\epsilon}{ 4 e^2 (2 \sqrt{\epsilon}-\epsilon) }\right)^{\frac{1}{3}}   \ , \]
  which is an improvement  on the previous bound in the range  $\epsilon>0.1042$.
  Nevertheless, these bounds can be quite crude for general values of $\epsilon$.
   The region $R_{\epsilon}$ constructed  above, which is contained in a quadrilateral with vertices $(0,0)$, $(0,1)$, $(\epsilon,1)$, $(\epsilon, 1+\epsilon-2\sqrt{\epsilon})$ contains the rectangle 
  with corners $(0, 1+\epsilon-2\sqrt{\epsilon})$, $(0,1)$, $(\epsilon,1)$,  $(\epsilon, 1+\epsilon-2\sqrt{\epsilon})$ and therefore we have a lower bound
  \[  \ts^{\hom}(R_{\epsilon})^{\frac{2}{3}} \geq   \left(\frac{ \epsilon (1-\sqrt{\epsilon})^2}{16} \right)^{\frac{1}{3}}  \ .   \
 \]
 This shows that  $ \ts^{\hom}(R_{\epsilon})   \sim    \frac{\sqrt{\epsilon}}{4}$ to leading order in $\epsilon$ as $\epsilon\rightarrow 0$.
    \end{example} 
  
  \subsection{Map to two rectangles (direct sums)}
  Consider the map 
  \begin{eqnarray}  (x,y)  \mapsto   ( xy, x ):   \C^2   \overset{\phi_x}{\To}  \C^2 \nonumber 
  \end{eqnarray} 
and similarly let $\phi_y: \C^2 \rightarrow \C^2$ be defined by $(x,y) \mapsto (xy,y )$.

  \begin{prop} For any bounded region $\tau \subset \C^2$, we have 
  \[ \ts^{\, \rec}(\tau)  \ \leq  \  \Big(\ts^{\, \hom} ( \phi_x(\tau) )\Big)^{\frac{1}{3}} \,  \Big(\ts^{\, \hom} ( \phi_y(\tau) )\Big)^{\frac{1}{3}}     \ . \] 
  \end{prop} 

  \begin{proof}
The $\Z$-module spanned by $x^i y^j$ for $0\leq i,j< n$ is a direct sum:
\[ \Z[x,y]^{\rec}_n = \Z[xy]_{n}  \oplus x\, \mathcal{N}_n^x \oplus y\, \mathcal{N}_n^y \qquad \hbox{ where } \quad 
 \mathcal{N}_n^x=    \Z[ xy,x]^{\hom}_{n-1} \quad \   , 
\quad  \mathcal{N}_n^y=  \Z[ xy,y]^{\hom}_{n-1} \ . \] 
 Recall that 
 \[  e^{\Z[x,y]^{\rec}}_n = n^2(n-1) \quad \ ,\quad  \  e^{\Z[xy]}_{n} = \frac{n(n-1)}{2} \quad  \ , \quad  e^{\Z[xy,x]^{\hom}}_{n-1}=  \frac{n(n-1)(n-2)}{3}\ \ , \]
 and similarly for $e^{\Z[y, xy]^{\hom}}$. 
 It follows that 
 \[ \lim_{n\rightarrow \infty}  \frac{e^{\Z[xy]}_{n}   }{   e^{\Z[x,y]^{\rec}}_n} =0   \quad  \ , \quad   \lim_{n\rightarrow \infty} \frac{e^{\Z[ xy,x ]^{\hom}}_{n-1}}{e^{\Z[x,y]^{\rec}}_n }  = \lim_{n\rightarrow \infty} \frac{e^{\Z[ xy,y ]^{\hom}}_{n-1}}{e^{\Z[x,y]^{\rec}}_n }   = \frac{1}{3} \ .  \]
 Since $\det |V_{x\mathcal{N}_n^x}(z_1,\ldots, z_N)| = |x_1\ldots x_N|  \det |V_{x\mathcal{N}_n^x}|$ and since $\tau$ is bounded, it follows from  an argument detailed previously that we may replace $x \mathcal{N}^x_n$  by $\mathcal{N}^x_n$, and similarly for $y \mathcal{N}^y_n$.
  We conclude that 
 \begin{eqnarray} \ts^{ \rec} (\tau)  &\leq&   \left(\ts_{\Z[xy]} (\tau)\right)^0  \left(\ts_{\Z[xy,x]^{\hom}} (\tau)\right)^{\frac{1}{3}}   \left(\ts_{\Z[xy,y]^{\hom}} (\tau)\right)^{\frac{1}{3}}    \nonumber  \\
 & = &   \left(\ts^{\hom} (\phi_x(\tau))\right)^{\frac{1}{3}}   \left(\ts^{\hom} (\phi_y(\tau))\right)^{\frac{1}{3}}      \nonumber 
 \end{eqnarray} 
 where the first line follows from corollary  \ref{cor: directsums} and the second by  \eqref{eqn: functorial}. 
  \end{proof} 
  
  \begin{example}
Applying the proposition to  $\tau_{\epsilon}$, we see that $\phi_x( \tau_{\epsilon})$ is contained in the rectangle $[0,\epsilon] \times [0,1]$ (more precisely in  the slightly smaller  region where $x\leq y$), and hence
  \[    \ts^{\rec}(\tau_{\epsilon})   \leq     \ts^{\hom} ( [0,1] \times [0,\epsilon])^{\frac{2}{3}} =   \ \left(\frac{\epsilon}{16}\right)^{\frac{1}{3}}\]
  which is exactly the same  as the first inequality in example \ref{ex: regiontauepsilonviatensor}.
    \end{example}
      
    \begin{rem}
    There  exist various generalisations of this proposition to higher dimensions. The typical shape of the regions arising from moduli space integrals \cite{DinnerParties} are
    of a similar form and typically contained in $\{(x_1,\ldots, x_r) \in [0,1]^r:  \prod_{i\in I} x_i \leq \epsilon_I\}$ for some  $\epsilon_I>0$  for every subset $I\subset \{1,\ldots, r\}$.

  \end{rem}

  \section{Small linear forms from the geometry of numbers} \label{sect: smalllinearforms}

  \subsection{Minkowski's theorem on linear forms}
      
  \begin{thm} Let $A\in M_{n\times n}(\R)$ be a matrix with real entries $(A)_{ij} = a_{ij}$. Then there exists a non-zero integer vector $\underline{c}=(c_1,\ldots, c_n)^T \in \mathbb{Z}^n \setminus \{0\}$ such that 
  \[  \left| \sum_{j=1}^n a_{ij} c_j \right| \leq \varepsilon_i  \, \hbox{ for all } 1\leq i \leq n \ ,  \] 
  provided that $\prod_{i=1}^n \varepsilon_i \geq |\det(A)|$. 
  \end{thm}
  In particular, if $\det(A)\neq 0$ there exists some  $1\leq i \leq n$  such that 
  \begin{equation} \label{MinkowskiCorollary}  0<  \left|  \sum_{j=1}^n a_{ij} c_j \right| \leq  |\det(A)|^{1/n}  \ 
  \end{equation}  
  which follows from the theorem on setting $\varepsilon_i =    |\det(A)|^{1/n} $ for all $i$. 
  By applying  this to a matrix $Q^{\sigma}_N$ of Mellin transforms we can deduce the existence of small linear forms in periods.
  One can also use Minkowski's theorem on short vectors for quadratic forms to obtain a similar bound, but the previous theorem has the advantage that it does not require $A$ to be symmetric nor positive definite.

  \subsection{Denominators}  \label{sect: Denominators}
  Let $(X,f)$ be a family of Mellin integrals as considered earlier and let us fix once and for all a basis $\xi_1,\ldots, \xi_m$ of periods as in lemma \ref{lem: periodsxi}. 
  Let    $Q^{\sigma}_N$ denote the matrix of Mellin integrals with respect to $\mathcal{M} = f^* \mathcal{N} \subset \mathcal{O}(V_f)$ free of rank $N$, where $\mathcal{N} \subset \Or(\mathbb{A}^r)$ is free and $f$-admissible.

  Let us choose invertible matrices $D^{\ell}_N, D^{r}_N \in \mathrm{GL}_{N}(\Q)$ such that 
  \begin{equation} \label{DDprime matrices}  \widetilde{Q}^{\sigma}_N = D^{\ell}_N Q^{\sigma}_N D^r_N  \quad \in \quad M_{N\times N} ( \Z[\xi_1,\ldots, \xi_m])
  \end{equation} 
  has entries which are \emph{integral} linear forms in the $\xi_i$.  Let 
  \[ \delta_N = |\det (D^{\ell}_N) \det(D^r_N)|\ .\]
  \begin{thm}  \label{thm: smalllinearform} If   $N$ is sufficiently large and $\det Q^{\sigma}_N \neq 0$,  then  there exists a non-zero linear form in the $\xi_1,\ldots, \xi_m$, which is a rational linear combination of Mellin integrals $I(a_1,\ldots, a_n)$,  such that 
  \[  0 <  \left| \sum_{i=1}^m  n_i \xi_i \right| \leq  \left( t_{\mathcal{N}}(f  \sigma )^2 \, \delta_N\right)^{\frac{1}{N}}   
\]   
where $n_i \in \Z$ and $t_{\mathcal{N}}(f \sigma)$ was defined in \eqref{eqn: tNdef}. 
  \end{thm}
  \begin{proof} Apply \eqref{MinkowskiCorollary} to $\widetilde{Q}^{\sigma}_N$ to deduce the existence of a non-zero integer   vector $0\neq \underline{c}=(c_1,\ldots, c_N)^T\in \Z^N$  and an index $1 \leq i \leq N$ such that 
  \[ 0 < \left|  (\widetilde{Q}^{\sigma}_N \underline{c})_i \right| <   \left| \det (\widetilde{Q}^{\sigma}_N) \right|^{\frac{1}{N}}  =  |\det(Q^{\sigma}_N)|^{\frac{1}{N}} \delta_N^{\frac{1}{N}} \] 
  and apply  corollary  \ref{cor: detQNBoundTr}  which states that $ |\det(Q^{\sigma}_N)|^{\frac{1}{N}}< t_{\mathcal{M}}(\sigma)^{\frac{2}{N}}$ for $N$ sufficiently large. By \eqref{eqn: tmorphism}, 
  we have $t_{\mathcal{M}}(\sigma) =  t_{\mathcal{N}}(f \sigma )$. 
   Finally, use the fact that the entries of $\widetilde{Q}^{\sigma}_N$ are integer linear combinations of $\xi_1,\ldots, \xi_m$ to write the $i^{\mathrm{th}}$ entry of  $\widetilde{Q}^{\sigma}_N \underline{c}$ as   $\sum_{i=1}^m  n_i \xi_i$ for some integers $n_i$.
  \end{proof}
  Under the additional assumptions (P1) -(P3),  $\det Q_N^{\sigma}\neq 0$ holds automatically, by lemma \ref{lem: posdefinite}. 
    Now consider the set up  of \S \ref{sect: setupexponentfunctions}, i.e.,  $\mathcal{N} = \varinjlim_n \mathcal{N}_n$ and $e^{\mathcal{N}}_n \in 
    \mathbb{N}$ is an exponent function. Define
  \[ \delta_{\mathcal{N},e} = \limsup_N  \delta_N^{\frac{1}{e^{\mathcal{N}}_n}} \]
  which we recall  depends on the choices of matrices $D^{\ell}_N, D^r_N$ as above. Henceforth we drop the dependence on $e$ from the subscript (as in remark \ref{rem: noe}).

  \subsection{Irrationality criterion} 
  \begin{cor}
Suppose that $\det\, Q_N^{\sigma} \neq 0$ for infinitely many values of $N$. If 
 \[    \ts^2_{\mathcal{N}}(f \sigma ) \,   \delta_{\mathcal{N}}   <1 \]  
  then $\dim_{\Q} (\Q \xi_1 \oplus \ldots \oplus \Q \xi_r) \geq 2$.  In particular, if $r=2$ and $\xi_1=1$, then we retrieve criterion \ref{intro: limitcrit}. \end{cor}

 The following quantity enables us to measure the quality of approximations  independently  of any choice of exponent function.
 \begin{defn}
 Let $\mathcal{N}$ be of rank $N$, and  $Q^{\sigma}_N$, $\delta_N$ be defined as   in \S \ref{sect: Denominators}.   
 Define the \emph{threshold}  to be:
 \begin{equation} \label{thetaNdef}  \vartheta_N = - \frac{  \log \det(Q^{\sigma}_N)}{\log \delta_N} \end{equation} 
Note that $\vartheta_N >1$ if and only if $\det(Q^{\sigma}_N) \delta_N <1$. 
 \end{defn} 
 

  \subsection{Denominators}
 We may always find $D^{\ell}_N, D^r_N$  satisfying \eqref{DDprime matrices} by taking, for example, $D^r_N$ to be the identity and $D^{\ell}_N$ to be the diagonal matrix whose $(i,i)$th  entry is the least common multiple of the denominators of the elements in the $i$th row of $Q^{\sigma}_N$. 
 This can be improved by taking $D^r_N$ to be the diagonal matrix whose $j^{\mathrm{th}}$ entry is the  inverse of the greatest common divisor of the coefficients of all entries in the $j^{\mathrm{th}}$ column of the matrix $D^{\ell}_N Q^{\sigma}_N \in M_{N\times N}(\Z[\xi_i])$, which has integral entries. 
 
 In the examples, we find that certain rows and columns of $Q^{\sigma}_N$  satisfy congruences modulo primes, which leads to further reductions in the prime factors dividing $\delta_N$. 
 The question of finding optimal denominator bounds is an interesting one which will not be pursued further here.

  \begin{example}
  Consider  the example of \S\ref{intro: exampleintuitive}.  The integer $N$ implicitly depends on $n$, which we omit from the notation. Let $D_N^{\ell}$ be the diagonal matrix whose entry  in the  $i^{\mathrm{th}}$ row is the least common multiple of the denominators of the elements in row $i$ of the matrix $Q^{\sigma}$, and let $D_N^r$ be the identity.

    \begin{lem} \label{lem: HardDenomBound} The determinant of $D^{\ell}_N$ satisfies
    \[\det(D^{\ell}_N) \sim \exp \left( w \frac{2r+1}{r+1} n^{r+1} \right)   \qquad \hbox{ as } n \rightarrow \infty\  \]
     \end{lem} 
     
     \begin{proof}
  The rows of $Q^{\sigma}$ are indexed by tuples $(n_1,\ldots, n_r)$ where $0\leq n_1,\ldots, n_r \leq n$. The least common multiple of the  denominators of all elements in that row is: 
  \[ (\mathrm{lcm} \{1,\ldots, m\})^w \quad \hbox{ where }  m = \max \{ n+n_1,n+n_2,\ldots, n+n_r\} \]
 The determinant of $D^{\ell}_N$ is the product over all row-indices. Since $\mathrm{lcm}\{1,\ldots, M\} \sim e^M$, 
 \[ \log D^{\ell}_N \sim \sum_{0\leq n_1,\ldots, n_r \leq n}   w\, \max \{ n+n_1,n+n_2,\ldots, n+n_r\} \sim  w  \frac{2r+1}{r+1} n^{r+1}\ . \]
  To verify the latter asymptotic, we check that
  \[ \int_{[n,2n]^r}  \max \{x_1,\ldots, x_r\} \, dx_1 \ldots dx_r =  r \int_{n}^{2n} x_r dx_r   \left(\prod_{i=1}^{r-1} \int_{n}^{x_r} dx_i \right)  =  r \int_{n}^{2n} x_r  (x_r-n)^{r-1} dx_r  
  =
     \frac{2r+1}{r+1} n^{r+1} \ . \]
  \end{proof}
 The rectangular exponent function \eqref{exponent:rectangularcase} has the asymptotic $e^{\mathcal{M}}_n \sim \frac{r}{2} n^{r+1}$, and hence 
  \begin{equation} \label{deltaboundforintroexample} \delta_{\mathcal{N}}  \leq  \exp \left( \frac{2w}{r} \frac{2r+1}{r+1}  \right) \  .\end{equation}
  \end{example}
 \begin{rem}
  A referee asked the excellent question of whether the matrices $D_n$ can be designed to make use of the cancellation of primes in the method of Rhin-Viola. I do not know the answer to this question, but the numerical experiments to be discussed later clearly demonstrate a large degree of `prime cancellation' of possibly a different nature. It would be very interesting to investigate this further. 
 \end{rem} 
 
  \subsection{Irrationality criterion via the determinant}
  The irrationality  criterion above is effective in the sense that it guarantees the existence of small linear forms in the periods $\xi_1,\ldots, \xi_m$. 
  Another approach is to note that the  determinant  $\det( D^{\ell}_N Q_N^{\sigma} D_N^r) \in \Z[\xi_1,\ldots, \xi_m]$ is a \emph{polynomial form} in $\xi_1,\ldots, \xi_m$  with integer coefficients. 
  This argument generalises the one-dimensional result of Zudilin \cite{ZudilinDet}. In this context, one can  simply replace $\delta_N=  \det( D^{\ell}_N D_N^r)$ with any $d_N$ such that  $d_N \det(Q_N^{\sigma}) \in \Z[\xi_1,\ldots, \xi_m]$.
  
  \begin{lem} Let $Q^{\sigma}_N \in M_{N\times N}(\Q[\xi_1,\ldots, \xi_m])$ and let $0<d_N  \in \Z$ be such that  $d_N \det(Q_N^{\sigma}) \in \Z[\xi_1,\ldots, \xi_m]$.
If $d_N |\det(Q_N^{\sigma})|\rightarrow 0$ as $N\rightarrow \infty$ and $\det(Q_N^{\sigma}) \neq 0$ for infinitely many $N$, then 
 $\dim_{\Q} \Q\langle \xi_1,\ldots, \xi_m\rangle >1$.
  \end{lem} 
  
  The proof is obvious, and implies criterion \ref{intro:crit} when combined with earlier results.  In our computations \S\ref{sect:MainExample} we write down $d_N$ as a (typically  good)  approximation to  $\delta_N$  for ease of computation.

  \begin{rem} In our examples the matrices $Q^{\sigma}_N$  and their determinants are  highly structured  owing to the presence  of large symmetry groups. It would be interesting to exploit this structure systematically. 
  \end{rem}

 \section{Comments, variants, generalisations} \label{sect: Comments}
 
 \subsection{Beyond positivity}  To establish the non-vanishing of the determinant of $Q_N^{\sigma}$ we made some assumptions  to ensure that it is positive-definite. However, the assumption $(P1)$, for example,  can sometimes be dropped. 
  For instance, theorem 5.9  in \cite{LimitsTransfinite} gives conditions for the homogeneous transfinite diameter to be a determinant of integrals similar to the ones considered here.

 The holonomic structure can also be exploited. 
Consider  the case when there is a single parameter $r=1$. Then the matrix $Q_N^{\sigma}$ is a Hankel matrix 
\[ H_N(\underline{a})= \begin{pmatrix}  a_0 & a_1 & \ldots &a_N  \\ 
  a_1 & a_2 & \ldots &a_{N+1} \\
\vdots & \vdots &  \ddots & \vdots \\
a_N & a_{N+1} &  \ldots & a_{2N}
\end{pmatrix} \]
where $\underline{a}= (a_n)_n$ with $a_n = \int_{\sigma} f^n \omega$. A classical theorem of Kronecker's implies that the determinant of $H_N$  vanishes for all $N$ sufficiently large  if and only if 
the generating function $\sum a_k x^{-k-1}$ is a rational function of $x$.  See \cite[\S2]{HankelSolvability2016}, for an overview. 
 The latter can be ruled out on an \emph{ad-hoc} basis from the asymptotics of the integrals in question, or their holonomic recurrence equations.\footnote{Incidentally, the fact that the Hankel determinants $\det(H_N)$ are invariant under various operations on the sequence $\underline{a}=(a_n)$ such as the binomial transform may suggest  approaches to improve the denominator estimates}

In the case when there are more parameters, the matrices $Q_N^{\sigma}$ are `multivariable Hankel forms' in the sense of  \cite{PowerHankel}, who generalised some of Kronecker's results to this setting.  These results presumably point to generalisations of the above  results with weaker assumptions on the entries of  $Q_N^{\sigma}$.

 \subsection{Cohomological determinants} 
 Suppose, as in \S\ref{setup},   that $Y\subset X$ is simple normal crossing and defined over $\Q$.   For any  integers $n_i\geq0$,  the  integrand of $I(n_1,\ldots, n_r)$ defines a de Rham class
\[   [  f_1^{n_1} \ldots f_r^{n_r} \omega ] \quad  \in \quad  M_{\dR}  = H_{\dR}^d(X, Y) \]
which is a finite-dimensional $\Q$-vector space. In particular, any set $\ff$ of linearly independent functions as in \eqref{introff} produces a set of cohomology classes in $M_{\dR}$, and we may  
 consider the symmetric matrix:
  \begin{equation}  Q^{\dR}_N =   \left(  [  \ff_i  \ff_j \,   \omega]  \right)_{1\leq i,j\leq N} \end{equation}
 with entries  in $M_{\dR}$. They can be paired with any class $[\sigma] \in M_B^{\vee}:= H_d(X(\C),Y(\C)) $  to retrieve the matrix 
$  Q^{\sigma}_N=     \langle \sigma , Q^{\dR}_N \rangle $   considered  earlier. The family of determinants 
  \[  \det Q^{\dR}_N  \quad \in \quad  \mathrm{Sym}^{\bullet} ( M_{\dR}) \]
  form a very interesting algebraic object associated to the data of the Mellin integrals \eqref{introIs}. Note that the problem of  estimating  the denominator of $\det Q^{\sigma}_N$, relative to a fixed choice of $\Q$-basis for $M_{\dR}$, only depends    on     the de Rham matrix $Q^{\dR}_N$. 
  
\subsubsection{$p$-adic and other periods} \label{subsect:padic}
  
  We can deduce many other interesting numbers from $\det Q^{\dR}_N$. For instance,  we  can pair it  with any other relative homology class $[\sigma'] \in M_B^{\vee}$, and  we may also take $p$-adic periods by considering the pairing $\langle \mathrm{Sym}^{N} v, F_p \, \det Q^{\dR}_N\rangle$  where $F_p$ is a crystalline Frobenius  for $p$  a finite prime of good reduction and $v\in M_{dR}^{\vee}$.  We can even consider its single-valued period (where $p=\infty$, and $F_{\infty}$ is the real Frobenius).   This  suggests a notion of $p$-adic transfinite diameter which is related to the question of estimating the $p$-adic valuation of the denominator  of $\det Q^{\dR}_N$.  See remark \ref{rem: zeta2padic}.

 %
 %
 %
 %
 
  \section{Main example: $\zeta(2)$ and $\mathcal{M}_{0,5}$}  \label{sect:MainExample}
  In this section we study in some detail the family of linear forms in $1,\zeta(2)$ arising from the moduli space of curves $\mathcal{M}_{0,5}$ of genus $0$ with $5$ marked points. Although the applications to the  irrationality of $\zeta(2)$ are not of particular interest, the underlying geometry and Mellin transforms are very well understood and so this provides fertile ground for exploring the 
 constructions introduced in this paper.
  
  \subsection{A family of Mellin transforms}
  Consider the family of integrals
  \begin{equation} \label{generalM05integrals}  I(s_1,\ldots, s_5) = \int_{[0,1]^2}  u^{s_1}_1 \ldots u^{s_5}_5 \,  \omega  \qquad \hbox{ for } s_1,\ldots, s_5 \geq 0, 
  \end{equation} 
  where    $\omega = \frac{dx dy}{1-xy}$, and the $u_i$ are   dihedral `coordinates' \cite{BrENS} given by: 
  \begin{equation}\label{uiasxy}  u_1 = x \  , \ u_2 = \frac{1-x}{1-xy} \ , \ u_3 = \frac{1-y}{1-xy} \ , \ u_4 = y \ , u_5 = 1- xy \ .  \end{equation}
     It was shown by Dixon in 1905  that the 
   integrals $I(s_1,\ldots, s_5)$ are invariant under dihedral symmetries:
      \[ I(s_1,\ldots, s_5) = I(s_{\rho(1)}, \ldots, s_{\rho(5)} )\]
      where $\rho \in D_{10}$ is  the  group of symmetries of a pentagon with edges labelled $1,\ldots, 5$ in order.
  One may show by elementary or geometric means that 
  $I(h ,i,j,k,\ell ) \in \Z \zeta(2) + \Q$
    and more precisely   \cite{RVzeta2}  that 
    \begin{equation} \label{IM05denominatorbounds}   d_{m_1} d_{m_2} I(h,i,j,k,\ell) \in \Z + \Z \zeta(2) \end{equation}
    where $m_1, m_2$ are the two successive maxima of  the entries of the pole vector $\underline{p}=(p_1,\ldots, p_5)$ where 
       \begin{equation}  \label{p1-5vectors}   p_1  = i+j  -\ell   \quad  , \quad  p_2 = j+k -h  \quad  , \quad   p_3 = k+\ell - i   \quad  , \quad   
      p_4 = \ell +h -j \quad  , \quad p_5 = h+i -k \ .  
   \end{equation}
  Note that \eqref{IM05denominatorbounds} is not optimal for certain values of the arguments $(h,i,j,k,\ell)$. This is exploited in \cite{RVzeta2} as a consequence of several things including Gauss' hypergeometric transformation.      \subsubsection{Geometric interpretation}
See \cite{BrENS, DinnerParties} for further background.     
    The $u_1,\ldots, u_5$ satisfy the following dihedrally-symmetric  set of algebraic relations 
    \begin{equation} \label{urelations}   u_1 u_4 + u_5 =1 \ ,  \  u_2 u_5 + u_1 =1   \ ,  \ u_3 u_1 + u_2 =1    \ , \ 
        u_4 u_2 + u_3 =1 \ ,   \ u_5 u_3 + u_4 =1       \end{equation}  
    which can be interpreted combinatorially via sets of crossing chords in a pentagon \cite{BrENS}.  The ideal $I$ of $\Z[u_1,\ldots, u_5]$ generated by \eqref{urelations}  defines a smooth affine scheme 
    $\mathcal{M}_{0,5}^{\delta} = \mathrm{Spec} \left( \Z[u_1,\ldots, u_5]/I\right) $ which contains  the moduli space $\mathcal{M}_{0,5}$ of curves of genus $0$ with 5 ordered marked points 
    \[\mathcal{M}_{0,5}  \cong  \mathbb{A}^2  \ \setminus  \  V\left(xy(1-x)(1-y)(1-xy)\right) \   \overset{\eqref{uiasxy}}{\hookrightarrow}  \ \mathcal{M}_{0,5}^{\delta} \ ,   \] 
    where $\mathbb{A}^2 = \mathrm{Spec} \, \Q[x,y]$.   
     The vanishing locus  
   $Y= V(u_1\ldots u_5) \subset \mathcal{M}_{0,5}^{\delta}$
     is a simple normal crossing divisor with five irreducible components $Y_i: \{u_i=0\}$ and one proves that 
    \[ M= H^2 (\mathcal{M}^{\delta}_{0,5}, Y)  \cong  \Q(0) \oplus \Q(-2)\]
    which has two periods $1, \zeta(2)$. The domain 
    \[ \sigma^{\delta} = \{ u_1,\ldots, u_5 \geq 0 \} \subset \mathcal{M}_{0,5}^{\delta} (\R)\]
    defines a relative homology class $[\sigma^{\delta} ] \in M_B^{\vee}$ whose interior is homeomorphic to $(0,1)^2$ via \eqref{uiasxy}. One shows that  $\omega \in \Omega^2(\mathcal{M}^{\delta}_{0,5})$ has no zeros or poles along $Y$.
    
    Thus, in our initial set-up, consider  $X=\mathcal{M}_{0,5}^{\delta}$ equipped with the function
    \[ \underline{u} :  \mathcal{M}_{0,5}^{\delta}\To \mathbb{A}^5\ \]
whose components are $u_1,\ldots, u_5$.        One may show that
    \begin{equation}  \label{geomM05atinfinity}
     \mathcal{M}_{0,5} = \mathcal{M}_{0,5}^{\delta} \setminus Y \quad \hbox{ and }    \quad  \mathcal{M}_{0,5}^{\delta} = \overline{\mathcal{M}}_{0,5} \backslash E  \end{equation}
    where $\overline{\mathcal{M}}_{0,5}$ is the Deligne-Mumford-Knudsen compactification of $\mathcal{M}_{0,5}$ and $E$ is a divisor  `at infinity' with five irreducible components. 
     The parameters $p_1,\ldots, p_5$ may be interpreted as (one less than) the order of the poles of the differential form $u_1^{h}u_2^i u_3^j u_4^k u_5^{\ell} \omega$ along each component of $E$.

   In the remainder of this section we apply  the constructions in this paper to  families of integrals obtained by restricting 
    the most general family \eqref{generalM05integrals} to  linear subpaces of the parameter space $s_1,\ldots, s_5 \in \N^5$.
        
    \subsection{One parameter family} We first consider the family, considered by Ap\'ery and Beukers  \cite{Apery, Beukers} where all parameters are equal  $n_1=\ldots=n_5=n$.  
    \[ I(n) = \int_{[0,1]^2}  f^n \omega\quad \quad \hbox{ where }\quad  f=  u_1u_2u_3u_4u_5= \frac{x(1-x)y(1-y)}{1-xy}  \ . \]
     It is associated to $(\mathcal{M}_{0,5}^{\delta}, f)$ where $f: \mathcal{M}_{0,5}^{\delta}\rightarrow \mathbb{A}^1$.  The anciliary image variety is $V_f = \mathbb{A}^1$.
     The function $f$ has a unique critical point in the unit square  given by $x_0=y_0 = - \frac{1}{2}  + \frac{\sqrt{5}}{2}$ where it takes the value 
     \begin{equation}\label{etadef}   \eta =  \frac{5\sqrt{5} -11 }{2} \sim 0.09017 \ , \end{equation} 
     and so we deduce that 
     \[  f([0,1]^2) \subset  [0,\eta] \subset V_f(\R) = \R\] 
 By \eqref{IM05denominatorbounds}, the denominator of $I(n)$ is $d_n^2$.        
   Already this family is sufficient to establish irrationality of $\zeta(2)$ using the classical irrationality criterion since $e^2 \eta \sim 0.6663 < 1$. For later comparisons, the corresponding threshold 
   would be the quantity $\theta_{\mathrm{classical}} = -\frac{\log \eta}{2} = 1.203\ldots$. 
    
    Now we compare with the 1-dimensional set up of this paper (which gives an equivalent criterion to that of \cite{ZudilinDet}). Applying the example of \S\ref{intro: exampleintuitive}, with $r=1$, $w=2$, and $f_1=f$, we obtain the 
  criterion   $\frac{1}{4} \eta e^3 =0.4527\ldots< 1$ using the formula   \eqref{eqn:supinterval} for the transfinite diameter  of an interval. The corresponding threshold \eqref{thetaNdef}
   is $\theta_{1} =  - \frac{\log(\eta/4)}{3}= 1.2641\ldots$, a significant improvement on $\theta_{\mathrm{classical}}$. 

    \subsection{Two-parameter family}  Consider the family of integrals   \eqref{generalM05integrals} where $s_2=s_4$ and $s_1=s_2=s_5$.
    Thus we set $\sigma = \sigma^{\delta} \subset \mathcal{M}_{0,5}^{\delta}(\R)$ and  
    \begin{equation} \label{M052paramchoices} f_1=  u_2u_4 = \frac{ y(1-x)}{1-xy}  \quad , \quad f_2 = u_1u_3u_5 = x(1-y)  \ . 
    \end{equation}
    These two functions define a map 
    $f=(f_1,f_2) : \mathcal{M}_{0,5}^{\delta} \rightarrow \mathbb{A}^2\ . $
    Since $f_1,f_2$ are algebraically independent, the anciliary image variety $V_f \cong  \mathbb{A}^2$.
      The corresponding period integrals take the form     \[  I(n_1,n_2) =  \int_{\sigma}   f_1^{n_1} f_2^{n_2} \omega =  \int_{[0,1]^2}   \frac{  x^{n_2} (1-x)^{n_1} y^{n_1} (1-y)^{n_2}}{ (1-xy)^{n_2+1} } \, dx dy  \ . \]
      We consider the modules  
      \begin{equation} \label{Mndefnzeta2} \mathcal{M}_n  = \bigoplus_{1\leq i,j\leq n-1}  f_1^i f_2^j  \Z  \  \subset  \ \Or (\mathcal{M}_{0,5}^{\delta})  \end{equation}
      which are obtained as  the pullback $\mathcal{M}_n = f^* (\Z[x_1,x_2]^{\mathrm{rec}}_{(1,1)})_n$ of the affine ring of $V_f = \mathbb{A}^2$,  equipped with the rectangular filtration \S\ref{sect: rectVdM}. In particular,
      \begin{equation} \label{rankexponentoftwoparamexample} \rank \, \mathcal{M}_n = n^2 \quad \hbox{ and } \quad e^{\mathcal{M}}_n = n^2(n-1)\ . 
      \end{equation}
      The matrix $Q^{\sigma}_n$ is an $(n^2 \times n^2)$ matrix whose  entries are the periods $I(n_1,n_2)$ where $0\leq n_1,n_2 \leq  2(n-1)$.  
      For $n=1$, we have $Q^{\sigma}_n= (1)$. 
       By an elementary computation, the image of $\sigma$  is given by 
  \begin{equation}   \label{tauregion2paramfamily} \tau = f(\sigma)= \{ u,v \in \R^2: 0\leq u \leq \left(\frac{v-1}{v+1} \right)^2  \ , \  0\leq v\leq 1 \ .\}   \end{equation} 
  It is contained in the region denoted by  $\tau_{\eta}$ in  \S\ref{sect: hyperbolaregion},  where $\eta$ was defined in \eqref{etadef}. 
    The vector $\underline{p}= (p_1,\ldots, p_5)$  of  \eqref{p1-5vectors}  associated to $I(n_1,n_2)$ satisfies
  $  \underline{p} =  ( n_1,n_1,n_2,n_2,n_2)$. 
    It follows that 
    \begin{equation} \label{DenomboundforIn1n2}  d^2_{\mathrm{max}\{n_1,n_2\}} I(n_1,n_2) \in \Z + \Z \zeta(2) \ .
    \end{equation} 
    \begin{example} For $n=2$,  $\mathcal{M}_2 $ has basis $(1,f_1,f_2,f_{1}f_2)$. The matrix $Q^{\sigma}_{2}$ is
\[   \begin{pmatrix}
\int_{\gamma} \omega & \int_{\gamma} f_1 \omega & \int_{\gamma} f_2 \omega & \int_{\gamma} f_1 f_2 \omega \\[6pt]
\int_{\gamma} f_1 \omega & \int_{\gamma} f^2_1 \omega & \int_{\gamma} f_1f_2 \omega & \int_{\gamma} f^2_1 f_2 \omega \\[6pt]
\int_{\gamma}f_2 \omega & \int_{\gamma} f_1 f_2\omega & \int_{\gamma} f^2_2 \omega & \int_{\gamma} f_1 f^2_2 \omega \\[6pt]
\int_{\gamma}f_1 f_2  \omega & \int_{\gamma} f^2_1f_2 \omega & \int_{\gamma} f_1f^2_2 \omega & \int_{\gamma} f^2_1 f^2_2 \omega 
\end{pmatrix}   = \begin{pmatrix}
\zeta(2)  &-1+ \zeta(2) & 2- \zeta(2) & 5-3  \zeta(2)  \\[6pt]
-1+  \zeta(2) &  - \frac{5}{4} + \zeta(2) & 5-3\zeta(2) &  \frac{33}{4}- 5  \zeta(2)  \\[6pt]
2- \zeta(2) &5-3\zeta(2) &   -\frac{3}{2}+  \zeta(2)  &  -  \frac{23}{2} + 7 \zeta(2)  \\[6pt]
5-3\zeta(2) &  \frac{33}{4}- 5  \zeta(2)  &-  \frac{23}{2} + 7 \zeta(2)    & - \frac{125}{4} +19 \zeta(2)
\end{pmatrix}   \ . \]
It resembles, but  is  \emph{not}  in fact a classical Hankel matrix. Its determinant is 
\[ \det Q^{\sigma}_{2} \  =  \ - 8 \zeta(2)^4 + 23 \zeta(2)^3 + \frac{11}{4} \zeta(2)^2 - \frac{851}{16} \zeta(2) + \frac{145}{4}  \quad \sim \quad   8.05 \times 10^{-6} \ . \]
 Here, and throughout this section, we let $d_n$ denote the denominator of $\det Q^{\sigma}_{n}$.  In this case, $d_1=16$. 
 In practice, we find that $d_n$ is very close to  $\delta_n$ defined earlier. 
  After clearing denominators, the linear form remains very small, around 0.000129. It factorises into two terms 
\begin{equation} \label{detFactorsf1f2} 16 \det Q^{\sigma}_2 = -  (8 \zeta(2)^2 - \zeta(2) -20) ( 16 \zeta(2)^2 -44 \zeta(2) +29) \ .\end{equation}
The exponent \eqref{rankexponentoftwoparamexample} in this case is $e_2 =4$ and we find that
\[  t_2= (\det Q^{\sigma}_2)^{\frac{1}{4}} \sim  0.0532 \quad , \quad    d^{1/e_2}_1=   16^{1/4}= 2\ . \]
  \end{example}

 \begin{example}
 For $n=3$, the module $\mathcal{M}_3$ has basis 
\[ (1,f_1,f_2,f_1^2,f_1f_2, f_2^2, f_1f_2^2,f_1^2f_2,f_1^2f_2^2)\]
The exponent $e_3 = 18$. The matrix $Q_3$ is a  $9 \times 9$ matrix whose entries are linear forms in $1, \zeta(2)$. Its determinant is 
\[\det Q^{\sigma}_3 = -2304 \, \zeta(2)^9 + 18160\,  \zeta(2)^8 + \ldots \ \in\  \Q[\zeta(2)]\] 
The numerical value of $\det Q^{\sigma}_3  \sim 3.76 \times 10^{-27}$ and its denominator is $d_3 = 2^{18} 3^{16}$.  After clearing denominators, this linear form is of order $4.24 \times 10^{-14}$. We find that 
\[  t_3= (\det Q^{\sigma}_3)^{\frac{1}{18}} \sim  0.0340 \quad , \quad    d_3^{1/18}= 2 (3)^{16/18} \sim   5.3106\]
  \end{example} 
         Proceeding in this manner we compute: 
 \[ 
\begin{array}{c|cccc}
n  &N_n&  e_n & \det(Q^{\sigma}_n) & d_n   \\ \hline
 4 &    16  &  48 &   5.19 \times 10^{-75} &  2^{48} 3^{30}5^{20} \\
   5 &  25  & 100  &   6.29 \times 10^{-160} &  2^{90} 3^{48} 5^{48} 7^{24}\\
  6 &   37 &180  &  1.97 \times 10^{-292}  &  2^{147} 3^{98} 5^{71} 7^{58}  \\
  7 &   49 &294  &  4.68 \times 10^{-483}  &  2^{226} 3^{160} 5^{93} 7^{96} 11^{32}  \\
  8 &   64 &448  &   2.46 \times 10^{-742}  & 2^{324} 3^{234} 5^{119} 7^{128} 11^{74} 13^{36} \\
  9 &   81& 648   &  8.41 \times 10^{-1081}  &  2^{432} 3^{320} 5^{158} 7^{160} 11^{128} 13^{82} \\
  10 &  100 &  900     &        5.42 \times 10^{-1509}                         &    2^{555} 3^{392} 5^{199} 7^{196} 11^{186} 13^{140} 17^{44} \\
  11  &  121 & 1210     &    1.93      \times 10^{-2037}                        &    2^{706} 3^{467} 5^{240} 7^{233} 11^{240} 13^{206} 17^{98} 19^{48}  \\
  12 &  144 & 1584     &        1.10  \times 10^{-2676} &    2^{872} 3^{560} 5^{280} 7^{284} 11^{288} 13^{274} 17^{164} 19^{106}  \\
  13 & 169 & 2028 & 2.95\times 10^{-3437} &    2^{1071} 3^{665} 5^{319} 7^{335} 11^{338} 13^{336} 17^{240} 19^{176} 23^{56}  \\
  14 & 196  & 2548 &  1.08 \times 10^{-4329} &  2^{1308} 3^{759} 5^{442} 7^{383} 11^{392} 13^{392} 17^{326} 19^{256} 23^{122} \\
  15 &  225  &   3150    &     1.56 \times 10^{-5364}  &     2^{1552} 3^{942} 5^{576} 7^{440} 11^{450} 13^{450} 17^{414}  19^{348}  23^{200} \\
  16 &  256 &  3840   &  2.61 \times 10^{-6552} &  2^{1804} 3^{1144} 5^{716} 7^{506} 11^{509} 13^{512} 17^{498} 19^{444} 23^{288} 29^{68}
\end{array}
\]
These computations were carried out on Maple using standard procedures.




    \subsubsection{Experimental estimates}
The  data above may be expressed in the following table. The entries $t_N$  are bounded above by the  supremal transfinite diameter.
\[
\begin{array}{c|cccccc}
 n  &  \mathrm{rank}&  e_n   & t^2_n & \log d^{1/e_n}_n   &  t^2_n d_n^{1/e_n}  & \vartheta_n  \\ \hline
2   & 4 &  4  &  0.05327 & 0.693   &   0.1065  & 4.231  \\
 3 & 9& 18  &  0.03404 & 1.669   &  0.1808 & 2.025 \\
 4 & 16 & 48  &  0.02834 &  2.050  &  0.2202 &  1.738 \\
 5 & 25 &100  & 0.02559  &  2.391 &  0.2794  &  1.533 \\
 6 & 36 & 180  & 0.02396   & 2.426 &   0.2710  &   1.538 \\
 7 & 49 & 294  &  0.02287  &2.536  & 0.2890    &  1.490  \\
 8 & 64 & 448  & 0.02211  &2.661 &   0.3164  &     1.432 \\
 9 &  81 & 648 &   0.02154 & 2.676 &  0.3128 &    1.434 \\ 
 10 & 100 & 900  & 0.02109  & 2.718  &  0.3198   &  1.420  \\ 
 11 &121 & 1210 & 0.02074         & 2.781 &    0.3346 & 1.394   \\
 12 &144 &  1584 & 0.02045        &  2.773 &    0.3274  & 1.403 \\
 13 &169&  2028 & 0.02020        & 2.803  &   0.3332    &  1.392  \\
 14 &196&  2548 & 0.02000 & 2.827 & 0.3378    &  1.383  \\
 15 &225& 3150 & 0.01982  & 2.842 & 0.3399   &  1.380 \\
 16 &256 &  3840 & 0.01967     & 2.872  & 0.3476   &  1.368
 \end{array}
\]
For  the irrationality criteria to apply, the entries in the last column must tend to a limit  $>1$, and those of the second last column $<1$. 
The  bound of  lemma \ref{lem: HardDenomBound}, applied to \eqref{DenomboundforIn1n2} yields $\log d^{1/e_n}_n \leq \frac{10}{3}=3.\overset{.}{3}$ (see below for an improvement).   However, we find that the product of the least 
common multiple of the denominators  in each row  of  $Q_n$ is consistently and significantly larger than the true denominator of the determinant.
Notice that the values $t_n$ are decreasing which suggests that the squared supremal transfinite diameter is $<0.01967$ (column headed $t^2_n$).

\begin{prop}  \label{prop: ts11taubound}  We have the bound  
 $   (  \ts_{(1,1)}^{\rec}(\tau) )^2< 0.023 $ 
 \end{prop}
 \begin{proof}
    Consider the  image of the region $\tau$  \eqref{tauregion2paramfamily} under the map $\phi$. It   satisfies 
  \[ T_{\min} \subset \phi(\tau) \subset T_{\max}\]
  where $T_{\min}$, $T_{\max}$ are the triangles depicted in figure \ref{figure: boundedzeta2region}. Since 
  \[ \ts^{\hom}(T_{\min})   \  \leq \  \ts^{\hom}(\phi(\tau))  \ \leq	 \  \ts^{\hom}(T_{\max})\]
  we deduce from  \eqref{eqn: transfinitetriangle}  that 
  \[  0.017<  \left( \frac{\mathrm{vol}(T_{\min})}{2 \, e^2}  \right)^{\frac{2}{3}}  \  \leq \  \ts_{(1,1)}^{\hom}(\tau) \   \leq  \    \left( \frac{\mathrm{vol}(T_{\max})}{2 \, e^2}  \right)^{\frac{2}{3}} <0.023  \  \]
  The inequalities follow by computing the areas of $T_1,T_2$; the statement from   proposition  \ref{prop: mapxyxplusy}.
\end{proof} 
   Experiments suggest that the square of the transfinite diameter is slightly smaller than the average $0.02$ and significantly smaller than our upper bound 0.023. Indeed, from figure  \ref{figure: boundedzeta2region} we see that there is a considerable amount of daylight between the lower part of the black curve and the  outer triangle. 

  \begin{figure}[h]
{\includegraphics[height=5cm]{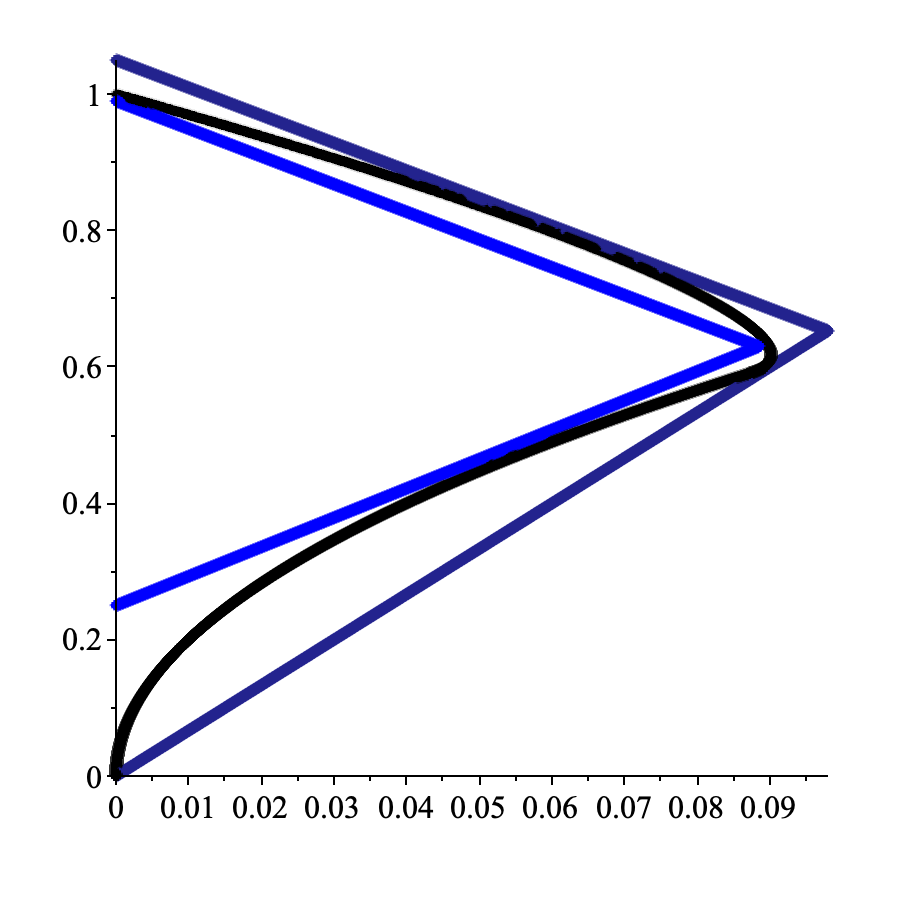}} 
\begin{caption}{The  region $\phi(\tau)$  is  bounded by the $y$-axis  and the black curve. The  top part of this curve is the image of 
$u(v+1)^2=(v-1)^2$ under the map $x=uv, y=u+v$; the  bottom is given by the parabola $y^2=4x$. 
The region $\phi(\tau)$ is visibly  contained in a large  blue triangle $T_{\max}$,  whose vertices are $(0,0)$, $(0,1.05)$, $(0.098,0.653)$, and  contains a smaller blue triangle $T_{\min}$ whose vertices are $(0,0.25)$, $(0,0.99)$, $(0.0885,0.63).$ 
}\label{figure: boundedzeta2region} 
\end{caption} 
\end{figure}

\subsubsection{Denominators}
A modest improvement of the denominator estimate may be obtained by changing basis of $\mathcal{M}_n$, which does not change its determinant.  One way to do this is to replace $f_1,f_2$ with 
\[ g_1 = 1-f_1  = u_3   \quad \hbox{ and } \qquad g_2 = 1+f_2 =u_1 +u_5   \ .\]
A monomial $g_1^a g_2^b$ is a linear combination of $u_3^a u_1^i u_5^j$ where $i+j=b$, whose  vector of poles  \eqref{p1-5vectors} is 
$\underline{p}=(a-j,a-i,j,b -a,i) $. Its two maxima are bounded above by
\[(m_1,m_2) = \begin{cases}   ( a, \max\{b,a-\textstyle{\frac{b}{2}}\})  \quad  \hbox{ if }  a\geq b,  \\   (b, \max\{a, b-a\})\quad   \hbox{ if }  b\geq a \end{cases}\]
 for all $i,j$. Define $L_n$ (resp. $R_n$) to be the diagonal matrix whose entry   corresponding to the  basis element $g_1^ag_2^b$ is given by 
 $d_{\max\{b,a-b/2\}+n}$ (resp. $d_{\max\{a,a-b\}+n}$). Then it follows from a short computation and  \eqref{IM05denominatorbounds}  that $L_n Q_n R_n$ has entries which are integral linear combinations of $1, \zeta(2)$. Using the fact that
 \[  \sum_{0\leq a,b<n} \max\{a+n,b-a+n\}  \sim \frac{19}{12} n^3  \quad , \quad  \sum_{0\leq a,b<n} \max\{b+n,a-\displaystyle{\frac{b}{2}}+n\}  \sim \frac{29}{18} n^3 \ ,    \]
 and that $e_n \sim n^3$  by \eqref{rankexponentoftwoparamexample}, we conclude that 
 \[  \limsup_N \left( \log (d_N^{\frac{1}{e_N}}) \right) \leq  \frac{19}{12}+\frac{29}{18}= \frac{115}{36} = 3.19\overset{.}{4} \ . \]
 It seems likely that one could reduce this further, since the  estimates used above  are not optimal.

\subsection{Variant: two copies}  
Let $\mathcal{M}_n \subset \Or(\mathcal{M}_{0,5}^{\delta})$ be as in \eqref{Mndefnzeta2} and now  consider the direct sum 
\[  \mathcal{M}_n \oplus u_1 \mathcal{M}_n\]
where  we recall from \eqref{uiasxy} that $u_1$ is the function $x$.  The exponent $e_n =2n^2(n-1)$ is defined to be exactly double the previous example \eqref{rankexponentoftwoparamexample}. 
We obtain:
\[
\begin{array}{c|cccccc}
 n  & \mathrm{rank} & e_n   & t^2_n & \log d^{1/e_n}_n  &  t^2_n d_n^{1/e_n} & \vartheta_n    \\ \hline
2    & 8 & 8  &  0.01312   &  2.022   &  0.0991 & 2.143  \\
 3  & 18 & 36  &  0.01625 & 2.554   & 0.2089  &   1.613 \\
 4 & 32& 96  &  0.01707 &   2.571 & 0.2233  &   1.583  \\
 5 & 50& 200  &  0.01738 & 2. 740 &   0.2691 &   1.479 \\
 6 & 72 & 360  &   0.01752 & 2.732  &  0.2690  &   1.480 \\
 7 & 98& 588 &    0.01759   & 2.802 &  0.2897 &   1.446 \\
 8 & 128 &896 &  0.01762 & 2.845 & 0.3032 &   1.420  \\
 9 & 162& 1296 & 0.01764 & 2.862 & 0.3088 & 1.411 
\end{array}
\]
The third column is bounded above by the square of the supremal transfinite diameter, by corollary \ref{cor: stationarytp}.
Notice that the entries in this column are now \emph{increasing} and seem to converge more quickly than the previous example.
They are  consistent with the upper bound    of proposition \ref{prop: ts11taubound}.

   \subsection{The five-parameter family}
We now consider the five-parameter family
\[ (u_1,\ldots, u_5) : \mathcal{M}_{0,5}^{\delta} \To \mathbb{A}^5 \ . \]
  The anciliary image variety is exactly the scheme $\mathcal{M}_{0,5}^{\delta}$ of dimension 2, embedded naturally in $\mathbb{A}^5$, and 
the region $\sigma= \sigma^{\delta}$ is  defined by $u_i\geq 0$.

We now take the  full affine ring $ \mathcal{M}= \mathcal{O}(\mathcal{M}_{0,5}^{\delta})$ viewed as a $\Z$-module.    
We may filter $\mathcal{M}$ by the maximum of the vector of poles \eqref{p1-5vectors}.

\begin{lem} We have the isomorphism of $\Z$-modules
\begin{equation} \label{MonomialBasisM05}  \mathcal{M}_n  = \Z \oplus  \bigoplus_{a+b\leq n , a>0}  \, \bigoplus_{i \!\!\!\!\!\mod 5} \,    u^a_{i} u^b_{i+1}  \Z \ .  
\end{equation}
\end{lem}
\begin{proof} We outline the main ideas. First check that $\mathcal{O}(\mathcal{M}_{0,5}^{\delta}) \subset \sum_{i,a,b} u_i^a u_{i+1}^b \Z$ by successively applying  the dihedral relations \eqref{urelations} to any monomial in the $u_i$'s to replace $u_{i-1} u_{i+1}$ with $1- u_i$. Since this decreases the total degree in the $u_i$,  this process terminates when it can be applied no more, namely to monomials which only involve $u_i, u_j$ where $i,j$ are consecutive modulo 5. To show independence of the generators on the right-hand side of \eqref{MonomialBasisM05}, compute the pole vector \eqref{p1-5vectors} of $u_1^h u_2^i$ which equals $(i,-h,-i,h,h+i)$. The divisor  $E$ at infinity \eqref{geomM05atinfinity} has 5 irreducible components $E_1,\ldots, E_5$ which are simple normal crossing and form a pentagon. 
We deduce  from the geometric interpretation of the pole vectors \cite{BrENS} that $u_i^{a} u_{i+1}^b$ has  poles of order $a,a+b,b$ along three consecutive such edges which are uniquely determined by the index $i$, and no pole along the other two. This is enough information to establish independence, which is left as an exercise.
\end{proof} 
The  Poincar\'e series for the associated graded module $\mathrm{gr} \,\mathcal{M}$ is given by 
\begin{equation} \label{PoincareSeriesgrM}  \frac{1+3t+t^2}{(1-t)^2} =1 + 5 t + 10 t^2 +15 t^3 +\ldots \ .
\end{equation} 
\begin{example}
$\mathcal{M}_1$ has dimension $6$ and has the basis $\{ 1, u_1,\ldots, u_5\}$. 
The corresponding matrix is \[Q^{\sigma}_{1} = \begin{pmatrix}  \int_{\sigma} \omega &     \int_{\sigma} u_1\omega &  \int_{\sigma} u_2\omega &  \int_{\sigma}u_3 \omega & \int_{\sigma} u_4\omega
& \int_{\sigma} u_5\omega\\[6pt]
\int_{\sigma} u_1 \omega &     \int_{\sigma} u^2_1\omega &  \int_{\sigma} u_1u_2\omega &  \int_{\sigma}u_1u_3 \omega & \int_{\sigma} u_1u_4\omega
& \int_{\sigma} u_1u_5\omega\\[6pt]
\int_{\sigma}  u_2 \omega &     \int_{\sigma} u_1u_2\omega &  \int_{\sigma} u^2_2\omega &  \int_{\sigma}u_2u_3 \omega & \int_{\sigma}u_2 u_4\omega
& \int_{\sigma} u_2u_5\omega\\[6pt]
\int_{\sigma} u_3 \omega &     \int_{\sigma} u_1u_3\omega &  \int_{\sigma} u_2u_3\omega &  \int_{\sigma}u^2_3 \omega & \int_{\sigma} u_3u_4\omega
& \int_{\sigma} u_3u_5\omega\\[6pt]
\int_{\sigma}  u_4\omega &     \int_{\sigma} u_1u_4\omega &  \int_{\sigma} u_2u_4\omega &  \int_{\sigma}u_3u_4 \omega & \int_{\sigma} u^2_4\omega
& \int_{\sigma} u_4u_5\omega\\[6pt]
\int_{\sigma}  u_5\omega &     \int_{\sigma} u_1u_5\omega &  \int_{\sigma} u_2u_5\omega &  \int_{\sigma}u_3u_5 \omega & \int_{\sigma} u_4u_5\omega
& \int_{\sigma} u^2_5\omega 
\end{pmatrix}\]
By performing the integrals we find that
\[  Q^{\sigma}_{1} = \begin{pmatrix} \zeta(2)  &    1 & 1  &  1  &1  & 1 \\[6pt]
1  &     \frac{3}{4} & \frac{1}{2}  &  -1+ \zeta(2)   &  -1+ \zeta(2)
& \frac{1}{2}  \\[6pt]
1 &    \frac{1}{2} &  \frac{3}{4} & \frac{1}{2} &   -1+ \zeta(2)
&  -1+ \zeta(2) \\[6pt]
1 &      -1+ \zeta(2)& \frac{1}{2} & \frac{3}{4} &\frac{1}{2}
&  -1+ \zeta(2)  \\[6pt]
1 &      -1+ \zeta(2) &    -1+ \zeta(2) &  \frac{1}{2} & \frac{3}{4}
& \frac{1}{2}  \\[6pt]
1 &  \frac{1}{2}&   -1+ \zeta(2) &   -1+ \zeta(2) & \frac{1}{2}
& \frac{3}{4} 
\end{pmatrix}\]
 Observe that the matrix $Q^{\sigma}_1$ has a large proportion of entries which are rational, i.e., for which the coefficient of $\zeta(2)$ vanishes. This arises 
whenever the corresponding pole vector \eqref{p1-5vectors} has a negative term (see \cite{DinnerParties} for a cohomological  interpretation). 
By dihedral symmetry of the coordinates $u_1,\ldots, u_5$ the matrix $Q^{\sigma}_{1}$ evidently satisfies $P^T Q^{\sigma}_{1} P$ for any permutation matrix $P$ which preserves the first column and acts as a dihedral symmetry on columns $2,\ldots, 6$.  Consequently, 
the determinant, viewed as  a function of $x= \zeta(2)$ has the curious factorization:
\[ \det( Q^{\sigma}_{1})  =\frac{1}{1024} (8x^2-x-20) (16 x^2 - 44x + 29)^2  \]
Notice that these factors already occurred as the determinant of a $4\times 4$ matrix  for the 2-parameter family for $\zeta(2)$ (see \eqref{detFactorsf1f2}), which involved  more complicated-looking linear forms in $\zeta(2)$. The factorisation of the determinant  of $Q^{\sigma}_n$ also arises for higher $n$; the precise structure is interesting and warrants further investigation. 
Substituting $x= \zeta(2)$ we find that 
$\det(Q^{\sigma}_{1}) \sim 1.059 \times 10^{-8}$.

\end{example} 

\begin{example}
The module $\mathcal{M}_2$ has the basis
\[ \m = (1,  u_1, u_2, u_3, u_4, u_5 ,u_1u_2,u_2u_3,u_3u_4,u_4u_5,u_5u_1,u_1^2,u_2^2,u_3^2,u_4^2,u_5^2 ) \ .\]
Then the resulting $16 \times 16$ matrix has a determinant of order $10^{-49}$ and denominator $2^{45} 3^{30}$.  Its determinant, as a polynomial in $\zeta(2)$ once again factorises as  a polynomial of degree $4$ and $6$ (the latter with multiplicity two).
\end{example}

\subsubsection{Supremal transfinite diameter} Recall the definition of $f_1,f_2$  \eqref{M052paramchoices}.
We may show that 
\begin{equation}  \mathcal{M}_n =  \Z \oplus \bigoplus_{i \in \Z/5\Z}  \rho^i \left(  u_1  R_n\right) \end{equation}
where $R_n= \Z[f_1f_2 , f_1+f_2]^{\hom}_n$ and $\rho$ is the cyclic symmetry $u_i \mapsto u_{i+1}$ where $i,i+1 \in \Z/5\Z$.   Let $R= \bigoplus_n R_n$.
Since the exponent function of $R_n$ is $\sim \frac{n^3}{3} $ by  \eqref{exponent:hom}  it is natural, in view of remark \ref{rem: exponentdirectsum}, to define the exponent of $\mathcal{M}_n$ to be  the total degree in the $u_i$ with respect to the basis \eqref{MonomialBasisM05}
\begin{equation}  e^{\mathcal{M}}_n =   \frac{5}{6} n(n+1)(2n+1) \ . \end{equation} 
The supremal transfinite diameter can then be estimated from corollary \ref{cor: directsums}. Thus, since 
$e_n^{\mathcal{M}}/e_n^{\mathcal{R}} \sim \frac{5}{3}/\frac{1}{3}= 5$ and $\mathcal{M}_n$ is essentially a direct sum of 5 copies of $R_n$,  we deduce that 
\[  \ts_{\mathcal{M}} (\sigma) \leq  \ts_{R} ( \sigma)=  \ts^{\hom} ( \phi(\tau))\]
where the second equality follows from  $R_n = \phi^*  \Z[f_1, f_2]^{\hom}_n$ and functoriality \eqref{eqn: tmorphism}, and
where $\phi(\tau)$ is the region discussed in proposition \ref{prop: ts11taubound}. We deduce from the upper bound that 
\begin{equation} \label{UBfortsofm05}  \ts_{\mathcal{M}} (\sigma)^2 < (0.023)^{3/2}= 0.003488\ldots  \end{equation} 
We expect that  this upper bound should be closer to $(0.02)^{\frac{3}{2}} \sim 0.0028$. 
\subsubsection{Denominator bounds}
Consider the basis of $\mathcal{M}_n$ given in \eqref{MonomialBasisM05}. A term in the matrix $Q^{\sigma}_n$ is equivalent, by dihedral symmetry, to an integral  with integrand $u_1^{a} u_2^{b} u_i^c u_{i+1}^d \omega$ where $i \in \Z/5\Z$ and $a+b, c+d\leq n$.  The pole vector can be computed for each value of $i$, and its maximum is attained when $i=1$, leading to a largest pole of order $a+b+c+d$. The second maximum is attained when $i=2$ (or $i=5$), in which case it is $a+c+d$ (or $b+c+d$). Thus we may bound the two successive maxima $m_2\leq m_1$, solely as functions of $a,b$, by 
\[ m_1 \leq  a+b +n\quad \ , \quad  m_2 \leq \max\{a,b\} +n \ .\]
An elementary computation yields
\[ \sum_{0\leq a, b\leq n} a+b+n \sim \frac{5}{6} n^3 \quad \hbox{ and } \quad  \sum_{0\leq a, b\leq n} \max\{a,b\}+n \sim \frac{3}{4} n^3\ . \] 
Since \eqref{MonomialBasisM05} is a direct sum over  5 such copies,  the relevant factor is $ 5(\frac{5}{6}+\frac{3}{4}) = \frac{95}{12}$ and  
in total that the denominator is (crudely) bounded by 
\begin{equation} \label{denboundM05} \lim_n \,  \exp \left( \frac{95}{12} n^3\right)^{\frac{1}{e_n}} = \exp \left( \frac{19}{4}\right) = \exp(4.75) \ ,\end{equation}
since $  e^{\mathcal{M}}_n  \sim \frac{5}{3}n^3$.
This bound is based on the least common multiple of the denominators of row entries. Experiments suggest that  the actual denominator is considerably smaller.

\subsubsection{Experimental estimates}
Numerical computations yield the following results:
\[
\begin{array}{c|cccccc}
 n  & \mathrm{rank} & e_n   & t_n & \log d^{1/e_n}_n   &  t_n d_n^{1/e_n} & \vartheta_n   \\ \hline
1    & 6 &  5  &  0.02541  &    2.843  &  0.436  &  1.29 \\
 2 & 16&  25 &  0.01095 & 2.638 &  0.153  &     1.711 \\
 3 & 31 & 70   &  0.00724   & 2.894  & 0.130  &  1.703 \\
 4 &  51& 150    &  0.00563   &   3.450     & 0.177  &  1.501 \\ 
 5 &  76& 275    &      0.00475            &   3.536     & 0.163  &   1.513 \\ 
 6 &  106 & 455  &         0.00421         &      3.703   &  0.171 &  1.477  \\ 
 7 &  141 & 700  &     0.00384    &    3.884      &  0.186 &    1.432 \\ 
 8 &  181 & 1020  &      0.00357            &  3.926     & 0.181  &  1.435 \\ 
 9 &    226 & 1425&          0.00337        &     3.995  &    0.183 &  1.425 \\ 
 10 &   276 & 1925   &        0.00321          &     4.068    &  0.188 & 1.411 \\ 
 11 &    331 & 2539  &       0.00309           &      4.037   &  0.175 & 1.432
  \end{array}
\]
Note that because of the large size of the matrices, we could not compute the  denominators for large $n$  exactly, but estimated them. The data is entirely consistent with the bounds established above.
The final two columns need to be $<1$ and $>1$ respectively for an irrationality proof. 
It is not clear what to expect of the limit of the  thresholds  $\vartheta_n$, but they seem to compare  very favorably with $\vartheta_{\mathrm{classical}} = 1.203..$ for the classical irrationality criterion, and the threshold  $\vartheta_1=1.264..$ for the 1-dimensional determinant method (which is already  a significant improvement on the classical one).

\subsubsection{Closing remarks}
We did not exploit the dihedral symetries in the above analysis. It would be interesting to perform similar  computations for the   supremal transfinite diameter for the much smaller module $\mathcal{M}= \Or(\mathcal{M}_{0,5}^{\delta})^{D_{10}}$ consisting of  dihedrally-invariant functions.

\begin{rem} \label{rem: zeta2padic} 
The $p$-adic zeta value $\zeta_p(2)$ vanishes, and hence the $p$-adic version  $Q^{p}_n$  of the matrices $Q^{\dR}_n$  (see  \S\ref{subsect:padic}) have only rational entries. The denominator of their determinant is very slightly smaller than the one for $Q^{\sigma}_n$ but negligeably so, and therefore one expects that in the limit, the denominator can be computed from the $p$-adic valuation of the $p$-adic realisation of $Q^{\dR}_n$.
Thus, we  expect that 
\[  \delta_n^{1/e_n} \sim  \prod_p \left| \det Q^p_n \right|^{1/e_n}_p 
\]
which suggests an interpretation of criterion \ref{intro: limitcrit} in  the form: 
\begin{equation} \limsup_n   \Big( \left| \det Q_n^{\sigma} \right|^2  \prod_p   \left| \det Q^p_n \right|_p\Big)^{\frac{1}{e_n}}  <1\end{equation} 

\end{rem}

   \appendix 
   \section{Vandermonde determinants for tensor products of modules}  \label{sect: Appendix} In this appendix we recall  a formula   \cite{BrVdM}  for the generalised Vandermonde determinant  of a tensor product of finite-rank free modules.    
  \subsubsection{Amalgams of  matrices}
    Let $m,n\geq 1$ and consider a  matrix with $mn$ rows and $m$ columns:
\begin{equation} \label{matrixAdef} A=   \begin{pmatrix}  a_{11} & a_{12} &  \ldots & a_{1m} \\ 
a_{21} & a_{22} &  \ldots & a_{2m} \\ 
 \vdots    &   \vdots&  & \vdots \\
a_{mn\, 1 } & a_{mn \, 2} &  \ldots & a_{mn \, m} \\ 
 \end{pmatrix}  = 
 \begin{pmatrix}  \underline{a}_1 \\ 
 \underline{a}_2 \\ 
      \vdots \\
\underline{a}_{mn} \\ 
 \end{pmatrix}    
 \end{equation} 
 with entries $a_{i  j}$ in a commutative ring $R$, and where $\underline{a}_i = (a_{i 1} \ldots a_{i  m} ) \in R^m$ denotes the $i^{\mathrm{th}}$ row vector.
 Likewise, consider the matrix with  $mn$ rows and $n$ columns 
\begin{equation} \label{matrixBdef}    B=   \begin{pmatrix}  b_{11} & b_{12} &  \ldots & b_{1 n } \\ 
b_{21} & b_{22} &  \ldots & b_{2 n } \\ 
 \vdots    &  \vdots &  & \vdots \\
b_{mn\, 1 } & b_{mn \, 2} &  \ldots & b_{mn \,  n } \\ 
 \end{pmatrix} =  \begin{pmatrix}   \underline{b}_1 \\ 
 \underline{b}_2 \\ 
      \vdots \\
\underline{b}_{mn} \\ 
 \end{pmatrix}     \end{equation}
 with entries $b_{ij}\in R$, and  $\underline{b}_i = (b_{i  1} \ldots b_{i  n} ) \in R^n$ denotes the $i^{\mathrm{th}}$ row.
  
  \begin{defn} Define their  amalgamated matrix product:
 \[  
 A \star B =     \begin{pmatrix} \underline{a}_1 \otimes   \underline{b}_1 \\ 
  \underline{a}_2 \otimes \underline{b}_2 \\ 
      \vdots \\
 \underline{a}_{mn} \otimes  \underline{b}_{mn} \\ 
 \end{pmatrix}  \]
 It is a square matrix with  $mn$ rows and columns, and entries in $R$.  
 \end{defn} 
 Our convention for the Kronecker tensor product of two row vectors is:
 \[  (a_1  ,  \ldots ,  a_m) \otimes (b_1,   \ldots  , b_n)  =   
   ( \underbrace{a_1 b_1  ,  a_2 b_1  ,  \ldots ,   a_m b_1}_{\underline{a} \otimes b_1} , \underbrace{a_1 b_2 , a_2 b_2 , \ldots, a_m b_2}_{\underline{a} \otimes b_2}  , \ldots ,   \underbrace{a_1 b_n , a_2 b_n , \ldots, a_m b_n}_{\underline{a} \otimes b_n}   )  \]
 where  we write commas between entries  for  legibility and $\underline{a} = (a_1,\ldots, a_m)$.  This convention applies to every row in the definition of $A \star B$.  The  Kronecker  product of two matrices is a special case.
 
 \begin{example} Let $m=3$ and $n=2$, and consider the matrices
\[ A= \begin{pmatrix}
a_{11}   &   a_{12}   & a_{13}    \\ 
a_{21}   &   a_{22}    & a_{23}  \\
a_{31}  &   a_{32}      &a_{33} \\
a_{41}   &   a_{42} & a_{43} \\
a_{51}   &   a_{52} & a_{53} \\
a_{61}   &   a_{62} & a_{63}     
\end{pmatrix} 
\\  \quad , \quad B= \begin{pmatrix}
b_{11}  &    b_{12}    \\ 
b_{21}  &    b_{22}     \\
b_{31}  &     b_{32}    \\
b_{41}  &     b_{42}    \\
b_{51}  &     b_{52}    \\
b_{61}  &     b_{62}    
\end{pmatrix} 
\]
Then 
\[ A\star B = 
\begin{pmatrix}
a_{11}b_{11}  &   a_{12} b_{11}    & a_{13} b_{11}   & a_{11}b_{12}  &   a_{12} b_{12} &  a_{13} b_{12}  \\ 
a_{21}b_{21}  &   a_{22} b_{21}   & a_{23} b_{21}  & a_{21}b_{22}  &   a_{22} b_{22}   &  a_{23} b_{22}  \\
a_{31}b_{31}  &   a_{32} b_{31}   & a_{33} b_{31}  & a_{31}b_{32}  &   a_{32} b_{32} &   a_{33} b_{32}   \\
a_{41}b_{41}  &   a_{42} b_{41}    & a_{43} b_{41} & a_{41}b_{42}  &   a_{42} b_{42}  & a_{43} b_{42}   \\
a_{51}b_{51}  &   a_{52} b_{51}    & a_{53} b_{51} & a_{51}b_{52}  &   a_{52} b_{52}  & a_{53} b_{52}   \\
a_{61}b_{61}  &   a_{62} b_{61}    & a_{63} b_{61} & a_{61}b_{62}  &   a_{62} b_{62}  & a_{63} b_{62}   \\
\end{pmatrix} 
 \]
In particular, if 
\[ A= \begin{pmatrix}
1  &   x_1   & x_1^2  \\ 
1   &   x_2  & x_2^2  \\
\vdots  &   \vdots      &\vdots \\
1   &   x_6 & x_6^2     
\end{pmatrix} 
\\  \quad , \quad B= \begin{pmatrix}
1 &   y_1    \\ 
1  &    y_2     \\
\vdots  &   \vdots    \\
1  &     y_6
\end{pmatrix} 
\]
Then $A\star B$ is the matrix $V^{\mathrm{rec}}_{(3,2)}(z_1,\ldots, z_6)$ considered above. 
 \end{example} 
\subsubsection{A  formula for the  determinant of an amalgam}
 
 Consider the trivial Young tableau  $1_{n^m}$ of shape $n^m$ consisting of $m$ rows and $n$ columns in which the numbers $1,\ldots,mn$ are placed from top to bottom and left to right, in that order. Thus the first column consists of the numbers $1,\ldots, m$ read from top to bottom; the second $m+1, \ldots, 2m$, etc.
 The symmetric group $\Sigma_{mn}$ acts on its entries by permutation. 

Consider the tableau $\alpha= \sigma  (1_{n^m})$, where $\sigma \in \Sigma_{mn}$. Let us write
\[  A_{\sigma} =\prod_{i=1}^n \det A_{\mathrm{col}_i(\alpha)} \quad \hbox{ and } \quad  B_{\sigma} =\prod_{i=1}^n \det B_{\mathrm{row}_i(\alpha)} \]
where, for any subset $I \subset \{1,\ldots, mn\}$, $A_I$ is the square submatrix of $A$ with rows indexed by $I$, and similarly for $B$; and $\mathrm{col}_i(\alpha)$ (resp. $\mathrm{row}_i(\alpha)$) denotes the ith row (resp. column) of $\alpha$. 
 The following theorem is a special case of \cite[Theorem 3.8]{BrVdM}.

 \begin{thm}  \label{thm: detAmalgam} Let $A,B$ and $m,n$ be as above and let  
 \begin{equation} \label{eqn: Hmndef}  H_{m,n} = \prod_{i=0}^{n-1} \frac{(m+i)!}{i!} \ . 
 \end{equation} 
 Then 
 \[ \det(A\star B) = \frac{1}{H_{m,n}} \sum_{\sigma \in \Sigma_{mn}} \varepsilon(\sigma) A_{\sigma \alpha} B_{\sigma \beta} \  . \]
  \end{thm}

\newpage
\section{Fast computation of $I(h,i,j,k,l)$}

The following method leads to very fast computation of the matrices  $Q^{\sigma}$ discussed in \S\ref{sect:MainExample}.

\subsection{Contiguity matrices}
Consider the de Rham cohomology group which underlies the family of Mellin integrals $I(s_1,\ldots, s_5)$ defined by \eqref{generalM05integrals},  namely
\[  \mathcal{M}_{\dR}=H_{\dR}^2 (\mathcal{M}_{0,5}/k; (\mathcal{O}, \nabla))\]
over the field $k=\Q(s_1,\ldots, s_5)$, 
where $\mathcal{O}$ is the structure sheaf of  $\mathcal{M}_{0,5}$ and 
\[ \nabla = d+   \sum_{i=1}^5 s_i \frac{d u_i}{u_i}\]
where $u_i$ are given by \eqref{uiasxy}. By a method to be explained elsewhere, we can easily compute the space $\mathcal{M}_{\dR}$  by computing the iterated higher direct images of the algebraic vector bundle with integrable connection $(\mathcal{O}, \nabla)$ relative to the two maps  $(x,y)\mapsto y: \mathcal{M}_{0,5} \rightarrow \mathcal{M}_{0,4} $ and $\mathcal{M}_{0,4} \rightarrow  \mathcal{M}_{0,3}\cong \mathrm{Spec} \, \Q$, where $\mathcal{M}_{0,4}\cong \mathbb{P}^1 \setminus \{0,1,\infty\}$ has the affine  coordinate $y$. We verify that $\mathcal{M}_{\dR}$ has rank 2 with basis given by the cohomology classes of the forms\[  \omega_1= \frac{dx dy}{1-xy} \qquad \hbox{ and } \qquad  \omega_0= dx dy\ . \]
 For each $1\leq i \leq 5$, define the shift operator $T_i : \mathcal{M}_{\dR} \rightarrow \mathcal{M}_{\dR}$ which is  induced by  multiplication by $u_i$ on $\omega_0,\omega_1$. In this basis they are represented explicitly by $2\times 2$ matrices we denote by $M_i (s)\in M_{2\times 2}(k)$.  They are  easily computed in Maple and given explicitly below. One checks the relation:
 \begin{equation}  \label{Mconsistency}  M_i(\tau_i s) M_j(s) = M_j (\tau_j s) M_i(s) 
 \end{equation} 
 for all $1\leq i,j\leq 5$, 
 where $s=(s_1,\ldots, s_5)$ and $\tau_i$ is the operator which replaces $s_i $ with $s_i+1$. 
It follows that  for a  generic 5-tuple $\underline{s} = (s_1,\ldots, s_5)\in \R^5$ and $\underline{s}' = \tau_5 \underline{s}$ such that the integrals converge one has:
\begin{equation} \label{computeIs}   \begin{pmatrix}  I(\tau_i\underline{s}) \\
   I(\tau_i\underline{s}')   
\end{pmatrix}  =
M_i(\underline{s})    \begin{pmatrix}  I( \underline{s}) \\
  I(\underline{s}')  
\end{pmatrix}  \ . 
 \end{equation} 
 To compute $I(h,i,j,k,l)$ for integer values of $h,i,j,k,l$, we simply compose the matrices $M_i$, specialised to integer values of $s_i$,  in a suitable order, and apply them to the initial vector of periods 
\[ v   =   \left(   \int_{\sigma} \omega_1   \ , \     \int_{\sigma} \omega_0 \right)^T =  \begin{pmatrix} \zeta(2) \\    1 \end{pmatrix} \]

Note that since the matrices $M_i$ have poles in the $s_i$, some care must be taken regarding the order in which they are composed (see below), but  
\eqref{Mconsistency} ensures that the answer is independent of the order taken (provided it avoids poles).  This is very similar to the approach in \cite{CMF}, where a set of matrices $M_i$ are derived  from  hypergeometric functions and is presumably equivalent in this case.

\subsection{Explicit formulae for the matrices $M_i(s)$} Our computer code produces: 
\begin{align*}
M_1 & =\frac{1}{a_4(1+a_3)}\begin{pmatrix}
\left(1+a_{3}\right) \left(a_{5} -a_{1} \right) & \left(1+a_{3}\right) a_{2} 
\\
 \left(a_{1}+a_{4} -a_{5} \right) \left(a_{5}-1\right) & \left(1-a_{1}+a_{2}\right) a_{4}-a_{2} \left(a_{5}-1\right)
\end{pmatrix} \\
M_2 & =\frac{1}{a_4a_5} \begin{pmatrix}
\left(a_{1}-a_{2}\right) a_{4}+a_{3} \left(a_{5}-a_{1}\right) & a_{2} a_{3} 
\\
 \left(a_{1}+a_{4} -a_{5}\right) a_{5} & -a_{2} a_{5} 
\end{pmatrix}
 \\
M_3& =\frac{1}{a_1a_5} \begin{pmatrix}
\left(a_4-a_{3}\right) a_{1}+a_{2} \left(a_{5}-a_{4}\right) & a_{2} a_{3} 
\\
  \left(a_{1}+a_{4} -a_{5}\right) a_{5}& -a_{3} a_{5} 
\end{pmatrix}
\\
M_4 &= \frac{1}{a_1(1+a_2)} \begin{pmatrix}
\left(1+a_{2}\right) \left(a_{5}-a_{4}\right) & \left(1+a_{2}\right) a_{3} 
\\
 \left(a_{1}+a_{4}-a_{5}\right) \left(a_{5}-1\right) & \left(a_{3}-a_{4}+1\right) a_{1}-a_{3} \left(a_{5} -1\right) 
\end{pmatrix}\\
M_5& =\frac{1}{(1+a_2)(1+a_3)} \begin{pmatrix}
0 & \left(1+a_{3}\right) \left(1+a_{2}\right) 
\\
 \left(a_{1}+a_{4} -a_{5}\right) \left(a_{5}-1\right) & \left(a_{4}-a_{5}+1\right)( a_{2}+1)+\left(a_{1}-a_{5}+1\right) a_{3}+\left(1-a_{4}\right) a_{1} 
\end{pmatrix}
\end{align*}
where the $a_i$ are the pole vectors $p_{i+1}+1$: 
\[ {a_1 = s_3 + s_4 - s_1 + 1 \ ,  \ \ldots \  , \   a_5 = s_2 + s_3 - s_5 + 1}\ . \]
The matrices $M_i$ can be used to compute $I(n_1,n_2,n_3,n_4,n_5)$ for $n_i\geq 0$ recursively   starting from $(n_1,\ldots, n_5) = (0,0,0,0,0)$  provided   that 
there always exists an index $i$ such that $M_i(n_1,\ldots, n_5)$  is finite (work backwards by induction).  This is always possible: if the index $1\leq i \leq 5$ satisfies $s_i = \max\{s_1,\ldots, s_5\}$ then $a_{i+2},a_{i+3}>0$, where indices are mod $5$, and so $M_i$ is finite by inspection of the formulae above.

%

\begin{example} We obtain $I(0,0,1,0,1)= \zeta(2)-1$ by first applying $M_3$ to $v$, and then applying $M_5$: 
\[ M_5(0,0,1,0,0) M_3(0,0,0,0,0)\,  v =\begin{pmatrix}  0 & 1 \\  \frac{1}{2} & 0  \end{pmatrix}  \begin{pmatrix}  0 & 1 \\ 1 & -1 \end{pmatrix}  \begin{pmatrix} \zeta(2) \\ 1 \end{pmatrix}  = \begin{pmatrix}  \zeta(2)-1 \\ \frac{1}{2} \end{pmatrix} \]
via two applications of \eqref{computeIs} and the definition of $v$. 
\end{example}

\begin{rem}
The contiguity operator with respect to shifting all parameters $s_1,\ldots, s_5$ simultaneously, is similarly a $2\times 2$ matrix which is  easily  computed. If we restrict it to the `diagonal' case $s_1=s_2=\ldots = s_5=n$ we obtain the matrix with determinant  $-(n+1)^2(n+2)^{-2}$ given explicitly by
\[\begin{pmatrix} 
-3 &  5 \\
\frac{5n^2+13n +8}{(n+2)^2}  &  - \frac{8n^2+21n+13}{(n+2)^2}\
\end{pmatrix}\ .  \]
  It enables us  to recover   the Ap\'ery recurrence for the sequence $I(n,\ldots, n)$. 
\end{rem}

\renewcommand\refname{References}

\bibliographystyle{alpha}

\bibliography{biblio}

\newcommand{\etalchar}[1]{$^{#1}$}
\begin{thebibliography}{WLK{\etalchar{+}}25}

\bibitem[And83]{Andreief1883}
C.~Andr{\'e}ief.
\newblock Note sur une relation entre les int{\'e}grales d{\'e}finies des
  produits des fonctions.
\newblock {\em M{\'e}moires de la Soci{\'e}t{\'e} des Sciences Physiques et
  Naturelles de Bordeaux}, 2:1--14, 1883.

\bibitem[Ape79]{Apery}
Roger Apery.
\newblock Irrationalit\'e de {$\zeta(2)$} et {$\zeta(3)$}.
\newblock In {\em Luminy Conference on Arithmetic}, number~61 in Ast\'erisque,
  pages 11--13. 1979.

\bibitem[BB18]{HankelSolvability2016}
Andrew Bakan and Christian Berg.
\newblock Solvability of the {H}ankel determinant problem for real sequences.
\newblock In {\em Frontiers in orthogonal polynomials and {$q$}-series},
  volume~1 of {\em Contemp. Math. Appl. Monogr. Expo. Lect. Notes}, pages
  85--117. World Sci. Publ., Hackensack, NJ, 2018.

\bibitem[BBL18]{BBL2018}
T.~Bayraktar, T.~Bloom, and N.~Levenberg.
\newblock Pluripotential theory and convex bodies.
\newblock {\em Sb. Math.}, 209(3):352--384, 2018.

\bibitem[BC99]{BloomCalvi}
Thomas Bloom and Jean-Paul Calvi.
\newblock On the multivariate transfinite diameter.
\newblock {\em Ann. Polon. Math.}, 72(3):285--305, 1999.

\bibitem[Beu79]{Beukers}
F.~Beukers.
\newblock A note on the irrationality of {$\zeta (2)$}\ and {$\zeta (3)$}.
\newblock {\em Bull. London Math. Soc.}, 11(3):268--272, 1979.

\bibitem[Bos89]{Bos}
Len Bos.
\newblock A characteristic of points in r2 having lebesgue function of
  polynomial growth.
\newblock {\em Journal of Approximation Theory}, 56(3):316--329, 1989.

\bibitem[Bro09]{BrENS}
Francis Brown.
\newblock {Multiple zeta values and periods of moduli spaces ${\cal M}_{0 ,n}(
  \mathbb R )$}.
\newblock {\em Annales Sci. Ecole Norm. Sup.}, 42:371, 2009.

\bibitem[Bro16]{DinnerParties}
Francis Brown.
\newblock Irrationality proofs for zeta values, moduli spaces and dinner
  parties.
\newblock {\em Mosc. J. Comb. Number Theory}, 6(2-3):102--165, 2016.

\bibitem[Bro17]{NotesMot}
F.~Brown.
\newblock Notes on motivic periods.
\newblock {\em Commun. Number Theory Phys.}, 11(3):557--655, 2017.

\bibitem[Bro25]{BrVdM}
Francis Brown.
\newblock Multivariable vandermonde determinants, amalgams of matrices and
  specht modules.
\newblock {\em Journal of Algebra}, 678:253--278, 2025.

\bibitem[CM17]{CoxMau}
David~A. Cox and Sione Ma`u.
\newblock Transfinite diameter on complex algebraic varieties.
\newblock {\em Pacific J. Math.}, 291(2):279--317, 2017.

\bibitem[DR07]{DeMarcoRumely}
Laura DeMarco and Robert Rumely.
\newblock Transfinite diameter and the resultant.
\newblock {\em J. Reine Angew. Math.}, 611:145--161, 2007.

\bibitem[Jed91]{LimitsTransfinite}
Mieczyslaw Jedrzejowski.
\newblock The homogeneous transfinite diameter of a compact subset of
  {${\mathbb{C}}^N$}.
\newblock In {\em Proceedings of the {T}enth {C}onference on {A}nalytic
  {F}unctions ({S}zczyrk, 1990)}, volume~55, pages 191--205, 1991.

\bibitem[LS91]{LoeserSabbah}
F.~Loeser and C.~Sabbah.
\newblock \'{E}quations aux diff\'{e}rences finies et d\'{e}terminants
  d'int\'{e}grales de fonctions multiformes.
\newblock {\em Commentarii Mathematici Helvetici}, 66(1):458--503, 1991.

\bibitem[LW21]{LevenbergWielonsky}
Norman Levenberg and Franck Wielonsky.
\newblock On transfinite diameters in {$\Bbb C^d$} for generalized notions of
  degree.
\newblock {\em Math. Scand.}, 127(2):337--360, 2021.

\bibitem[Ma`11]{Mau}
S.~Ma`u.
\newblock Chebyshev constants and transfinite diameter on algebraic curves in
  {$\Bbb C^2$}.
\newblock {\em Indiana Univ. Math. J.}, 60(5):1767--1796, 2011.

\bibitem[Pow82]{PowerHankel}
S.C. Power.
\newblock Finite rank multivariable hankel forms.
\newblock {\em Linear Algebra and its Applications}, 48:237--244, 1982.

\bibitem[RV96]{RVzeta2}
Georges Rhin and Carlo Viola.
\newblock On a permutation group related to {$\zeta(2)$}.
\newblock {\em Acta Arith.}, 77(1):23--56, 1996.

\bibitem[RV01]{RVzeta3}
Georges Rhin and Carlo Viola.
\newblock The group structure for {$\zeta(3)$}.
\newblock {\em Acta Arith.}, 97(3):269--293, 2001.

\bibitem[SS62]{SchifferSiciak}
M.~Schiffer and J.~Siciak.
\newblock Transfinite diameter and analytic continuation of functions of two
  complex variables.
\newblock In {\em Studies in mathematical analysis and related topics}, pages
  341--358. Stanford University Press, 1962.

\bibitem[TB12]{BloomBosLevenberg}
N.~Levenberg T.~Bloom, L.~Bos.
\newblock The transfinite diameter of the real ball and simplex.
\newblock {\em Annales Polonici Mathematici}, 106(1):83--96, 2012.

\bibitem[WLK{\etalchar{+}}25]{CMF}
Shachar Weinbaum, Elyasheev Leibtag, Rotem Kalisch, Michael Shalyt, and Ido
  Kaminer.
\newblock On conservative matrix fields: Continuous asymptotics and arithmetic,
  2025.

\bibitem[Zah75]{Zaharjuta}
V.~P. Zaharjuta.
\newblock Transfinite diameter, \v{C}eby\v{s}ev constants and capacity for a
  compactum in {$C\sp{n}$}.
\newblock {\em Mat. Sb. (N.S.)}, pages 374--389, 503, 1975.

\bibitem[Zud17]{ZudilinDet}
Wadim Zudilin.
\newblock A determinantal approach to irrationality.
\newblock {\em Constr. Approx.}, 45(2):301--310, 2017.

\end{thebibliography}

\end{document}